\documentclass[a4paper,11pt,twoside]{article}

\usepackage{hyperref}
\hypersetup{colorlinks,citecolor=blue,urlcolor=blue}

\usepackage[latin1]{inputenc}
\usepackage[T1]{fontenc}
\usepackage[english]{babel}
\usepackage{charter}

\usepackage{indentfirst}
\usepackage{xcolor}
\usepackage[top=3.cm, bottom=4.0cm, left=2.95cm, right=2.95cm]{geometry}
\usepackage{tikz}
\usetikzlibrary{matrix}

\usepackage{graphicx}
\usepackage{subfigure}

\usepackage{amsmath}
\usepackage{amsthm}	
\usepackage{amsfonts}	
\usepackage{amssymb	
	,bbm
	,xfrac
	}

\theoremstyle{plain}
\newtheorem{theorem}{Theorem}[section]
\newtheorem{lemma}[theorem]{Lemma}
\newtheorem{proposition}[theorem]{Proposition}
\newtheorem{corollary}[theorem]{Corollary}
\newtheorem{definition}[theorem]{Definition}

\newtheorem{assumption}[theorem]{Assumption}

\theoremstyle{definition}
\newtheorem{remark}[theorem]{Remark}

\numberwithin{equation}{section}
\numberwithin{figure}{section}
 \usepackage[nodayofweek]{datetime}

\usepackage[nobysame
           ,alphabetic
           ,initials
           ]{amsrefs}

\usepackage[nodayofweek]{datetime}

\newcommand{\N}{\mathbb{N}}%naturalnumbers
\newcommand{\R}{\mathbb{R}}	%real numbers
\newcommand{\F}{\mathcal{F}}	%for sigma-fields
\newcommand{\abs}[1]{\lvert#1\rvert }	%for absolute value
\newcommand{\Abs}[1]{\left\lvert#1\right\rvert }	%for absolute value
\newcommand{\half}{\frac{1}{2}}%for 1/2
\renewcommand{\matrix}[1]{\begin{pmatrix}#1\end{pmatrix}}
\newcommand{\scalar}[2]{\left\langle #1 , #2 \right\rangle}%produit scalaire
\newcommand{\goesto}{\to}%for convergence

\newcommand{\bE}{\mathbb{E}}

\newcommand{\bP}{\mathbb{P}}

\newcommand{\bR}{\mathbb{R}}

\newcommand{\cE}{\mathcal{E}}
\newcommand{\cF}{\mathcal{F}}

\newcommand{\cH}{\mathcal{H}}
\newcommand{\cI}{\mathcal{I}}

\newcommand{\cS}{\mathcal{S}}

\newcommand{\cY}{\mathcal{Y}}
\newcommand{\cZ}{\mathcal{Z}}

\newcommand{\ti}{{t_{i}}}
\newcommand{\tip}{{t_{i+1}}}
\newcommand{\tj}{{t_{j}}}

\newcommand{\ip}{{i+1}}
\newcommand{\jp}{{j+1}}

\newcommand{\ud}{\mathrm{d}}
\newcommand{\uds}{\mathrm{d}s}
\newcommand{\udt}{\mathrm{d}t}
\newcommand{\udu}{\mathrm{d}u}
\newcommand{\udr}{\mathrm{d}r}

\newcommand{\udwr}{\mathrm{d}W_r}

\newcommand{\1}{\mathbbm{1}}

\newcommand{\hfgrowth}{\textbf{(Growth)}}
\newcommand{\hfregY}{\textbf{(RegY)}{ }}
\newcommand{\hfreg}{\textbf{(Reg)}{ }}
\newcommand{\hfmon}{\textbf{(Mon)}{ }}
\newcommand{\hfmongrowth}{\textbf{(MonGr)}{ }}

\newcommand{\htgrowth}{\textbf{(TGrowth)} }
\newcommand{\htregY}{\textbf{(TRegY)} }
\newcommand{\htreg}{\textbf{(TReg)} }
\newcommand{\htmon}{\textbf{(TMon)} }
\newcommand{\htmongrowth}{\textbf{(TMonGr)} }

\newcommand{\htfcvg}{\textbf{(TCvg)} }
\newcommand{\hH}{\textbf{(AH)}}
\newcommand{\hxiN}{(\textbf{A}$\mathbf{\xi^N}$)}

\newcommand{\RregY}{\mathcal{R}^{\text{(regY)}}}
\newcommand{\Rmon}{\mathcal{R}^{\text{(mon)}}}

\newcommand{\Rbla}{\mathcal{R}}
\newcommand{\Ai}{A_i}%{\fbox{$\textcolor{blue}{A_i}$}}
\newcommand{\Bi}{B_i}%{\fbox{$\textcolor{blue}{B_i}$}}
\newcommand{\Biz}{B^0_i}%{\fbox{$\textcolor{blue}{B^0_i}$}}

\title{Convergence and qualitative properties of modified explicit schemes for BSDEs with polynomial growth}
\author{
\normalsize Arnaud Lionnet 
		\footnote{A part of A. Lionnet's research on this paper was funded by a 150th Anniversary Postdoctoral Mobility Grant from the London Mathematical Society. 
				The hospitality of the University of Edinburgh is gratefully acknowledged.}  \\[8pt] 
		\small  University of Edinburgh\\ 
        \small  (Visitor) \\
        \small  and\\  
        \small INRIA Paris - ENPC \\
        \small 2 rue Simone Iff\\
        \small 75012 Paris, France\\
        \small  arnaud.lionnet@inria.fr 
 \and
 	\normalsize Gon\c calo Dos Reis\footnote{G. dos Reis acknowledges support from the \emph{Funda{\c c}$\tilde{\text{a}}$o para a Ci$\hat{e}$ncia e a Tecnologia} (Portuguese Foundation for Science and Technology) through the project UID/MAT/00297/2013 (Centro de Matem\'atica e Aplica\c c$\tilde{\text{o}}$es CMA/FCT/UNL).} \\[8pt]
         \small  University of Edinburgh\\ 
         \small  School of Mathematics \\
         \small  Edinburgh, EH9 3JZ, UK\\  
         \small  and \\
	\small  Centro de Matem\'atica e Aplica\c c$\tilde{\text{o}}$es \\
	\small (CMA), FCT, UNL, Portugal \\
        \small  G.dosReis@ed.ac.uk
 \and
         \normalsize Lukasz Szpruch \\[8pt]
         \small  University of Edinburgh\\ 
         \small  School of Mathematics \\
         \small  Edinburgh, EH9 3JZ, UK\\  
         \\
		\\
         \\
         \small  l.szpruch@ed.ac.uk
}

\date{ \ddmmyyyydate\today}

\begin{document}
\selectlanguage{english}
\maketitle
%%%%%%%%%%%%%%%%%%%%%%%%%%%%%%%%%%%%%%%%%%%%%%%%%%%%%%%%%%%%%%%%%%%%%%%%%%%%%%%%%%%%%%%%%%%%%%%%%%%%%%%%%%%%%%%%%%%%%%%%
%%%%%%%%%%%%%%%%%%%%%%%%%%%%%%%%%%%%%%%%%%%%%%%%%%%%%%%%%%%%%%%%%%%%%%%%%%%%%%%%%%%%%%%%%%%%%%%%%%%%%%%%%%%%%%%%%%%%%%%%
%%%%%%%%%%%%%%%%%%%%%%%%%%%%%%%%%%%%%%%%%%%%%%%%%%%%%%%%%%%%%%%%%%%%%%%%%%%%%%%%%%%%%%%%%%%%%%%%%%%%%%%%%%%%%%%%%%%%%%%%
%%%%%%%%%%%%%%%%%%%%%%%%%%%%%%%%%%%%%%%%%%%%%%%%%%%%%%%%%%%%%%%%%%%%%%%%%%%%%%%%%%%%%%%%%%%%%%%%%%%%%%%%%%%%%%%%%%%%%%%%
%%%%%%%%%%%%%%%%%%%%%%%%%%%%%%%%%%%%%%%%%%%%%%%%%%%%%%%%%%%%%%%%%%%%%%%%%%%%%%%%%%%%%%%%%%%%%%%%%%%%%%%%%%%%%%%%%%%%%%%%

%%%%%%%%%%%%%%%%%%%%%%%%%%%%%%%%%
\begin{abstract}
The theory of Forward-Backward Stochastic Differential Equations (FBSDEs) paves a way to probabilistic numerical methods for nonlinear parabolic PDEs. The majority of the results on the numerical methods for FBSDEs relies on the global Lipschitz assumption, which is not satisfied for a number of important cases such as the Fisher--KPP or the FitzHugh--Nagumo equations. 
Furthermore, it has been shown in \cite{LionnetReisSzpruch2015} that for BSDEs with monotone drivers having polynomial growth in the primary variable $y$, only the (sufficiently) implicit schemes converge. But these require an additional computational effort compared to explicit schemes.

This article develops a general framework that allows the analysis, in a systematic fashion, of the integrability properties, convergence and qualitative properties (e.g.~comparison theorem) for whole families of modified explicit schemes. The framework yields the convergence of some modified explicit scheme with the same rate as implicit schemes and with the computational cost of the standard explicit scheme. 

To illustrate our theory, we present several classes of easily implementable modified explicit schemes that can computationally outperform the implicit one and preserve the qualitative properties of the solution to the BSDE. These classes fit into our developed framework and are tested in computational experiments.  
\end{abstract}
%%%%%%%%%%%%%%%%%%%%%%%%%%%%%%%%%

%%%%%%%%%%%%%%%%%%%%%%%%%%%%%%%%%
; secondary: 
%60H07%Stochastic calculus of variations and the Malliavin calculus
60H30%Applications of stochastic analysis (to PDE, etc.)
.\\
%%%%%%%%%%%%%%%%%%%%%%%%%%%%%%%%%

%%%%%%%%%%%%%%%%%%%%%%%%%%%%%%%%%

%%%%%%%%%%%%%%%%%%%%%%%%%%%%%%%%%%%%%%%%%%%%%%%%%%%%%%%%%%%%%%%%%%%%%%%%%%%%%%%%%%%%%%%%%%%%%%%%%%%%%%%%%%%%%%%%%%%%%%%%%%%%%%%%%%%%%%%%%%%%%%%%%
\newpage
\tableofcontents
%\newpage
%%%%%%%%%%%%%%%%%%%%%%%%%%%%%%%%%%%%%%%%%%%%%%%%%%%%%%%%%%%%%%%%%%%%%%%%%%%%%%%%%%%%%%%%%%%%%%%%%%%%%%%%%%%%%%%%%%%%%%%%%%%%%%%%%%%%%%%%%%%%%%%%%

%%%%%%%%%%%%%%%%%%%%%%%%%%%%%%%%%%%%%%%%%%%%%%%%%%%%%%%%%%%%%%%%%%%%%%%%%%%%%%%%%%%%%%%%%%%%%%%%%%%%%%%%%%%%%%%%%%%%%%%%%%%%%%%%%%%%%%%%%%%%%%%%%
%%%%%%%%%%%%%%%%%%%%%%%%%%%%%%%%%%%%%%%%%%%%%%%%%%%%%%%%%%%%%%%%%%%%%%%%%%%%%%%%%%%%%%%%%%%%%%%%%%%%%%%%%%%%%%%%%%%%%%%%%%%%%%%%%%%%%%%%%%%%%%%%%
%%%%%%%%%%%%%%%%%%%%%%%%%%%%%%%%%%%%%%%%%%%%%%%%%%%%%%%%%%%%%%%%%%%%%%%%%%%%%%%%%%%%%%%%%%%%%%%%%%%%%%%%%%%%%%%%%%%%%%%%%%%%%%%%%%%%%%%%%%%%%%%%%
%%%%%%%%%%%%%%%%%%%%%%%%%%%%%%%%%%%%%%%%%%%%%%%%%%%%%%%%%%%%%%%%%%%%%%%%%%%%%%%%%%%%%%%%%%%%%%%%%%%%%%%%%%%%%%%%%%%%%%%%%%%%%%%%%%%%%%%%%%%%%%%%%
%%%%%%%%%%%%%%%%%%%%%%%%%%%%%%%%%%%%%%%%%%%%%%%%%%%%%%%%%%%%%%%%%%%%%%%%%%%%%%%%%%%%%%%%%%%%%%%%%%%%%%%%%%%%%%%%%%%%%%%%%%%%%%%%%%%%%%%%%%%%%%%%%
\section{Introduction}

%**************************************************************************************************
Since the initial papers of Zhang \cite{Zhang2004} and Bouchard and Touzi \cite{BouchardTouzi2004}, an important literature has developped concerning the methods for approximating numerically the solution to a nonlinear backward stochastic differential equation (BSDE thereafter).
The importance of BSDEs with nonlinear drivers is due to their frequent use in stochastic control problems, 
and to their deep connection with parabolic partial differential equations (PDEs), 
which are used to describe many biological and physical phenomena as well as the solution to many decision problems. 
Indeed, the solution $v$ to the PDE 
	\begin{align*}
		\Big( \partial_t v + \half (\sigma \sigma^*) \cdot v_{xx} + b \cdot v_x + f(\cdot,v,v_x \sigma) \Big)(t,x) = 0
		\quad \text{with} \quad 
		v(T,x)=g(x),
	\end{align*}
is given by solving, for $(s,x) \in [0,T] \times \R^d$, for $t \in [t,T]$, the forward-backward SDE
	\begin{align}
		\label{canonicSDE}
		X^{s,x}_t &= x + \int_s^t b(r,X^{s,x}_r)\udr + \int_s^t \sigma(r,X^{s,x}_r)\udwr,\\
		\label{canonicBSDE}
		Y^{s,x}_t &= g(X^{s,x}_T) + \int_t^T f(r,Y^{s,x}_r,Z^{s,x}_r) \udr - \int_t^T Z^{t,x}_r \udwr,
	\end{align}
and setting $v(s,x)=Y^{s,x}_s$ (see e.g. \cite{ElKarouiPengQuenez1997}).
In solving \eqref{canonicBSDE}, one seeks a pair of processes $\cS=(Y,Z)$, adapted to the filtration $\F$ of the Brownian motion $W$. The data of the BSDE are the $\F_T$-measurable random variable $\xi = g(X^{s,x}_T)$, called terminal condition, and the function $f$, generally referred to as driver and which will depend only on time, $Y$ and $Z$ for the simplicity of exposition.

Consider a discretization of the time interval $[0,T]$ by a regular subdivision $\pi^N$ : $0=t_0 < t_1 < \ldots < t_N = T$, where $\ti = ih$ for all $i \in \{0,\ldots,N\}$ and $h=T/N$. 
To construct a numerical methods for BSDEs, one typically begins by discretizing the time dynamics for $Y$ over $[\ti,\tip]$ as
	\begin{align}		\label{equation--introduction--heuristic.BTZ.scheme.for.Y}
		Y^N_i = \bE_i\Big[ Y^N_\ip + (1-\theta) f(\ti,Y^N_\ip,Z^N_i) h \Big] + \theta f(\ti, Y^N_i,Z^N_i) h ,
	\end{align}
where $\bE_i$ is $\bE\big[ \cdot | \F_\ti \big]$ and the approximation $Z^N_i$ is suitably computed.  When the parameter $\theta=0$ this is the explicit scheme while $\theta=1$ is the implicit scheme\footnotemark, both dubbed BTZ schemes in \cite{Chassagneux2012}. 
\footnotetext{Numerical schemes most often compute first $Z^N_i$ explicitly from the input $Y^N_\ip$, and then use this to compute $Y^N_i$. 
In this paper, all mentions of implicit and explicit scheme in the context of BSDEs refer the to $Y$-component.}
Such a time-discretization scheme is first initialized by setting $Y^N_N$ to be a numerical approximation $\xi^N$ of the terminal condition and then applied in a backward recursive fashion, producing a family $(Y^N_i,Z^N_i)_{i=0, \ldots, N}$ approximating the solution $(Y_t,Z_t)_{t \in [0,T]}$ of the BSDE.

%**************************************************************************************************
\paragraph*{}

There has been a significant progress in the analysis of variants of scheme \eqref{equation--introduction--heuristic.BTZ.scheme.for.Y} or the ways to approximation the conditional expectations $\bE_i$, although the vast majority of works impose a restrictive Lipschitz condition on driver $f$ in both its $Y$ and $Z$ variables (see \cite{BenderDenk2007}, \cite{CrisanManolarakis2014}, \cite{BriandLabart2014}, \cite{GobetTurkedjiev2016}, and references therein). 

However, in many cases of interest, the driver is not Lipschitz but instead has superlinear growth in one of its variables. 
For instance, for PDEs of reaction-diffusion type such as the Allen--Cahn equation, the FitzHugh--Nagumo equations, the Fisher--KPP equation or the standard nonlinear heat and Schr\"odinger equation (see \cite{Henry1981}, \cite{Rothe1984}, \cite{EstepLarsonWilliams2000}, \cite{Kovacs2011} and references), the function $f$ is a polynomial in $Y$. 
Meanwhile, in stochastic control problems, the driver typically has quadratic growth in the $Z$ variable. 

Chassagneux and Richou obtained in \cite{ChassagneuxRichou2016} the convergence of an implicit scheme in the case where the terminal condition $\xi$ is bounded and the driver $f$ has quadratic growth in $Z$.
In \cite{LionnetReisSzpruch2015} we studied the case where the terminal condition has all moments, the driver has polynomial growth in $Y$ and satisfies a so-called monotonicity condition (also known as one-sided Lipschitz condition). 
The monotonicity condition is a structure property which states, in the scalar case, that for all $y$, $y'$ in the domain $\R$ and for all $z$,
	\begin{align}		\label{equation--introduction--monotonicity.in.1d}
		\big( f(y,z)-f(y',z) \big) \; (y'-y) \le M_y \abs{y'-y}^2 ,
	\end{align}
where $M_y \in \R$. For instance, the driver $f(y,z)=y-y^3$ typical of the FitzHugh--Nagumo equation is one-sided Lipschitz over the domain $\R$ with $M_y=1$.

As explained in \cite{LionnetReisSzpruch2015}, the explicit scheme described in \eqref{equation--introduction--heuristic.BTZ.scheme.for.Y} can explode.  
This is due to the superlinear growth of the driver $f$ and the unboundedness of the terminal condition $\xi$. 
As a remedy, \cite{LionnetReisSzpruch2015} proposed to use the implicit scheme, which was shown to converge, or an explicit scheme with a truncated numerical terminal condition $T^N\big(\xi^N\big)$, where the truncation function $T^N$ fades to the identity when the number $N$ of time-steps goes to $+\infty$. 
However, the implicit scheme requires an extra computational effort to solve the nonlinear equation \eqref{equation--introduction--heuristic.BTZ.scheme.for.Y} defining $Y^N_i$ when $\theta=1$. 
And for the explicit scheme with the truncated terminal conditions, severe restrictions on the size of the time-step had to be imposed, and the tuning/performance of the algorithm depends on knowledge of $f$, in particular on its growth. %HEREHERE 
The purpose of this paper is to obtain converging explicit schemes by working instead on the dynamics of the scheme itself, replacing the driver $f$ by a modified driver $f^h$, with no time-step restriction. In addition, we can obtain some numerical schemes of black-box type, where no a priori knowledge on the structure of the driver is required.

%**************************************************************************************************
\paragraph*{} 

Explosion problems of naïve explicit schemes were already stressed in \cite{HutzenthalerJentzenKloeden2011} in the context of the numerical methods for stochastic differential equations (SDEs thereafter). 
A significant body of works considered various modifications to the explicit schemes for SDEs, dubbed ``tamed schemes'', to recover integrability and convergence in the non-Lipschitz setting, 
see \cites{HutzenthalerJentzenKloeden2011,HutzenthalerJentzenKloeden2012,hutzenthaler2013exponential,hutzenthaler2014perturbation,hutzenthaler2015numerical,
ChassagneuxJacquierMihaylov2014,szpruch}. 
In the context of BSDEs, very particular instances of modified drivers were already used by \cite{ChassagneuxRichou2016} to deal with the quadratic growth in $Z$, and in \cite{LionnetReisSzpruch2015} to deal with the dependence of the driver of  \eqref{canonicBSDE} in the solution $X$ to the SDE \eqref{canonicSDE}. 
But this was only used as an ad hoc tool to handle a particular issue.

This motivates us to study systematically the family of modified explicit schemes where the BSDE driver $f$ is replaced by a ``tamed'' driver $f^h$. 
%Some of these modifications $f^h$ could be
%	\begin{align*}
%		f^h(y,z) = \frac{f(y,z)}{1+\abs{y}^m h^\alpha}  , 
%		\quad 	
%		f^h(y,z) = T_o^h \big( f(y,z) \big) 
%		\quad \text{or} \quad  
%		f^h(y,z) = f\big( T_i^h(y),z \big) ,
%	\end{align*}
%where $m$ is the degree of the polynomial growth of $f$, $\alpha = \half$, $T^N_o$ is the projection on the ball of radius $R(h)=R_0 h^{-\half}$ and $T^N_i$ is the projection on the ball of radius $r(h)=h^{-\frac{1}{2(m-1)}}$. However, many other modifications are possible.
Provided the modified driver $f^h$ is appropriately ``tamed'', the explosion of the scheme is prevented. 
One then expects that such a modified scheme will converge to the continuous-time solution provided $f^h \rightarrow f$. In addition, if this convergence can happen fast enough, the usual convergence rate of the implicit schemes can be recovered.
Our approach in this work is to identify the essential properties of these modified drivers $f^h$ which guarantee the convergence of the corresponding modified explicit scheme for BSDEs. 
As a consequence, we show at once the convergence of a whole range of modified explicit BSDE schemes. %such as those that have been proposed in the context of SDEs.

%**************************************************************************************************
%\paragraph*{}
%The heuristics for this type of modified explicit schemes is the following. For a fixed number $N$ of time-steps, and associated time grid $\pi^N$, since the superlinear driver $f$ is replaced by the more regular (Lipschitz) and linearly growing driver $f^N$, the discrete-time dynamics as $i$ goes from $N$ to $0$, becomes more stable -- in particular, the scheme is almost $L^2$-stable over one time-step. This allows to obtain clean error bounds for a finite-mesh time grid. 
%%Now, as $N \goesto +\infty$, on the one hand, the discretization error vanishes, on the other hand, $f^N \goesto f$ and so the output $(Y^N_i,Z^N_i)_{i=0, \ldots, N}$ of the scheme indeed approximates the solution to the BSDE with driver $f$. 
%Of course, as $N \goesto +\infty$, the regularity and growth constants of $f^N$ explode. Nonetheless, a careful analysis shows that if one can exploit the one-sided Lipschitz property in a suitable way, then such modifications are sufficient to obtain required results. 

%The key idea is that because superlinear coefficients are still locally Lipschitz, explicit schemes have a certain ``radius of stability'' that usually is expanding as the size of the time-steps is going to $0$.  If one can modify suitably the scheme, in particular outside the ``stability region'', then integrability and convergence could be restored. 

%**************************************************************************************************
\paragraph*{}
In a certain sense, this question is reminiscent of the stability with respect to the driver for solutions of continuous-time BSDEs. 
Given the fixed set of times $[0,T]$, let us denote by $\cS^f$ the solution to the BSDE with driver $f$ and by $\cS^{f^\epsilon}$ the solution for the driver $f^\epsilon$, 
such that $f^\epsilon$ converges to $f$ in some sense (uniformly on compact, typically). The stability theorem, valid for monotone drivers as well as for Lipschitz drivers, states that $\cS^{f^\epsilon}$ converges to $\cS^f$, in the appropriate norm, and gives an upper bound on the distance.
Here, for the set of times $\pi^N$ (which ``converges to $[0,T]$'' as $N \goesto +\infty$), let us denote by $S^{N,f^h} = (Y^{N,f^h}_i,Z^{N,f^h}_i)_{i=0 \ldots N}$ the output of the modified explicit scheme with drivers $f^h$, and by $S^{N,f}$ the output of the standard (BTZ) explicit scheme, with driver $f$. 
One may then be tempted to say that, as $f^h$ converges to $f$ when $h=T/N \goesto 0$, $S^{N,f^h}$ should be close to $S^{N,f}$ for large $N$. 
However, the BTZ explicit scheme $S^{N,f}$, in general, does not converge to $\cS^{f}$ when $f$ is not Lipschitz, while in this paper it is proved that $S^{N,f^h}$ converges to $\cS^{f}$. Therefore, convergence of $S^{N,f^h}$ to $S^{N,f}$ cannot hold (at least not uniformly over $\pi^N$ : for fixed $\pi^N$ and $f^\epsilon \goesto f$, one should have $S^{N,f^\epsilon} \goesto S^{N,f}$). 
We depict the situation on Figure \ref{fig:1} where by $(!)$ we marked the convergence results that, in general, do not hold in the non-Lipschitz setting.

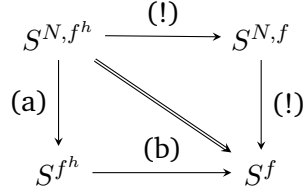
\begin{figure} \label{fig:1}
\begin{center}
	\begin{tikzpicture}
  \matrix (m) [matrix of math nodes,row sep=3em,column sep=4em,minimum width=2em]
  {
    S^{N,f^h} & S^{N,f} \\
    S^{f^h} & S^{f} \\};
  \path[-stealth]
    (m-1-1) edge node [left] {(a)} (m-2-1)
            edge  node [above] {(!)} (m-1-2)
            edge [double]  (m-2-2)
    (m-2-1.east|-m-2-2) edge node [below] {}
            node [above] {(b)} (m-2-2)
    (m-1-2) edge node [right] {(!)} (m-2-2);
          % edge [dashed,-] (m-2-1);
\end{tikzpicture}
\end{center}
\caption{Schematic diagram of key findings of this paper. Double arrow marks the conclusion of Theorem \ref{theorem---main--convergence.of.tamed.scheme}. $(!)$ marks convergence that in general does not hold. }
\end{figure}

One could nonetheless say that the BSDE solution $\cS^f$ with driver $f$ and the BSDE solution $\cS^{f^h}$ with driver $f^h$ must be close by the stability result for continuous time BSDEs. This depicted on the Figure \ref{fig:1} by the arrow (b).
Then, if $f^h$ is Lipschitz for each $N$ (albeit with Lipschitz and growth constants which explode when $h\rightarrow 0$), the standard estimate for the convergence of the explicit BTZ scheme with Lipschitz drivers should allow to bound the distance between $\cS^{f^h}$ and $S^{N,f^h}$. 
This depicted on the Figure \ref{fig:1} by the arrow (a). 
However, this strategy will not work since the upper bound in that estimate involves (the exponential of) a constant that grows with the regularity and growth constants of $f^h$.
%
%However, firstly, that estimate (just like our estimate in Theorem \ref{theorem---main--convergence.of.tamed.scheme}) is typically provided with ``some constant C'',
%whose precise dependence on the growth and regularity constants of the driver is not explicit. 
%That estimate thus allows to conclude to the convergence $S^{N,f} \goesto \cS^f$, with a rate,  but for a fixed driver $f$. 
%Here, as we replace $f$ with polynomial growth by drivers $f^h$ with linear growth, the growth and regularity constants of $f^h$ explode as $N \goesto +\infty$. 
One must therefore carefully analyse the modified explicit schemes and work with modifications $f^h$ that, on one hand, allow to tame the superlinear growth and, on the other hand, preserve certain structural properties of $f$ such as  the monotonicity condition.  
In addition, we manage to show our convergence $S^{N,f^h} \goesto \cS^{f}$ for modified drivers $f^h$ that are not necessarily Lipschitz, but only almost-Lipschitz and of linear growth (cf the assumptions \htregY and \htgrowth below, when $\RregY \neq 0$).

%**************************************************************************************************
\paragraph*{} 

Taming the driver does not merely allow to recover the convergence of the explicit scheme for BSDEs in the regime $h \goesto 0$. It also makes it more robust and qualitatively more satisfying for finite $h$. 
An important qualitative property is the comparison theorem, which plays a major role in the study of continuous-time BSDEs. In the context of numerical schemes for BSDEs with driver that is quadratic in $Z$, this property was obtained in \cite{ChassagneuxRichou2016} and used to prove the convergence of their scheme. 
In this paper, we show in the case of drivers with superlinear growth in $Y$, that a comparison theorem holds for suitably modified explicit schemes. This qualitative property is of interest for its own sake, but also allows to deduce that the scheme remains in some domain $D$ of the space when the continuous-time solution does. This is important if the monotonicity condition for $f$ is only satisfied on $D$, as is the case for instance for the Fisher--KPP equation where the driver is $f(y,z)=y-y^2$. Here the monotonicity condition for $f$ is satisfied on the domain $D=[0,+\infty[$, and the solution stays positive.
To the best of our knowledge, this is the first result on a comparison theorem for explicit schemes for BSDEs, even for a driver $f$ that is Lipschitz in both its $Y$ and $Z$ variables.

%**************************************************************************************************
\paragraph*{}
The paper is organized as follow. 
In Section \ref{section---set-up}, we describe the precise assumptions under which we work, the numerical schemes considered and our main results. 
Sections \ref{section---size.of.scheme} and \ref{section---convergence.of.scheme} are concerned with proving the convergence of the scheme. 
Specifically, in Section \ref{section---size.of.scheme}, we show how taming the driver prevents the scheme from exploding, deriving almost-sure local and then global bounds, which lead to moment estimates. 
In Section \ref{section---convergence.of.scheme}, we then prove that the scheme converges. For this we first prove that the scheme is almost $L^2$-stable over one time-step, which allows to control the propagation of errors. We then analyze the discretization error created at each time-step, and show that their sum converges with rate at least $1/2$. 
In Section \ref{section---qualitative.properties}, we study the discrete comparison and the preservation of positivity.
Finally, Section \ref{section---numerical.simulations} illustrates with some examples of tamed drivers the results we obtained. 
The paper closes with two appendices. Appendix \ref{appendix:SideResults} provides proofs regarding our main results, while Appendix \ref{appendix:VerifyingExamples} provides proofs for the examples studied in Section \ref{section---numerical.simulations}.

%%%%%%%%%%%%%%%%%%%%%%%%%%%%%%%%%%%%%%%%%%%%%%%%%%%%%%%%%%%%%%%%%%%%%%%%%%%%%%%%%%%%%%%%%%%%%%%%%%%%%%%%%%%%%%%%%%%%%%%%%%%%%%%%%%%%%%%%%%%%%%%%%
%%%%%%%%%%%%%%%%%%%%%%%%%%%%%%%%%%%%%%%%%%%%%%%%%%%%%%%%%%%%%%%%%%%%%%%%%%%%%%%%%%%%%%%%%%%%%%%%%%%%%%%%%%%%%%%%%%%%%%%%%%%%%%%%%%%%%%%%%%%%%%%%%
%%%%%%%%%%%%%%%%%%%%%%%%%%%%%%%%%%%%%%%%%%%%%%%%%%%%%%%%%%%%%%%%%%%%%%%%%%%%%%%%%%%%%%%%%%%%%%%%%%%%%%%%%%%%%%%%%%%%%%%%%%%%%%%%%%%%%%%%%%%%%%%%%
%%%%%%%%%%%%%%%%%%%%%%%%%%%%%%%%%%%%%%%%%%%%%%%%%%%%%%%%%%%%%%%%%%%%%%%%%%%%%%%%%%%%%%%%%%%%%%%%%%%%%%%%%%%%%%%%%%%%%%%%%%%%%%%%%%%%%%%%%%%%%%%%%
%%%%%%%%%%%%%%%%%%%%%%%%%%%%%%%%%%%%%%%%%%%%%%%%%%%%%%%%%%%%%%%%%%%%%%%%%%%%%%%%%%%%%%%%%%%%%%%%%%%%%%%%%%%%%%%%%%%%%%%%%%%%%%%%%%%%%%%%%%%%%%%%%
\section{Assumptions, schemes and main results}
\label{section---set-up}

%_________________________________________________________________________________________________________________________________________________________________________________________________
\paragraph{Notation.}

Fix $T>0$. We work on a canonical Wiener space $(\Omega,
\cF, \bP)$ carrying a $d$-dimensional Wiener process $W = (W^1,\cdots, W^d)$
restricted to the time interval $[0,T]$. We denote by 
$\cF=(\cF_t)_{ t\in[0,T]}$ its natural filtration enlarged in the usual way
by the $\bP$-zero sets and by $\bE$ and $\bE[\cdot|\cF_t]=\bE_t[\cdot]$ the
usual expectation and conditional expectation operator respectively.

We denote by $\scalar{\cdot}{\cdot}$ and $\abs{\cdot}$ the canonical inner product and Euclidean norm on $\R^d$, and by $x^*$ the transposed of $x \in \R^d$ when seen as a $\R^{d \times 1}$ matrix (column-vector).
$I_d$ denotes the $d$-dimensional identity matrix. 
$L^p = L^{p}(\cF_T,\bR^d)$ is the space of $\R^d$-valued $\cF_T$-measurable random variables $X$ with norm  $\|X\|_{L^p} = \bE[\, |X|^p]^{1/p} < \infty$. 
$\cS^{p}$ is the space of $d$-dimensional $\cF$-adapted processes $Y$ satisfying $\|Y \|_{\cS^p} = \bE[\sup_{t\in[0,T]}|Y_t|^p]^{1/p}<\infty$.
$\cH^{p}$ is the space of $d$-dimensional $\cF$-adapted processes $Z$ satisfying $\|Z\|_{\cH^p}=\bE\big[\big(\int_0^T|Z_s|^2 \uds\big)^{p/2} \big]^{1/p}<\infty$.

%////////////////////////////////////////////////////////////////////////////////////////////////////////////////////////////////////////////////////////////////////////////////////////
%////////////////////////////////////////////////////////////////////////////////////////////////////////////////////////////////////////////////////////////////////////////////////////
\subsection{Assumptions on the continuous-time dynamics}
\label{subsection-AnalyticAssumption}

%_________________________________________________________________________________________________________________________________________________________________________________________________
\paragraph*{The SDE and the terminal condition.}

We assume that the functions $b:[0,T]\times\bR^{d}\to\bR^d$, $\sigma:[0,T]\times\bR^{d}\to\bR^{d\times d}$ in \eqref{canonicSDE}  are $1/2$-H\"older continuous in their time variable, 
Lipschitz continuous and of linear growth in their spatial variables.  
The terminal condition $g : \bR^d \to \bR^n$ is a Lipschitz function. The terminal condition $\xi := g(X_T)$ is thus in $L^p$, for $p \ge 1$.

%_________________________________________________________________________________________________________________________________________________________________________________________________
\paragraph*{The driver of the BSDE.}

We work with drivers $f : [0,T] \times \R^n \times \R^{n \times d} \to \R^n$ having polynomial growth in $y$ and satisfying the so-called monotonicity condition (also known as one-sided Lipschitz condition), while being Lipschitz functions of $z$. Specifically, our drivers $f$ satisfy the growth, monotonicity and regularity conditions stated below.
We choose not include the possibility of a Lipschitz dependence of $f$ on the variable $x$. It can be seen in \cite{LionnetReisSzpruch2015} that this can easily be dealt with. 
For clarity of our results, we exclude it here.

	\begin{itemize}
		\item[\hfgrowth] There exist $m \in \N^*$ and $K_t, K_y, K_z \ge 0$ such that, for all $(t,y,z)\in [0,T]\times \bR^n\times \bR^{n\times d}$, 
			\begin{align*}	
			%\label{equation---assumption.f.growth}
				\Abs{f(t,y,z)} \le K_t  + K_y \abs{y}^m + K_z \abs{z} \ .
			\end{align*}
			That is, $f$ has polynomial growth in $y$ of degree $m$ and linear growth in $z$.
		\item[\hfmon] There exist a constant $M_y \in \R$ such that for all $t,y,y',z$,
			\begin{align*}	
			%\label{equation---assumption.f.mon}
				\scalar{y'-y}{f(t,y',z) - f(t,y,z)} \le M_y \abs{y'-y}^2 \ .
			\end{align*}
			That is, $f$ is monotone (``decreasing'') in the variable $y$. The monotonicity constant $M_y$ can be, but is not necessarily, strictly negative.
		\item[\hfreg] There exist constants $L_t, L_z \ge 0$ such that for all $t,t',y,z'$, 
			\begin{align*}
				%\label{equation---assumption.f.reg}
				\Abs{f(t',y,z') - f(t,y,z)} \le L_t \abs{t'-t}^\half + L_z\abs{z'-z} \ .
			\end{align*}		
			That is, $f$ is  $\half$-H\"older in time and  Lipschitz in $z$.
	\end{itemize}
In a number of places, we need to know about the regularity of $f$ in the variable $y$ (notice that \hfmon does not even imply continuity). We assume that $f$ satisfies the following.
	\begin{itemize}
		\item[\hfregY] There exists a constant $L_y \ge 0$ such that for all $t, y, y', z$, 
			\begin{align*}	
			%\label{equation---assumption.f.regY}
				\Abs{f(t,y',z) - f(t,y,z)} \le L_y \big(1 + \abs{y'}^{m-1} + \abs{y}^{m-1}\big) \abs{y'-y} \ .
			\end{align*}
			That is, $f$ is locally Lipschitz in $y$ with local Lipschitz constant growing polynomially with degree $m-1$.	
	\end{itemize}
We are primarily interested in drivers $f$ that are polynomials in $y$, and for these we see that \hfregY is clearly satisfied indeed.
Finally, we introduce the following monotone growth assumption.
	\begin{itemize}
		\item[\hfmongrowth] There exist constants $\bar{M}_t, \bar{M}_z \ge 0$ and $\bar{M_y} \in \R$ such that for all $t,y,z$, 
			\begin{align*}
				%\label{equation---assumption.f.mongrowth}
				\scalar{y}{f(t,y,z)} \le  \bar{M_t}  +  \bar{M_y} \abs{y}^2 +  \bar{M_z} | z|^2 \ .
						\end{align*}	
	\end{itemize}

	\begin{remark} 
	%\label{remark---mon+gr.imply.mongr}
		\hfmongrowth is a direct consequence of \hfmon and \hfgrowth and $\bar{M_z}$ (as well as $\bar{M}_t$) can be chosen arbitrarily small, as proved below.
		We single out this property because it controls the growth of the driver and therefore the integrability of the solution $Y$. 
		When running the scheme with a tamed driver $f^h$, we only need assumptions similar to \hfgrowth and \hfmongrowth to guarantee moment bounds for the scheme, and thereby non-explosion.

		Let $f$ satisfies {\hfmon} and {\hfgrowth}. For all $t,y,z$ and for any $\alpha > 0$, we have 
			\begin{align*}
				\scalar{y}{f(t,y,z)} 
					&= \scalar{y-0}{f(t,y,z) - f(t,0,z)} + \scalar{y}{f(t,0,z)}																					\\
					&\le M_y \abs{y-0}^2 + \abs{y} \, \big( K_t  + K_z\abs{z} \big)
					%\\
					%&\le (M_y+\alpha)\abs{y}^2 + \frac{1}{4\alpha} \Big( K_t  + K_z\abs{z} \Big)^2											\\
					%&
					\le (M_y+\alpha)\abs{y}^2 + \frac{K_t^2}{2\alpha}  + \frac{K_z^2}{2\alpha}\abs{z}^2		\ .
			\end{align*}
		Hence we can take $\bar{M_t} = \frac{K_t^2}{2\alpha}$ and $\bar{M}_z = \frac{K_z^2}{2\alpha}\abs{z}^2$ arbitrarily small, while taking $\bar{M}_y = M_y + \alpha$.
		We also note that by combining {\hfmon} and {\hfreg} we obtain the general estimate
			\begin{align*}
				\scalar{y'-y}{f(t,y',z')-f(t,y,z)} &= \scalar{y'-y}{f(t,y',z')-f(t,y,z')}															\\
					&\qquad\qquad\qquad\qquad 	+ \scalar{y'-y}{f(t,y,z')-f(t,y,z)}														\\
					&\le M_y \abs{y'-y}^2 + \abs{y'-y} \, L_z \abs{z'-z}																		\\
					&\le (M_y+\alpha)\abs{y'-y}^2 + \frac{L_z^2}{4\alpha} \abs{z'-z}^2 \ .
			\end{align*}
	\end{remark}

%_________________________________________________________________________________________________________________________________________________________________________________________________
\paragraph*{Results from BSDE theory.}

The assumptions \hfgrowth, \hfmon, \hfreg and \hfregY are a more detailed version of assumptions (HY0) and (HY$0_{\text{loc}}$) in \cite{LionnetReisSzpruch2015}*{Section 2.2} and they imply the fundamental BSDE results of Section 2 and 3 of \cite{LionnetReisSzpruch2015}. Essentially, those results are the existence and uniqueness of the solution, a priori bound estimates and the path-regularity theorem. We recall them in Appendix \ref{appendix-BackgroundResults}. Throughout we denote by $(Y_t,Z_t)_{t\in[0,T]}$ the unique solution to \eqref{canonicBSDE}, which we aim at approximating numerically.

%////////////////////////////////////////////////////////////////////////////////////////////////////////////////////////////////////////////////////////////////////////////////////////
%////////////////////////////////////////////////////////////////////////////////////////////////////////////////////////////////////////////////////////////////////////////////////////
\subsection{Time-discretization: dynamics and assumptions} 
\label{sec--AssumptionSection}

%_________________________________________________________________________________________________________________________________________________________________________________________________
We discretize the time-interval $[0,T]$ using a partition $\pi : 0=t_0 < t_1 < \ldots < t_{N-1} < t_N = T$ with $N$ intervals. The modulus of the partition is $\abs{\pi} = \max_{i=0, \ldots, N-1} h_\ip$ where $h_\ip = \tip - \ti$. While our results would hold for more general partitions, we restrict ourselves to regular partitions for notational simplicity. Consequently, given the number $N$ of time-intervals, we work with the partition $\pi^N$ where $t_i = ih$, $h=T/N \in(0,T]$ being the modulus of the partition.

%_________________________________________________________________________________________________________________________________________________________________________________________________
\paragraph*{}
We wish to focus on the numerical approximation of the backward SDE \eqref{canonicBSDE}. So we do not discuss the numerical approximation of the forward SDE \eqref{canonicSDE} and that of the terminal condition. We work with the following assumption regarding the numerical approximation $\xi^N$ of the terminal condition $\xi=g(X_T)$.

	\begin{itemize}
		\item[\hxiN] There exists a constant $c$ (independent of $N$) such that 
				\begin{align*}
					\text{ERR}_h(\xi):= \bE[\, \abs{\xi - \xi^N}^2 ]^\frac12 \leq c\, h^\frac12.
				\end{align*}
			Moreover, for any $p\geq 2$, $\xi^N\in L^p(\cF_T)$.
	\end{itemize}
Given the assumptions made on $b$, $\sigma$ and $g$, one can use the standard Euler scheme for SDEs to produce an approximation $X^N=(X^N_i)_{i = 0, \ldots, N}$ of $X$ and set $\xi^N = g(X^N_N)$. This $\xi^N$ satisfies \hxiN.

%********************************************************************************************************************************************************************************************************
%********************************************************************************************************************************************************************************************************
\subsubsection{The modified explicit schemes}

For $i \in \{0, \ldots, N-1\}$ we denote the Brownian increments by $\Delta W_\tip := W_\tip-W_\ti$. We also take a family of $\R^d$-valued random variables $(H_\ip)_{i=0,\cdots,N-1}$ which approximate ${\Delta W_\tip}/{h}$ and satisfy the following assumption. 

\begin{itemize}
	\item[\hH]
		\begin{enumerate}
			 \item For any $i=0, \ldots, N-1$, $H_\ip$ is independent from $\cF_i$  and satisfies $\bE_i[H_\ip] = 0$ ($H_\ip$ is a martingale increment);
			 \item For any $i=0, \ldots, N-1$, $\bE_i\big[ (H_\ip h)(H_\ip h)^* \big] = \Lambda h \ I_d$, where $\half \leq \Lambda \leq 1$. 
					As a consequence, $\bE_i\big[ \abs{H_\ip h}^2 \big] = \Lambda d h$;
	 		\item There exists a constant $C \ge 0$ (independent of $N$) such that
				\begin{align*}
					\max_{i=0 \ldots N-1} \bE\Bigg[ \Abs{\frac{\Delta W_\tip}{h} - H_\ip }^2 \Bigg] \le C .
				\end{align*}
		\end{enumerate}	
\end{itemize}
In works on numerical methods for BSDEs, $H_\ip$ is often defined as ${\Delta W_\tip}/{h}$ (in which case $\Lambda=1$). 
However, as will be seen in Section \ref{section---qualitative.properties}, in order to have an explicit scheme which is numerically stable (i.e. reproduces qualitative properties of the continuous-time BSDE, such as the posivity of the solution), the $H_\ip$ are required to be bounded. One way to do this is to truncate the Brownian increment $\Delta W_\tip$ to $\Delta W^h_\tip$, by projecting it on the centered ball of radius $R^h$, where $R^h \goesto +\infty$ as $h \goesto 0$. In that case, $H_\ip = {\Delta W^h_\tip}/{h}$. Also, working with $H_\ip$ rather than $\Delta W_\tip/h_\ip$ also allows to include in the analysis tree-based methods such as cubature \cite{CrisanManolarakis2014}.

\paragraph*{}
We work with the following scheme. It is initialized with $Y_N = \xi^N$ (and $Z_N=0$). Then, for $i = N-1$ to $0$, the \emph{output of one step of the scheme when the input is $Y_\ip$} is $(Y_i,Z_i) = S_i(Y_\ip)$ defined by
	\begin{align} 
	\label{equation---reference--definition.of.the.scheme}
		\left\{\begin{aligned}
			Y_i 
			&
			= \bE_i\Big[ Y_\ip + f^h(t_i,Y_\ip,Z_i)h \Big], 
			\\
			Z_i 
			&
			=  \bE_i\Big[ \big( Y_\ip +(1-\theta') f^h(t_i,Y_\ip,0)h \big) H_\ip^* \Big],
		\end{aligned}\right.
	\end{align}
where $\theta' \in [0,1]$ and the driver $f^h$ is a modification of $f$. The precise assumptions on $f^h$ are described later. The global output of the scheme is the sequence of random variables $S^{N,f^h} = \big( (Y_i,Z_i) \big)_{i=0,\cdots,N-1}$ valued in $\R^n \times \R^{n \times d}$.  The superscript $N$ is omitted since the discrete subscript $i \in \{ 0 , \ldots , N \}$ already indicates that $Y_i$ (say) refers to the numerical approximation, while $Y_\ti$ is simply the solution to the continuous-time BSDE at time $\ti$.

This (explicit) scheme corresponds to the case where the parameter $\theta=0$ in \eqref{equation--introduction--heuristic.BTZ.scheme.for.Y}, $\theta'$ being another parameter here.
Most schemes for BSDEs proposed in the literature choose $\theta' = 1$. We show in this work that the scheme converges for any $\theta' \in [0,1]$. However, a reader familiar with the analysis of continuous-time BSDEs will easily see in Sections \ref{section---size.of.scheme} and \ref{section---convergence.of.scheme} that having $\theta' = \theta $ (so, $\theta'=0$) allows to analyze the scheme  \eqref{equation---reference--definition.of.the.scheme} by mimicking more closely the continuous-time analysis, as suggested below.

%_________________________________________________________________________________________________________________________________________________________________________________________________
\paragraph{Discrete-time martingale representation.} 
 
While $Y_i$ is defined as a conditional expectation in \eqref{equation---reference--definition.of.the.scheme}, it is useful to rewrite it using a martingale increment.
Note that such a representation was already used in \cite{ChassagneuxRichou2016}, \cite{ChassagneuxRichou2014} and \cite{BriandDelyonMemin2002}, although the way we use it in the estimates of Sections \ref{section---size.of.scheme} and \ref{section---convergence.of.scheme} differs.

	\begin{lemma} 	\label{lem:mrt} 
		Given $i \in \{0, \ldots, N-1\}$, consider a $\F_\ip$-measurable $\R^n$-valued random variable $\cY_\ip$ 
		as well as a $\F_i$-measurable $\R^{n \times d}$-valued random variable $\cZ_i$, such that $\cY_\ip + f^h(t_i,\cY_\ip,\cZ_i)h \in L^1$. 
		Define $\cY_i = \bE_i\Big[ \cY_\ip + f^h(t_i,\cY_\ip,\cZ_i)h \Big]$.
		Then,  $\cY_i$ can be written  
			\begin{align}	\label{equation---reference--definition.of.Yi.with.martingale}
				\cY_i =  \cY_\ip + f^h(t_i,\cY_\ip,\cZ_i)h  - \Delta M_\ip\ ,	
			\end{align}
		where $\bE_i[\Delta M_\ip] = 0$. 
		Moreover, there exists a unique pair $(\zeta_i,\Delta N_\ip)$, with $\zeta_i$ a $\F_i$-measurable $\R^{n \times d}$-valued random variable
		and $\Delta N_\ip$ a martingale increment orthogonal to $H_\ip h$, such that
			\begin{align}		\label{equation---reference--decomposition.of.martingale.increment}
				\Delta M_\ip = \zeta_i \Lambda^{-1} H_\ip h + \Delta N_\ip .
			\end{align}
		$\zeta_i$ is given by
			\begin{align}		\label{equation---reference--definition.of.zeta}
				\zeta_i =  \bE_i\Big[ \big( \cY_\ip + f^h(t_i,\cY_\ip,\cZ_i)h \big) H_\ip^* \Big].
			\end{align}
	\end{lemma}
The proof of this statement is relatively straightforward. For completeness, it is is presented in Appendix \ref{appendix-proof-discrete-MRT}.

%_________________________________________________________________________________________________________________________________________________________________________________________________
\paragraph{Derivation of the scheme.} 

Lemma \ref{lem:mrt} suggests that, if one has already chosen a time-discretization for the process $(Y_t)$ over $[\ti,\tip]$ given by 
	\begin{align}
		\label{eq:correctYscheme}
		\hat{Y}_i = \bE_i\Big[ \hat{Y}_\ip + f^h(t_i,\hat{Y}_\ip,\hat{Z}_i)h \Big] \ ,
	\end{align}
where the input $\hat{Y}_\ip$ already approximates $Y_\tip$, and $\hat{Z}_i$ is yet to-be-determined, it is more natural to choose subsequently
	\begin{align}
		\label{eq:correctZscheme}
		\hat{Z}_i = \bE_i\Big[ \big( \hat{Y}_\ip + f^h(t_i,\hat{Y}_\ip,\hat{Z}_i)h \big) H_\ip^* \Big]  = \bE_i\Big[ \big( \hat{Y}_\ip + (1-\theta') f^h(t_i,\hat{Y}_\ip,\hat{Z}_i)h \big) H_\ip^* \Big] \ ,
	\end{align}
with $\theta'=0$.
Indeed, since ``$\hat{Z}_i = \zeta_i$'', this means that $(\hat{Y}_i,\hat{Z}_i) =: \hat{S}_i(\hat{Y}_\ip)$ is defined by
	\begin{align*}
		\hat{Y}_i = \hat{Y}_\ip + f^h(t_i,\hat{Y}_\ip,\hat{Z}_i)h - \big( \hat{Z}_i \Lambda^{-1} H_\ip h + \Delta N_\ip \big) \ .
	\end{align*}
Due to the resemblance between the above equation and the continuous-time BSDE, this scheme has the advantage that the analysis carried out in the later Sections \ref{section---size.of.scheme} and \ref{section---convergence.of.scheme} is more natural to the reader accustomed with the continuous-time analysis.
The above explains why it would be natural, and later convenient for the analysis of the scheme, to take $\theta'=\theta$ (hence $\theta'=0$). 
But due to the prevalence and convenience of implementation of the choice $\theta'=1$, we study the scheme for a general $\theta' \in [0,1]$, thus covering both cases.

However, in the scheme \eqref{eq:correctYscheme}-\eqref{eq:correctZscheme}, $\hat{Z}_i$ is defined implicitly (unless $\theta'=1$), which leads to solving a nonlinear equation to compute $\hat{Z}_i$ at each step of the scheme. In theory this could be dealt with, seeing as $z \mapsto f^h(t_i,\hat{Y}_\ip,z)h$ is Lipschitz and therefore a contraction for $h$ small enough, so $\hat{Z}_i$ could be approximated by iteration. But this defeats the purpose of this work, as we aim at fully explicit schemes. Also, we want to avoid imposing restrictions on the size of the time-steps.
So we replace $f^h(t_i,\hat{Y}_\ip,\hat{Z}_i)h$ by $f^h(t_i,{Y}_\ip,0)h$ in our scheme \eqref{equation---reference--definition.of.the.scheme}.
Given the Lipschitz dependence of $f$ in $z$, this creates an error that should not affect the convergence rate of the scheme, which we aim to be $\half$, as for the implicit scheme.

To be able to take advantage of Lemma \ref{lem:mrt} and the natural character of scheme \eqref{eq:correctYscheme}-\eqref{eq:correctZscheme} when analysing scheme \eqref{equation---reference--definition.of.the.scheme}, we introduce the random variable $D_i:=Z_i-\zeta_i$, where $\zeta_i$ is given by \eqref{equation---reference--definition.of.zeta} with $\cY_\ip = Y_\ip$ and $\cZ_i = Z_i$.
This $D_i$ measures the difference between the theoretically-natural scheme and the scheme used in practice.
The proofs of Lemma \ref{lemma-EstimationofGrowthwithStepRestriction} and Proposition \ref{proposition---one-step.size.estimate} are written in such a way that the interested reader can easily track the consequence of $D_i \neq 0$. For the scheme  \eqref{eq:correctYscheme}-\eqref{eq:correctZscheme}, the analogue $\hat{D}_i:=\hat{Z}_i-\hat{\zeta}_i$, where $\hat{\zeta}_i$ is given by \eqref{equation---reference--definition.of.zeta} with $\cY_\ip = \hat{Y}_\ip$ and $\cZ_i = \hat{Z}_i$, is null.

%********************************************************************************************************************************************************************************************************
%********************************************************************************************************************************************************************************************************
\subsubsection{Assumptions on the tamed driver and comments} \label{sec:fh}

We now introduce the general assumptions on the tamed driver $f^h$.  They summarize the fact that we want $f^h$ to enjoy most of the properties of $f$, in particular to preserve as much as possible the monotonicity of \hfmon and \hfmongrowth, but with the polynomial growth of $f$ and its local Lipschitz constant tamed, see \hfgrowth and \hfregY.

These abstract assumptions can be split in three categories. First, we have the growth conditions \htgrowth and \htmongrowth which ensure the non-explosion of the scheme. Second, the monotonicity and regularity conditions \htreg, \htregY and \htmon will ensure the stability of the scheme. Finally, \htfcvg ensures the convergence of the scheme.
Recall that $h \in (0,T]$ and we are interested in doing $h \goesto 0$.

%_________________________________________________________________________________________________________________________________________________________________________________________________
\subsubsection*{Assumptions on the growth}

\begin{itemize}
	\item[\htgrowth] 
		There exist $K^h_t$, $K^h_y$ and $K^h_z$ $\geq 0$ such that,
		for all $(t,y,z) \in [0,T] \times \R^n \times \R^{n \times d}$, 
			\begin{align*}	
				\abs{f^h(t,y,z)} \le K^h_t  + K^h_y \abs{y} + K^h_z \abs{z} \ .
			\end{align*} 
		The constants $K^h_t$, $K^h_y$ and $K^h_z$ may depend on $h$ but in such a way that $(K^h_t)^2 h$, $(K^h_y)^2 h$ and $K^h_z$ are bounded in $h$.
		Also, $\abs{f^h(t,y,z)} \le \abs{f(t,y,z)}$.

	\item[\htmongrowth] There exist $\bar{M}^h_t, \bar{M}^h_z \geq 0$ and $\bar{M}^h_y \in \bR$ such that,
		for all $(t,y,z) \in [0,T] \times \R^n \times \R^{n \times d}$, 
			\begin{align*}
				\scalar{y}{f(t,y,z)} \le  \bar{M}^h_t  +  \bar{M}^h_y \abs{y}^2 +  \bar{M}^h_z  \abs{z}^2 \ .
			\end{align*}
		The constants  $\bar{M}^h_t, \bar{M}^h_y, \bar{M}^h_z$ may depend on $h$, but are bounded in $h$.
\end{itemize}

%_________________________________________________________________________________________________________________________________________________________________________________________________
\subsubsection*{Assumptions on the regularity}

\begin{itemize}
	\item[\htreg] There exist $L^h_t, L^h_z \geq  0$ such that,
		for all $t,t',y,z,z'$,
			\begin{align*}	
				\abs{f^h(t',y,z') - f^h(t,y,z)} \le L^h_t \abs{t'-t}^\half + L^h_z\abs{z'-z} \ .
			\end{align*}		
		$L^h_t$ and $L^h_z$ may depend on $h$, but in a bounded way.

	\item[\htregY] There exist $L^h_y \geq 0$ and a positive function $\RregY$ satisfying \htfcvg such that 
		for all $t,y,y',z$,
			\begin{align*} 	
				\abs{f^{h}(t,y',z) - f^{h}(t,y,z)} \le L^h_y  \abs{y'-y} + \RregY(t,y',y,z).
			\end{align*}	
		$L^h_y$ may depend on $h$, but in such a way that $(L^h_y)^2h$ is bounded in $h$.

	\item[\htmon] There exists $M^h_y \in \R$ and a function $\Rmon$ satisfying \htfcvg such that
		for all $t,y,y',z$,
			\begin{align*}	
				\scalar{y'-y}{f^h(t,y',z) - f^h(t,y,z)} \le M^h_y \abs{y'-y}^2 + \Rmon(t,y',y,z) \ .
			\end{align*}
		$M^h_y$ may depend on $h$, but in a bounded way.
\end{itemize}

%_________________________________________________________________________________________________________________________________________________________________________________________________
\subsubsection*{Assumptions on the convergence}

We need to ensure that $f^h \to f$ as $h\to 0$. This is in some sense a \emph{consistency condition}, ensuring that the output of the scheme converges to the solution to the correct BSDE, and not a BSDE with a different driver. We introduce for this $R^h = f-f^h$.
Also, we need the remainders $\RregY$ and $\Rmon$ to vanish sufficiently fast, so as not to prevent convergence of the scheme.
The following assumption guarantees that $R^h$, $\RregY$ and $\Rmon$ converge to zero. In its statement, $\Rbla$ stands for both of the remainders $\RregY$ and $\Rmon$.
	\begin{itemize}
		\item[\htfcvg] One of the following holds. 
			\begin{enumerate}
				 \item There exist constants $C \ge 0$, $p,q \ge 1$ and $\alpha > 0$ such that for any $y',y,z$
					 \begin{align*}
						 & \abs{R^h(t,y,z)} \le C \ \big( 1 + \abs{y}^{q} + \abs{z}^p \big) \ h^{\alpha} 																											\\ 
						 & \Rbla(t,y',y,z) \leq C \big( 1 + \abs{y'}^{q} + \abs{y}^{q} + \abs{z}^p \big) h^{\alpha}	.
					 \end{align*}
				
				\item There exist constants $C \ge 0$, $p,q \ge 1$, $r_0>0$ and $\beta > 0$ such that, with $r(h)=r_0 h^{-\beta}$, for any $y',y,z$
					\begin{align*}
						&\abs{R^h(t,y,z)} \le C\ \big( 1 + \abs{y}^{q} + \abs{z}^q \big) \ \1_{\{\Abs{f(t,y,z)}>r(h)\}} 																						\\
						& \Rbla(t,y',y,z) \le C \big( 1 + \abs{y'}^{q} + \abs{y}^{q} + \abs{z}^p \big) \ \1_{\{ \abs{f(t,y',z)} > r(h) \text{ or } \abs{f(t,y,z)} > r(h) \}}	.
					\end{align*}

				\item There exist constants $C \ge 0$, $p,q \ge 1$, $r_0>0$ and $\gamma > 0$ such that, with $r(h)=r_0 h^{-\gamma}$, for any $y',y,z$
					\begin{align*}
						& \abs{R^h(t,y,z)} \le C \ \big( 1 + \abs{y}^{q} + \abs{z}^p \big) \ \1_{\{\abs{y}>r(h)\}} 																								\\
						& \Rbla(t,y',y,z) \le C \big( 1 + \abs{y'}^{q} + \abs{y}^{q} + \abs{z}^p \big) \ \1_{\{ \abs{y'} > r(h) \text{ or } \abs{y} > r(h) \}}	.
					\end{align*}						
			\end{enumerate}
	\end{itemize}

	\begin{remark}
		The above \htfcvg implies that $f^h \to f$ as $h\to 0$ pointwise. 
		We would expect from the stability theorems for continuous-time BSDEs that the solutions would converge if $f^h \to f$ uniformly on compacts, which \htfcvg also implies. 
		However, in order to obtain convergence rates for the scheme, stronger assumptions are needed.
	\end{remark}

%////////////////////////////////////////////////////////////////////////////////////////////////////////////////////////////////////////////////////////////////////////////////////////
%////////////////////////////////////////////////////////////////////////////////////////////////////////////////////////////////////////////////////////////////////////////////////////
\subsection{Main result and outline of the proof}
\label{sec:PC}

The path-regularity theorem (see Theorem \ref{theorem---path.regularity}) implies that the distance between the solution $(Y_t,Z_t)_{t\in[0,T]}$ and its projection on the grid, $(Y_\ti, \overline{Z}_\ti)_{i=0 \ldots N-1}$, is of order $h^{1/2}$, where $\overline{Z}_\ti $ is defined as
\begin{align*} 
    %%\label{eq:zbar}
		\overline{Z}_\ti = \bE_i\Big[ \frac{1}{h} \int_\ti^\tip Z_u \ud u \Big] 
			&= \bE_i\left[ \int_\ti^\tip Z_u \ud W_u \ \frac{\Delta W_\tip^*}{h} \right] 				\\
		\nonumber
		&= \bE_i\left[ \Big( Y_\ip + \int_\ti^\tip f(u,Y_u,Z_u)\udu \Big) 	\ \frac{\Delta W_\tip^*}{h} \right] \ .
	\end{align*}

We measure the distance between the numerical approximation $(Y_i,Z_i)_{i=0,\ldots,N}$ and the solution of the BSDE with the following error criterion : 
	\begin{align*}	
		%\label{eq:errornorm}
		\mathrm{ERR}_N = \bigg( \sup_{i=0, \ldots, N} \bE\Big[ \, \abs{Y_\ti - Y_i}^2\, \Big] + \bE\bigg[ \sum_{i=0}^{N-1} \abs{\overline{Z}_\ti - Z_i}^2 h \bigg] \bigg)^{1/2}.
	\end{align*}

%_________________________________________________________________________________________________________________________________________________________________________________________________
\paragraph{Main Result.}

Our principal result, which Sections \ref{section---size.of.scheme} and \ref{section---convergence.of.scheme} are devoted to proving, is that if the driver $f$ is tamed is such a way that the assumption of the previous subsection are satisfied, then the resulting scheme converges. 
Specifically, we define the rate $\mu$ as follows.
	\begin{itemize}
		\item If \htfcvg.1 is satisfied, with $q,p \ge 1$ and $\alpha > 0$, then $\mu = \alpha$.
		\item If \htfcvg.2 is satisfied, with $q,p \ge 1$ and $\beta > 0$, then $\mu = \frac{\beta l}{2}$, for arbitrary $l \ge 1$. 
				Since $\beta > 0$, we will take $l$ such that $\mu \ge 1$.
		\item If \htfcvg.3 is satisfied, with $q,p \ge 1$ and $\gamma > 0$, then $\mu = \frac{\gamma l}{2}$, for arbitrary $l \ge 1$. 
				Since $\gamma > 0$, we will take $l$ such that $\mu \ge 1$.
	\end{itemize}

	\begin{theorem}	\label{theorem---main--convergence.of.tamed.scheme}	
		Assume that $f^h$ satisfies \htgrowth, \htmongrowth, \htreg, \htregY, \htmon and \htfcvg. 
		Then the scheme \eqref{equation---reference--definition.of.the.scheme} converges, i.e. $\mathrm{ERR}_N \goesto 0$ as $N \goesto +\infty$. 
				
		More precisely, there exists a constant $C$ (independent of $N$) such that
			\begin{align*}
				\mathrm{ERR}_N^2 \le C \, h + C \ h^{\mu} .
			\end{align*}
	\end{theorem}

Having $\mu > 0$ guarantees convergence of the scheme. Having $\mu \ge 1$ guarantees that taming the driver does not slow down the convergence and that the standard convergence rate of the implicit scheme can be recovered.

%_________________________________________________________________________________________________________________________________________________________________________________________________
\paragraph{Outline of the proof.}

We follow a standard strategy which consists in seeing the error at time $\ti$ as resulting from the one-step time-discretization error ---by how much the BSDE and the scheme differ over one time-step when initialized with the same input--- and the propagation to time $\ti$ of the error already present at time $\tip$ ---the control of this latter error coming from the stability of the scheme.

To express this, we introduce the family of random variables $\big( \widehat{Y}_i,\widehat{Z}_i \big)_{i = 0, \ldots, N-1}$ defined, for all $i$, by 
	\begin{align}
		\label{eq:SchemeOutputWithTrueInput-Y}
			\widehat{Y}_i &= \bE_i\left[ Y_\tip + f^h(t_i,Y_\tip,\widehat{Z}_i)h \right]				\\
  	 	\label{eq:SchemeOutputWithTrueInput-Z}
			\widehat{Z}_i &= \bE_i\left[ \Big(Y_\tip + (1-\theta') f^h(t_i,Y_\tip,0)h \Big) H_\ip^* \right] \ .
				\end{align}
Otherwise said, $(\widehat{Y}_i,\widehat{Z}_i)$ is the output of one step of the scheme \eqref{equation---reference--definition.of.the.scheme} when the input is $Y_\tip$, the value of the solution at the time $\tip$.
Then, the above mentioned decomposition of the error at time $\ti$ writes as
	\begin{align*}
		Y_\ti-Y_i
			= \underbrace{ Y_\ti -	\widehat{Y}_i }_\text{one-step error} 
				+ \underbrace{ \widehat{Y}_i -  Y_{i} }_\text{propagation of error}, 
		\quad \text{and} \quad 
		\bar{Z}_\ti-{Z}_i	
			= 	\underbrace{ \bar{Z}_\ti - \widehat{Z}_i }_\text{one-step error}
				+ \underbrace{ \widehat{Z}_i -  Z_{i} }_\text{propagation of error}.
	\end{align*}
	
For the time-discretization errors, we define	
	\begin{align*}
		\tau_i(Y) = \bE\big[ \abs{Y_\ti - \widehat{Y}_i}^2 \big] \qquad\text{and}\qquad \tau_i(Z) = \bE\big[ \abs{\overline{Z}_\ti - \widehat{Z}_i}^2 \big] h .
	\end{align*} 
To study the propagation of errors, we introduce the following notion of stability.

	\begin{definition} \label{definition---almost.stability}
		We say that the	scheme \eqref{equation---reference--definition.of.the.scheme} is \emph{almost-stable} if there exist constants $c$ and $C$, independent of $N$, 
		such that for all $i \in \{0, \ldots , N-1\}$
			\begin{align*}
				\bE[\abs{\widehat{Y}_i - Y_i}^2] + \frac{1}{4} \bE[\abs{\widehat{Z}_i - Z_i}^2]  d h  %+ \bE[\abs{\delta \Delta N_\ip}^2] 
					\le (1+c \, h) \bE[\abs{Y_\tip - Y_\ip}^2]  + C h^{\mu+1} . % + (C^h_i)^\half h^{\mu+\half} + C^h_i h^{\mu} .
			\end{align*}
		The terms $C h^{\mu+1}$ are the \emph{stability imperfections}.
	\end{definition}

From the fundamental lemma below, the global error is then controlled by three terms. The first is the error made on approximating the terminal condition. The second is the sum, essentially, of the one-step discretization errors, $\sum_{i=0}^{N-1} \frac{\tau_i(Y)}{h}   + \tau_i(Z)$. The third is the total contribution of the stability imperfections, and is of order $h^\mu$.

	\begin{lemma}[Fundamental Lemma]		\label{lem:fundamental}
		Assume that the scheme \eqref{equation---reference--definition.of.the.scheme} is almost-stable.
		Then there exist a constant $C \ge 0$ such that, for all $N \ge 1$, 
			\begin{align*}
				\big(\mathrm{ERR}_N\big)^2 \le C \bE[\,\abs{\xi - \xi^N}^2] + C \Bigg( \sum_{i=0}^{N-1} \frac{\tau_i(Y)}{h}   + \tau_i(Z)  \Bigg)  + C h^\mu .
			\end{align*}	
	\end{lemma}

We place the proof in Appendix \ref{sec--proof.of.lemmas---iteration.and.fundamental}.
From \hxiN, the first term is known to be of order $h$. So Theorem \ref{theorem---main--convergence.of.tamed.scheme} will be proved if we can prove that the sum of discretization errors if of order $h$ and that the scheme is almost-stable.

	\begin{remark}[On the almost-stability]
		If one could take $C=0$ in the definition \ref{definition---almost.stability} of the almost-stability, then the scheme would be \emph{stable}, in the usual sense. 
		This is the case for the standard explicit and implicit (BTZ) schemes for BSDEs with Lipchitz drivers.
	\end{remark}

\section{Size estimates and non-explosion of the schemes}  
\label{section---size.of.scheme}

In this section, we undertake the size-analysis for the modified explicit scheme \eqref{equation---reference--definition.of.the.scheme} and show that, under our assumptions, it cannot explode. Specifically, we obtain bounds on the $p$-moments of the scheme that are uniform in $N$. For this, we first carry out the size-analysis for \emph{one step} of the scheme. Thanks to the linear growth \emph{and} the monotone growth of the driver $f^h$, we obtain an estimate which, unlike that for the explicit BTZ scheme (see \cite{LionnetReisSzpruch2015}), 
can satisfactorily be iterated. This then leads to an almost-sure and uniform-in-$N$ global bound, which in turn leads to bounds on the moments.

%////////////////////////////////////////////////////////////////////////////////////////////////////////////////////////////////////////////////////////////////////////////////////////
%////////////////////////////////////////////////////////////////////////////////////////////////////////////////////////////////////////////////////////////////////////////////////////
\subsection{The one-step almost-sure estimate}
%\label{sec:onestepestimates}

The first results are useful estimates about the size of the output of the scheme over a single time-step. For clarity of the computations we first prove the result for scheme \eqref{eq:correctYscheme}-\eqref{eq:correctZscheme} (which has $\hat{D}_i=0$) in Lemma \ref{lemma-EstimationofGrowthwithStepRestriction} then extend it to the scheme \eqref{equation---reference--definition.of.the.scheme} (for which $D_i = Z_i-\zeta_i \neq 0$) in Lemma \ref{proposition---one-step.size.estimate}. We do this in order to show the differences in the estimations. 
Also, the first result needs a smallness assumption on the step size $h$ while the second, due to the enhanced estimation, \emph{does not}.  

  \begin{lemma}	\label{lemma-EstimationofGrowthwithStepRestriction}	
		Assume that the driver $f^h$ satisfies \htgrowth and \htmongrowth with $\bar{M}^h_z \le \frac{d}{8}$. Let $h$ be such that $h \le h_0:=\frac{d}{24 (K^h_z)^2}$.
		Then there exist constants $c, C \ge 0$ such that, for any $i\in \{0, \ldots, N-1\}$, and for any random variable $Y_\ip \in L^2(\F_\ip)$, 
		with $(Y_i,Z_i)$ the output of scheme \eqref{eq:correctYscheme}-\eqref{eq:correctZscheme} for the input $Y_\ip$, one has
			\begin{align*}
				\abs{Y_i}^2 + \frac{1}{8} \abs{Z_i}^2  d h+ \bE_i[\Delta N_\ip^2] &\le \big( 1 +  c h \big) \bE_i[\abs{Y_\ip}^2] + C h.
			\end{align*}	
		The constants $c$ and $C$ are uniform in $N$. %(see Remark \ref{rem:ontheconstChch} below). 
	\end{lemma}

	\begin{proof}
		By Lemma \ref{lem:mrt} we write 
			\begin{align*}
				Y_i & = Y_\ip + f^h(t_i,Y_\ip,Z_i)h - \Delta M_\ip 			\\
					&= Y_\ip + f^h(t_i,Y_\ip,Z_i)h - \big( ( Z_i - D_i )\Lambda^{-1} H_\ip h + \Delta N_\ip \big) \ .
			\end{align*}
		Squaring $Y_i + \Delta M_\ip = Y_\ip + f^h(t_i,Y_\ip,Z_i) h$ and taking $\bE_i$, we have
			\begin{align*}
				\abs{Y_i}^2 + \bE_i[\abs{\Delta M_\ip}^2] &= \bE_i[\abs{Y_\ip}^2] + \bE_i\Big[ 2\scalar{Y_\ip}{f^h(t_i,Y_\ip,Z_i)} \Big] h + \bE_i\Big[ \abs{f^h(t_i,Y_\ip,Z_i)}^2 \Big] h^2	 \ .					
			\end{align*}
		Using \htmongrowth leads to 
			\begin{align*}
				\abs{Y_i}^2 &+ \bE_i[\abs{\Delta M_\ip}^2]																																												\\
					&\le \Big( 1 + 2\bar{M}^h_y h \Big)\bE_i[\abs{Y_\ip}^2] + 2\bar{M}^h_t  h + 2\bar{M}^h_z \abs{Z_i}^2 h + \bE_i\Big[ \abs{f^h(t_i,Y_\ip,Z_i)}^2 \Big] h^2 	  .
			\end{align*}
		Due to the orthogonality of $H_\ip h$ and $\Delta N_\ip$ and using a Young inequality, we have
			\begin{align*}
				\bE_i[\abs{\Delta M_\ip}^2] &= \abs{\zeta_i}^2 \Lambda^{-2} \bE_i[\abs{H_\ip h}^2] + \bE_i[\abs{\Delta N_\ip}^2]																	\\
					&= \abs{Z_i - D_i}^2 \Lambda^{-1} d h + \bE_i[\abs{\Delta N_\ip}^2]																																	\\
					&\ge \half \abs{Z_i}^2  \Lambda^{-1} d h - \abs{D_i}^2 \Lambda^{-1} d h  + \bE_i[\abs{\Delta N_\ip}^2] .
			\end{align*}		
		So the estimate currently yields
			\begin{align*}
				\abs{Y_i}^2 +& \half \abs{Z_i}^2 \Lambda^{-1} d h  - \abs{D_i}^2  \Lambda^{-1} d h + \bE_i[\abs{\Delta N_\ip}^2]																											\\
					&\le \Big( 1 + 2\bar{M}^h_y  h \Big)\bE_i[\abs{Y_\ip}^2] + 2\bar{M}^h_t  h + 2\bar{M}^h_z \abs{Z_i}^2 h 	+ \bE_i\Big[ \abs{f^h(t_i,Y_\ip,Z_i)}^2 \Big] h^2,
			\end{align*}
		hence, since $1 \le \Lambda^{-1} \le 2$,
			\begin{align*}
				\abs{Y_i}^2 +& \Big(\half - \frac{2\bar{M}^h_z }{d} \Big) \abs{Z_i}^2  d h + \bE_i[\abs{\Delta N_\ip}^2]																					\\
					&\le \Big( 1 + 2\bar{M}^h_y  h \Big)\bE_i[\abs{Y_\ip}^2] + 2\bar{M}^h_t  h + 2 \abs{D_i}^2  d h + \bE_i\Big[ \abs{f^h(t_i,Y_\ip,Z_i)}^2 \Big] h^2.
			\end{align*}			
		Now, using the growth of $f^h$ given by \htgrowth, we obtain that
					\begin{align*}
						\bE_i\left[ \abs{f^h(t_i,Y_\ip,Z_i)}^2 \right] h^2 \le 3 (K^h_t)^2 h^2 + 3 (K^h_y)^2 h^2 \bE_i[\abs{Y_\ip}^2] + 3 (K^h_z)^2 \abs{Z_i}^2 h^2 \ .
		\end{align*}
		This immediately implies the core estimate
			\begin{align}		\label{equation---one.step.size.estimate.stopthere.equation}
				\nonumber
				\abs{Y_i}^2 +& \Big(\half - \frac{2\bar{M}^h_z }{d} \Big) \abs{Z_i}^2  d h + \bE_i[\abs{\Delta N_\ip}^2]																						\\
				\nonumber
					&\le \left( 1 + \big[ 2\bar{M}^h_y  + 3 (K^h_y)^2 h \big] h \right)\bE_i[\abs{Y_\ip}^2]																												\\
									&\hspace{3cm} + \left( 2\bar{M}^h_t + 3 (K^h_t)^2 h\right)  h + 2 \abs{D_i}^2  d h + 3 (K^h_z)^2 \abs{Z_i}^2 h^2 .
			\end{align}
		Now, recall that $(Y_i,Z_i)$ are produced by the scheme \eqref{eq:correctZscheme}-\eqref{eq:correctZscheme} with input $Y_\ip$, and we have therefore $D_i = 0$ here. 
		Also, from the assumptions we have 
			\begin{align*}
				\frac{2\bar{M}^h_z }{d} \le \frac{1}{4} .
				\qquad\text{and}\qquad	
				3 \frac{ \, (K^h_z)^2}{d} h \le \frac{1}{8}	
			\end{align*}
		This allows to pass the term in $\abs{Z_i}^2 h^2$ on the RHS of the inequality to its LHS. 
	\end{proof}

We now state the one-step estimate for \eqref{equation---reference--definition.of.the.scheme}, for which $D_i\neq 0$ a priori. 
We emphasize the additional assumption \htreg, the {absence} of a smallness condition on $h$, and that our handling of the end of the proof is slightly different (even in the case $D_i=0$).

	\begin{proposition}		\label{proposition---one-step.size.estimate} 
		Let the driver $f^h$ satisfy \htgrowth, \htmongrowth with $\bar{M}^h_z \le \frac{d}{8}$ and \htreg. 
		Then there exist $c, C \ge 0$ such that, for any $i\in \{0, \ldots, N-1\}$, and for any random variable $Y_\ip \in L^2(\F_\ip)$, 
		with $(Y_i,Z_i)$ the output of scheme \eqref{equation---reference--definition.of.the.scheme} for the input $Y_\ip$, one has
			\begin{align*}
				\abs{Y_i}^2 + \frac{1}{4} \abs{Z_i}^2  d h + \bE_i[\abs{\Delta N_\ip}^2] 
					\le \big( 1 +  c h \big) \bE_i[\abs{Y_\ip}^2] + C h.
			\end{align*}			
		The constants $c$ and $C$ are uniform in $N$.
	\end{proposition}

	\begin{proof}
		We resume from estimate \eqref{equation---one.step.size.estimate.stopthere.equation} of the previous proof, before passing the $|Z_i|^2 h^2$ to the LHS. 
		We start by estimating $\abs{D_i}^2 d h$ as a function of $|Z_i|^2 h^2$.
		
		Following the definition of $Z_i$ in \eqref{equation---reference--definition.of.the.scheme} and that of $\zeta_i$ in \eqref{equation---reference--definition.of.zeta}, we have
			\begin{align*}
				-D_i = \zeta_i - Z_i 
					&=  \bE_i\Big[ \big( Y_\ip + f^h(t_i,Y_\ip,Z_i)h \big) H_\ip^* \Big] - \bE_i\Big[ \big( Y_\ip + (1-\theta') f^h(t_i,Y_\ip,0)h \big) H_\ip^* \Big]												\\
					&= \bE_i\Big[ \big( f^h(t_i,Y_\ip,Z_i) - (1-\theta') f^h(t_i,Y_\ip,0) \big) h\ H_\ip^* \Big]  .
			\end{align*}
		Using the Cauchy--Schwartz inequality,
			\begin{align*}
				\abs{D_i}^2 d h
					&\le dh  \bE_i[\abs{H_\ip}^2] \times \bE_i\bigg[\Abs{\big( f^h(t_i,Y_\ip,Z_i) - (1-\theta') f^h(t_i,Y_\ip,0) \big) h}^2\bigg]																				\\
					&= \Lambda d^2 \times  \bE_i\bigg[\Abs{ \Big( f^h(t_i,Y_\ip,Z_i) - f^h(t_i,Y_\ip,0) \Big) + \theta' f^h(t_i,Y_\ip,0) }^2\bigg] h^2																		\\
					&\le 2 d^2 (L^h_z)^2 \abs{Z_i}^2 h^2 + 2d^2 \theta'^2 \bE_i[\abs{f^h(t_i,Y_\ip,0)}^2] h^2 \ ,
			\end{align*}
		where we used the $z$-regularity of $f^h$ from \htreg, $(a+b)^2 \le 2a^2 + 2b^2$ and $\Lambda \le 1$.
		
		Injecting this in \eqref{equation---one.step.size.estimate.stopthere.equation}, we obtain
			\begin{align*}
				\abs{Y_i}^2 + \Big(\half &- \frac{2\bar{M}^h_z }{d} \Big) \abs{Z_i}^2  d h + \bE_i[\abs{\Delta N_\ip}^2]																								\\
					& \le \left( 1 + \big[ 2\bar{M}^h_y + 3 (K^h_y)^2 h \big] h \right)\bE_i[\abs{Y_\ip}^2] + \left( 2\bar{M}^h_t + 3 (K^h_t)^2 h\right) h												\\
							& \qquad \qquad + \left( 3 (K^h_z)^2 + 2 d^2 (L^h_z)^2 \right) \abs{Z_i}^2 h^2 + 2d^2 \theta'^2 \bE_i\big[\abs{f^h(t_i,Y_\ip,0)}^2\big] h^2 .
			\end{align*}		
		Using again the growth assumption \htgrowth, we have
			\begin{align*}
				\bE_i\big[\, |f^h(t_i,Y_\ip,0)|^2 \big] h^2 \le 2 (K^h_t)^2 h^2 + 2 (K^h_y)^2 h^2 \bE_i[\abs{Y_\ip}^2].
			\end{align*}
		Hence, the previous estimate implies that
			\begin{align*}
				\abs{Y_i}^2 +& \Big(\half - \frac{2\bar{M}^h_z }{d} \Big) \abs{Z_i}^2  d h + \bE_i[\abs{\Delta N_\ip}^2]																												\\
					&\le \left( 1 + \big[ 2\bar{M}^h_y + 3 (K^h_y)^2 h + 4d^2 \theta'^2 (K^h_y)^2 h \big] h \right)\bE_i[\abs{Y_\ip}^2]																							\\
							&\qquad \qquad + \left( 2\bar{M}^h_t + 3 (K^h_t)^2 h + 4d^2 \theta'^2 (K^h_t)^2 h\right)h + \left( 3 (K^h_z)^2 + 2 d^2 (L^h_z)^2 \right) \abs{Z_i}^2 h^2.
			\end{align*}
								
		At this stage, instead of assuming $h$ small enough and passing the term in $\abs{Z_i}^2 h^2$ from the RHS to the LHS as in the previous proof, 
		we estimate $Z_i$ directly from its explicit definition in \eqref{equation---reference--definition.of.the.scheme}. 
		From the Cauchy--Schwartz inequality, \hH, $\Lambda \le 1$ and \htgrowth,
			\begin{align*}
				\abs{Z_i}^2 h 
					&
					= \big|\bE_i\big[ \big( Y_\ip + (1-\theta') f^h(t_i,Y_\ip,0) h \big) H_\ip ^* \big] \big|^2 h
					\\
					&
					%\le h \bE_i\big[ \abs{H_\ip}^2 \big] \times \bE_i\big[ \,\big|{ Y_\ip + (1-\theta') f^h(t_i,Y_\ip,0) h }\big|^2 \big]
					%\\
					%&
					%= 
					\leq 
					\Lambda d \, \bE_i\big[ \,\big|{ Y_\ip + (1-\theta') f^h(t_i,Y_\ip,0) h }\big|^2 \big]
					\\
					&
					\le 2  d \bE_i\big[ \abs{Y_\ip}^2 \big] + 2 d (1-\theta')^2 \bE_i\big[\, \big|f^h(t_i,Y_\ip,0)\big|^2 \big] h^2
					%\\
					%&
					%\le 2  d \bE_i\big[ \abs{Y_\ip}^2 \big] + 4  d (1-\theta')^2 (K^h_t)^2 h^2 + 4  d (1-\theta')^2 (K^h_y)^2 h^2 \bE_i\big[ \abs{Y_\ip}^2 \big]
					\\
					& 
					\leq 
					\big( 2 d + 4  d (1-\theta')^2 (K^h_y)^2 h^2 \big)\bE_i[ \abs{Y_\ip}^2] + 4  d (1-\theta')^2 (K^h_t)^2 h^2.
			\end{align*}		
		Injecting this estimate into the previous one and re-arranging the terms we obtain
			\begin{align*}
				\abs{Y_i}^2 + \Big(\half - \frac{2\bar{M}^h_z }{d} \Big) \abs{Z_i}^2 d h + \bE_i[\abs{\Delta N_\ip}^2]
					&\le \big( 1 + c^h  h \big) \bE_i[\abs{Y_\ip}^2] + C^h  h \ ,
			\end{align*}
		where $c^h$ and $C^h$ are given by 
			\begin{align*}
				c^h &:=  2\bar{M}^h_y + 3 (K^h_y)^2 h + 4d^2 \theta'^2 (K^h_y)^2 h 																																	\\
						&\qquad\qquad\qquad + \left( 3 (K^h_z)^2 + 2 d^2 (L^h_z)^2 \right) \left( 2  d + 4  d (1-\theta')^2 (K^h_y)^2 h^2 \right)											\\
				C^h &:=  2\bar{M}^h_t + 3 (K^h_t)^2 h + 4d^2 \theta'^2 (K^h_t)^2 h 																																	\\
						& \qquad\qquad\qquad + \Big( 3 (K^h_z)^2 + 2 d^2 (L^h_z)^2 \Big) 4  d (1-\theta')^2 (K^h_t)^2 h^2  .
			\end{align*}
		Now, first, oberve that since we have assumed $\frac{2\bar{M}^h_z }{d} \le \frac{1}{4}$, the LHS simplifies. %into that in the statement of the proposition.
		Second, from the assumptions on the constants that are made in \htgrowth, \htmongrowth and \htreg, and the fact that $h \le T$, there exist $c$ and $C$ such that $c^h \le c$ and $C^h \le C$.	
	\end{proof}

	\begin{remark}[On the constants $c^h$ and $C^h$]	\label{rem:ontheconstChch}
		The constants $c^h$ and $C^h$ defined in the above proof depend on $h$ in a bounded way. 
		
		Firstly, since $(K^h_t)^2 h$ and $(K^h_y)^2 h$ are bounded, we see that $(K^h_t)^2 h^2$ and $(K^h_y)^2 h^2$ vanish as $h \goesto 0$. 
		So the last term in $C^h$ and most of the last term in $c^h$ vanish. Here, our handling of the estimates allows to conclude without restrictions on $h$, at the price of having bigger constants. 
		As $h \goesto 0$, these bigger constant decrease to, essentially, what they would have been if we had assumed $h$ small 
		and handled the estimates as in lemma \ref{lemma-EstimationofGrowthwithStepRestriction}.
		
		Secondly, we note that if $\theta'=\theta=0$ and $f$ does not depend on $z$ (in which case $K^h_z=L^h_z=0$), then the scheme \eqref{equation---reference--definition.of.the.scheme}
		coincides with the scheme \eqref{eq:correctYscheme}-\eqref{eq:correctZscheme} and the constants $c^h$ and $C^h$ are the same as those found in the previous lemma.
		As will be clearer in the proof of the almost-stability of the scheme (Proposition \ref{proposition---stability.of.the.tamed.scheme.NEWVERSION.FINAL.DEFINITIVE.FOR.REAL}),
		using scheme \eqref{equation---reference--definition.of.the.scheme} (not taking $\theta=0$ and replacing $Z_i$ by zero in the definition of $Z_i$)
		creates some errors that vanish at the same rate as the standard error, and only make for bigger constants.		
	\end{remark}

%////////////////////////////////////////////////////////////////////////////////////////////////////////////////////////////////////////////////////////////////////////////////////////
%////////////////////////////////////////////////////////////////////////////////////////////////////////////////////////////////////////////////////////////////////////////////////////
\subsection{The global almost-sure estimate}

The one-step size estimate of Proposition \ref{proposition---one-step.size.estimate} can be readily iterated to yield an informative almost-sure bound on the size of the $Y_i$'s, where $(Y_i,Z_i)_{i=0, \ldots, N-1}$ is the output from the iteration of scheme \eqref{equation---reference--definition.of.the.scheme} with terminal condition initialized to $\xi^N$.

	\begin{proposition}		\label{prop:aspathestimate}
		Under \htgrowth, \htmongrowth and \htreg, for any $i \in \{0, \ldots, N-1\}$,
			\begin{align*}
				\abs{Y_i}^2 + \bE_i\Bigg[ \frac{1}{4} \sum_{j=i}^{N-1} \abs{Z_j}^2 d h + \sum_{j=i}^{N-1} \abs{\Delta N_\jp}^2 \Bigg] 
					\le e^{c (T-\ti)} \bE_i[\abs{\xi^N}^2] + e^{c (T-\ti)} C (T-\ti).
			\end{align*}
	\end{proposition}

	\begin{proof}
		This proof follows by directly iterating Proposition \ref{proposition---one-step.size.estimate} ---see Lemma \ref{lemma---iterating.a.one.step.estimate}.	
	\end{proof}

%////////////////////////////////////////////////////////////////////////////////////////////////////////////////////////////////////////////////////////////////////////////////////////
%////////////////////////////////////////////////////////////////////////////////////////////////////////////////////////////////////////////////////////////////////////////////////////
\subsection{Moment estimates}

We now show that $(Y_i,Z_i)_{i=0, \ldots, N}$ has $p$-moments which are bounded uniformly in $N$. 
This is crucial in the next section to show that, while the scheme might not be strictly stable, the stability imperfections are small enough that the scheme is almost-stable.

	\begin{proposition}	\label{proposition---moment.estimate}
		Assume \hxiN, \htgrowth, \htmongrowth and \htreg. For every $p \ge 1$, there exists a constant $C$ (independent on $N$) such that
			\begin{align*}
				\sup_{i=0, \ldots, N-1} \bE\big[\, \abs{Y_i}^{2p}\big] \le C 
				\qquad\text{and}\qquad
				\bE\bigg[ \sum_{i=0}^{N-1} \big(\abs{Z_i}^2h\big)^p \bigg] \le C .
			\end{align*}
	\end{proposition}

The proof of this result is somewhat similar to that in \cite{LionnetReisSzpruch2015}*{Proposition 5.1}. For the convenience of the reader, we place it in Appendix \ref{sec--proof.of.proposition---Prop:Integrability}.

\paragraph*{}
We have analogue results for $(\widehat{Y}_i,\widehat{Z}_i)_{i=0, \ldots, N-1}$ defined in \eqref{eq:SchemeOutputWithTrueInput-Y}-\eqref{eq:SchemeOutputWithTrueInput-Z}.

	\begin{proposition}[Integrability of $\widehat Z_i$]	\label{lemm:integrabilityWidehatZi}
		Let \htgrowth and the assumptions of Section \ref{subsection-AnalyticAssumption} on $f$ hold. 
		Then, for any $p\geq 1$ there exists $C \ge 0$ such that, for any $N \ge 1$,
			\begin{align*}
				\sup_{i=0, \ldots, N-1} \bE\big[\, \abs{\widehat Y_i}^{2p} \big] \le C 
					\qquad\text{and}\qquad
				\sup_{0 \le i \le N-1}\bE\big[ \abs{\widehat{Z}_i}^{2p} \big] \le C.
			\end{align*}
		Consequently, we also have the estimate
			%\begin{align*}
				$\bE\big[ \sum_{i=0}^{N-1}  \big( \abs{\widehat{Z}_i}^2 h\big)^{p} \big] \le C$.
			%\end{align*}
	\end{proposition}

	\begin{proof}
		We start by estimating $\widehat{Z}_i$.
		Using the Cauchy--Schwartz inequality, $1-\theta' \le 1$, the fact that $\bE_i[ \abs{H_\ip}^2 ] = \frac{\Lambda d}{h}$ and $(a+b)^2 \le 2 (a^2 + b^2)$,
			\begin{align*}
				\abs{\widehat{Z}_i}
					& 
					=  
					\Abs{ \bE_i\Big[ \Big( Y_\tip - \bE_i[Y_\tip] + (1-\theta') f^h(t_i,Y_\tip,0)h \Big) H_\ip^* \Big]	}
					\\
					%&\le \bE_i\Big[ \Abs{ Y_\tip - \bE_i[Y_\tip] + (1-\theta') f^h(t_i,Y_\tip,0)h }^2 \Big]^\half	\ \bE_i\Big[ \abs{H_\ip}^2 \Big]^\half
					%\\
					&
					\le 
					2^\half \Big( \bE_i\big[ \Abs{ Y_\tip - \bE_i[Y_\tip] }^2 \big] + \bE_i\big[ \abs{f^h(t_i,Y_\tip,0)}^2 h^2 \big] \Big)^\half	\ \Big( \frac{\Lambda d}{h} \Big)^\half		.
			\end{align*}
		So for all $p \ge 2$, since $(a+b)^q \le 2^{q-1} (a^q+b^q)$,
			\begin{align*}
				\abs{\widehat{Z}_i}^p
					%&
					%\le 2^{\frac{p}{2}} \Big( \bE_i\big[ \Abs{ Y_\tip - \bE_i[Y_\tip] }^2 \big] + \bE_i\big[ \abs{f^h(t_i,Y_\tip,0)}^2 h^2 \big] \Big)^{\frac{p}{2}}	\ \Big( \frac{\Lambda d}{h} \Big)^{\frac{p}{2}} 
					%\\
					%&
					%\le 2^{\frac{p}{2}} 2^{\frac{p}{2}-1} \Big( \bE_i\big[ \Abs{ Y_\tip - \bE_i[Y_\tip] }^2 \big]^{\frac{p}{2}} + \bE_i\big[ \abs{f^h(t_i,Y_\tip,0)}^2 h^2 \big]^{\frac{p}{2}} \Big)
									%\ \Big( \frac{\Lambda d}{h} \Big)^{\frac{p}{2}} 
									%\\
					&
					\le 
					2^{p-1} \Big( \bE_i\big[ \Abs{ Y_\tip - \bE_i[Y_\tip] }^2 \big]^{\frac{p}{2}} + \bE_i\big[ \abs{f^h(t_i,Y_\tip,0)}^2 h^2 \big]^{\frac{p}{2}} \Big)
									\ \Big( \frac{\Lambda d}{h} \Big)^{\frac{p}{2}} .
			\end{align*}
		We now take the expectation,
			\begin{align*}
				\bE\big[ \abs{\widehat{Z}_i}^p \big]
					&\le 2^{p-1} \bigg( \bE\Big[ \bE_i\big[ \Abs{ Y_\tip - \bE_i[Y_\tip] }^2 \big]^{\frac{p}{2}} \Big] + \bE\Big[ \bE_i\big[ \abs{f^h(t_i,Y_\tip,0)}^2 h^2 \big]^{\frac{p}{2}} \Big] \bigg)
									\ \Big( \frac{\Lambda d}{h} \Big)^{\frac{p}{2}} .
			\end{align*}

		Now, let us first observe that, using $\abs{f^h}\le f$, \hfgrowth, Jensen's inequality, 
		the tower property of expectations 
			%\begin{align*}
				%\bE_i\big[ \abs{f^h(t_i,Y_\tip,0)}^2 h^2 \big]^{\frac{p}{2}} 
					%&
					%\le \bE_i\big[ \abs{f^h(t_i,Y_\tip,0)}^2 \big]^{\frac{p}{2}}  h^p
					%\\
					%&
					%%\le \bE_i\big[ 2(K_t)^2 + 2(K_y)^2\abs{Y_\tip}^{2m} \big]^{\frac{p}{2}}  h^p
					%%\\
					%%&
					%%\le 2^{\frac{p}{2}} \Big( (K_t)^2 + (K_y)^2 \bE_i[\abs{Y_\tip}^{2m}]  \Big)^{\frac{p}{2}}  h^p
					%%\\
					%%&
					%%\le 2^{\frac{p}{2}} 2^{\frac{p}{2}-1} \Big( (K_t)^{2\frac{p}{2}} + (K_y)^{2\frac{p}{2}} \bE_i[\abs{Y_\tip}^{2m}]^{\frac{p}{2}}  \Big)  h^p
					%%\\
					%%&
					%%\le 2^{p-1} \Big( (K_t)^{2\frac{p}{2}} + (K_y)^{2\frac{p}{2}} \bE_i[\abs{Y_\tip}^{2m\frac{p}{2}}]  \Big)  h^p
					%%\\
					%%& 
					%\leq   2^{p-1} \Big( (K_t)^{2\frac{p}{2}} + (K_y)^{2\frac{p}{2}} \bE_i[\abs{Y_\tip}^{pm}] \Big)  h^p	.
			%\end{align*}
		%Therefore, taking expectations
		and using Theorem \ref{theorem---path.regularity}, we have
			\begin{align*}
				\bE\big[ \bE_i\big[ \abs{f^h(t_i,Y_\tip,0)}^2 h^2 \big]^{\frac{p}{2}} \big]
					& 
					=  2^{p-1} \Big( (K_t)^{2\frac{p}{2}} + (K_y)^{2\frac{p}{2}} \bE[\abs{Y_\tip}^{pm}] \Big)  h^p
					%\\
					%&
					%\le 2^{p-1} \Big( (K_t)^{2\frac{p}{2}} + (K_y)^{2\frac{p}{2}} C^Y_{pm} \Big)  h^p
					%\\
					%&
					\le C h^p .
			\end{align*}
		We now address the main difficulty, the term 	$\bE\big[\, \bE_i\big[ \Abs{ Y_\tip - \bE_i[Y_\tip] }^2 \big]^{\frac{p}{2}} \big]$. 
		Using the dynamics of the BSDE over $[\ti,\tip]$ 
%		First, since
%			\begin{align*}
%				Y_\ti &= Y_\tip + \int_\ti^\tip f(u,Y_u,Z_u) du - \int_\ti^\tip Z_u dW_u		\qquad \text{and so}																															\\
%				Y_\ti &= \bE_i\bigg[ Y_\tip + \int_\ti^\tip f(u,Y_u,Z_u) du \bigg] ,
%			\end{align*}
		we easily find the identity 
			\begin{align*}
				Y_\tip - \bE_i[Y_\tip] = \bE_i\bigg[ \int_\ti^\tip f(u,Y_u,Z_u) du \bigg] - \int_\ti^\tip f(u,Y_u,Z_u) du + \int_\ti^\tip Z_u dW_u .
			\end{align*} 
		Taking the square, using $\big( \sum_{i=1}^n a_i \big)^2 \le n \sum_{i=1}^n a_i^2$ and Jensen and/or Cauchy--Schwartz inequalities we obtain
			\begin{align*}
					&
				\abs{ Y_\tip - \bE_i[Y_\tip] }^2 
				%\\
					%&
					%\le 
					%3 \Abs{\bE_i\bigg[ \int_\ti^\tip f(u,Y_u,Z_u) du \bigg]}^2 + 3 \Abs{\int_\ti^\tip f(u,Y_u,Z_u) du}^2 + 3 \Abs{\int_\ti^\tip Z_u dW_u}^2
					\\
					&
					\le 
					3 \bE_i\bigg[ h \int_\ti^\tip \abs{f(u,Y_u,Z_u)}^2 du \bigg] + 3 h \int_\ti^\tip \abs{f(u,Y_u,Z_u)}^2 du + 3 \Abs{\int_\ti^\tip Z_u dW_u}^2 .
			\end{align*} 
		We then take the conditional expectation, use the It?´ isometry and \hfgrowth to have
			\begin{align*}
				\bE_i\big[ &\abs{ Y_\tip - \bE_i[Y_\tip] }^2 \big]
				%\\
					%&
					%\le 
					%6h \bE_i\bigg[ \int_\ti^\tip \abs{f(u,Y_u,Z_u)}^2 du \bigg] + 3 \bE_i\bigg[ \Abs{\int_\ti^\tip Z_u dW_u}^2 \bigg]
					\\
					& 
					\leq  
					6h \bE_i\bigg[ \int_\ti^\tip \abs{f(u,Y_u,Z_u)}^2 du \bigg] + 3d \bE_i\bigg[ \int_\ti^\tip \abs{Z_u}^2 du \bigg]
					\\
					%&\le 6h \bE_i\bigg[ \int_\ti^\tip 3(K_t)^2 + 3(K_y)^2 \abs{Y_u}^{2m} + 3(K_z)^2 \abs{Z_u}^2 du \bigg]
					%%\\
									%%&\hspace{3cm} 
									%+ 3d \bE_i\bigg[ \int_\ti^\tip \abs{Z_u}^2 du \bigg]
									%\\
					& 
					\leq  18(K_t)^2 h^2 + 18(K_y)^2 h \bE_i\bigg[ \int_\ti^\tip \abs{Y_u}^{2m} du \bigg]
					%\\
									%&\hspace{3cm} 
					+ \Big( 18(K_z)^2 h + 3d \Big) \bE_i\bigg[ \int_\ti^\tip \abs{Z_u}^{2} du \bigg] .
			\end{align*} 
		Consequently, taking the power $\frac{p}{2}$ and using$\big( \sum_{i=1}^n a_i \big)^q \le n^{q-1} \sum_{i=1}^n a_i^q$ we have
			\begin{align*}
				\bE_i\big[ \abs{ Y_\tip - \bE_i[Y_\tip] }^2 \big]^{\frac{p}{2}}
					%&
					%\le 
					%3^{\frac{p}{2}-1} \ \big( 18(K_t)^2 h^2 \big)^{\frac{p}{2}} + 3^{\frac{p}{2}-1} \ \big(18(K_y)^2 h\big)^{\frac{p}{2}} \bE_i\bigg[ \int_\ti^\tip \abs{Y_u}^{2m} du \bigg]^{\frac{p}{2}} 
					%\\
									%&\hspace{3cm} + 3^{\frac{p}{2}-1} \ \Big( 18(K_z)^2 h + 3d \Big)^{\frac{p}{2}} \bE_i\bigg[ \int_\ti^\tip \abs{Z_u}^{2} du \bigg]^{\frac{p}{2}}
									%\\
					&
					\le 
					C h^p + C h^{\frac{p}{2}} \bE_i\bigg[ \int_\ti^\tip \abs{Y_u}^{2m} du \bigg]^{\frac{p}{2}} + C \bE_i\bigg[ \int_\ti^\tip \abs{Z_u}^{2} du \bigg]^{\frac{p}{2}}		.
			\end{align*} 
		Here the constant $C$ depends on $p$, the growth constants of $f$ and uses $h \le T$.
		We can then take the expectation, and further use the Jensen inequality (repeatedly) to obtain
			\begin{align*}
				\bE\Big[ \bE_i\big[ \abs{ Y_\tip - \bE_i[Y_\tip] }^2 \big]^{\frac{p}{2}} \Big]
					%&\le C h^p + C h^{\frac{p}{2}} \bE\Bigg[ \bE_i\bigg[ \int_\ti^\tip \abs{Y_u}^{2m} du \bigg]^{\frac{p}{2}}  \Bigg]
							%+ C \bE\Bigg[ \bE_i\bigg[ \int_\ti^\tip \abs{Z_u}^{2} du \bigg]^{\frac{p}{2}} \Bigg]																																	\\
					%&\le C h^p + C h^{\frac{p}{2}} \bE\Bigg[ \bE_i\bigg[ h^{\frac{p}{2}-1} \int_\ti^\tip \abs{Y_u}^{2m\frac{p}{2}} du \bigg]  \Bigg]
							%+ C \bE\Bigg[ \bE_i\bigg[ \bigg(\int_\ti^\tip \abs{Z_u}^{2} du\bigg)^{\frac{p}{2}} \bigg] \Bigg]																												\\
					%&\le C h^p + C h^{p-1} \bE\bigg[ \int_\ti^\tip \abs{Y_u}^{mp} du \bigg]
							%+ C \bE\bigg[ \bigg(\int_\ti^\tip \abs{Z_u}^{2} du\bigg)^{\frac{p}{2}} \bigg] 																																			\\
					%&\le C h^p + C h^{p} \sup_{\ti \le u \le \tip}\bE\big[ \abs{Y_u}^{mp} \big]
							%+ C \bE\bigg[ \bigg(\int_\ti^\tip \abs{Z_u}^{2} du \bigg)^{\frac{p}{2}} \bigg]																																			\\
					%&\le C h^p + C h^p C^Y_{pm} + C \bE\bigg[ \bigg(\int_\ti^\tip \abs{Z_u}^{2} du \bigg)^{\frac{p}{2}} \bigg]																										\\
					&
					\le 
					C h^p + C \bE\bigg[ \bigg(\int_\ti^\tip \abs{Z_u}^{2} du \bigg)^{\frac{p}{2}} \bigg] .
			\end{align*} 
		Gathering the two estimates, we see that we have in the end
			\begin{align*}
				\bE\big[ \abs{\widehat{Z}_i}^p \big]
					&
					\le 
					C \bigg( C h^p + C \bE\bigg[ \bigg(\int_\ti^\tip \abs{Z_u}^{2} du \bigg)^{\frac{p}{2}} \bigg]	+ C h^p \bigg) \ \Big( \frac{\Lambda d}{h} \Big)^{\frac{p}{2}}
					\\
					&
					\le 
					C h^{\frac{p}{2}} + C h^{-\frac{p}{2}} \bE\bigg[ \bigg(\int_\ti^\tip \abs{Z_u}^{2} du \bigg)^{\frac{p}{2}} \bigg] .
			\end{align*}
		We have now proven the pivotal estimate that will allow us to conclude. 
		Using \cite{LionnetReisSzpruch2015}*{Theorem 3.4, equation (3.20)}, i.e. $|Z_t|\leq C(1+|X_t|)$ $\ud t\otimes \ud \bP$-a.s., 
		to better control the last integral term we obtain, in combination with Theorem \ref{theo:BasicExistandUniqTheo}, that
			\begin{align*}
				\bE\big[ \abs{\widehat{Z}_i}^p \big]
					%&
					%\le  
					%C h^{\frac{p}{2}} 
					%+ C h^{-\frac{p}{2}} 
					           %\bE\bigg[ \bigg(\int_\ti^\tip \abs{Z_u}^{2} du \bigg)^{\frac{p}{2}} \bigg]
					%\\
					&
          \le  
					C h^{\frac{p}{2}} 
					+ C h^{-\frac{p}{2}} 
					           \bE\bigg[ \bigg(\int_\ti^\tip \big(1+|X_u|^{2}\big) du \bigg)^{\frac{p}{2}} \bigg]
					\\
					&
          \le  
					C \Big( h^{\frac{p}{2}} 
					            + h^{-\frac{p}{2}} h^{\frac p2}\big(1+\bE[\sup_{0\leq u\leq T}|X_u|^{p}]\big)\Big)
					%\\
					%&
					%\le 
					%C h^{\frac{p}{2}} 
					%+ C h^{-\frac{p}{2}} 
					           %\bE\bigg[ h^{\frac{p}{2}-1} \int_\ti^\tip \abs{Z_u}^{2 \frac{p}{2}} du \bigg]
					%\\
					%& 
					%=  
					%C h^{\frac{p}{2}} + C h^{-1} \int_\ti^\tip \bE\big[ \abs{Z_u}^{p} \big] du
					%\\	
					%&
					%\le 
					%C h^{\frac{p}{2}} + C h^{-1} \ C^Z_p h
					%\\
					%& 
					\leq   
					C \big( h^{\frac{p}{2}} + 1\big)
			\end{align*}
		%It then follows, 
		Since $h \le T$ there exists $C$, independent on $h$, such that
			%\begin{align*}
				$\sup_{0 \le i \le N-1} \bE\big[ \abs{{\widehat{Z}}_i}^{p} \big] \le C$.
			%\end{align*}		
		
		The second estimate for $(\widehat{Z}_i)_i$ can be obtained either directly from the above pivotal estimate or from the latest estimate: since $p \ge 1$, 
			\begin{align*}
				\bE\bigg[ \sum_{i=0}^{N-1}  \big( \abs{\widehat{Z}_i}^2 h\big)^{p} \bigg] 
				\le 
				\sup_{i=0, \ldots, N-1}  \bE\big[ \abs{\widehat{Z}_i}^{2p} \big] \ T \ h^{p-1} \le C.
			\end{align*}
Following the arguments just used to prove the estimate for $(\widehat{Z}_i)_i$ it is rather straightforward to prove the remaining estimate for $\widehat{Y}_i$ and hence we omit it.
	\end{proof}

%%%%%%%%%%%%%%%%%%%%%%%%%%%%%%%%%%%%%%%%%%%%%%%%%%%%%%%%%%%%%%%%%%%%%%%%%%%%%%%%%%%%%%%%%%%%%%%%%%%%%%%%%%%%%%%%%%%%%%%%%%%%%%%%%%%%%%%%%%%%%%%%%
%%%%%%%%%%%%%%%%%%%%%%%%%%%%%%%%%%%%%%%%%%%%%%%%%%%%%%%%%%%%%%%%%%%%%%%%%%%%%%%%%%%%%%%%%%%%%%%%%%%%%%%%%%%%%%%%%%%%%%%%%%%%%%%%%%%%%%%%%%%%%%%%%
%%%%%%%%%%%%%%%%%%%%%%%%%%%%%%%%%%%%%%%%%%%%%%%%%%%%%%%%%%%%%%%%%%%%%%%%%%%%%%%%%%%%%%%%%%%%%%%%%%%%%%%%%%%%%%%%%%%%%%%%%%%%%%%%%%%%%%%%%%%%%%%%%
%%%%%%%%%%%%%%%%%%%%%%%%%%%%%%%%%%%%%%%%%%%%%%%%%%%%%%%%%%%%%%%%%%%%%%%%%%%%%%%%%%%%%%%%%%%%%%%%%%%%%%%%%%%%%%%%%%%%%%%%%%%%%%%%%%%%%%%%%%%%%%%%%
%%%%%%%%%%%%%%%%%%%%%%%%%%%%%%%%%%%%%%%%%%%%%%%%%%%%%%%%%%%%%%%%%%%%%%%%%%%%%%%%%%%%%%%%%%%%%%%%%%%%%%%%%%%%%%%%%%%%%%%%%%%%%%%%%%%%%%%%%%%%%%%%% 
\section{Convergence of the schemes}
\label{section---convergence.of.scheme}

In this section we prove that scheme \eqref{equation---reference--definition.of.the.scheme} converges to the solution of the BSDE, as claimed by Theorem \ref{theorem---main--convergence.of.tamed.scheme}. Following the outline of proof described in Section \ref{sec:PC}, we prove that the scheme is almost-stable and estimate the sum of the discretization errors.

%////////////////////////////////////////////////////////////////////////////////////////////////////////////////////////////////////////////////////////////////////////////////////////
%////////////////////////////////////////////////////////////////////////////////////////////////////////////////////////////////////////////////////////////////////////////////////////
\subsection{Stability estimate of the scheme} 

We consider one step of the scheme with inputs $Y_\tip$ and $Y_\ip$, for which the outputs are respectively $(\widehat{Y}_i,\widehat{Z}_i)$ and $(Y_i,Z_i)$ 
(recall \eqref{eq:SchemeOutputWithTrueInput-Y}-\eqref{eq:SchemeOutputWithTrueInput-Z} and \eqref{equation---reference--definition.of.the.scheme}).
We denote by $\delta x$ the difference $\widehat{x} - x$ for a generic quantity $x$, $\widehat{x}$ being the counterpart of $x$ when the input is $Y_\tip$. 
We consequently have
	\begin{align*}
		\left\{\begin{aligned}
			&\delta Y_i = \delta Y_\ip + \big\{ f^h(t_i,Y_\tip,\widehat{Z}_i) - f^h(t_i,Y_\ip,Z_i) \big\} h - \delta \Delta M_\ip															\\
			&\delta Z_i = \bE_i\Big[ \Big( \delta Y_\ip + (1-\theta')\big\{ f^h(t_i,Y_\tip,0) - f^h(t_i,Y_\ip,0) \big\} \Big) H_\ip^* \Big] .
		\end{aligned}\right.
	\end{align*}
From Lemma \ref{lem:mrt} we recall that $\delta \Delta M_\ip$ has the decomposition $\delta \Delta M_\ip = \delta \zeta_i \Lambda^{-1} H_\ip h + \delta \Delta N_\ip$, 
and we have
	\begin{align*}
		\delta \zeta_i = \bE_i\Big[ \Big( \delta Y_\ip + \big\{ f^h(t_i,Y_\tip,\widehat{Z}_i)-f^h(t_i,Y_\ip,Z_i) \big\} h \Big) H_\ip^* \Big] .
	\end{align*}

	\begin{lemma}		\label{proposition---stability.of.the.tamed.scheme.NEWVERSION}
		Under \htmon, \htreg and \htregY, there exist $c, C \ge 0$ such that, without any restriction on the time-step (aside from $0<h \le T$), we have
			\begin{align*}
				\abs{\delta Y_i}^2 + \frac{1}{4} \abs{\delta Z_i}^2  d h + \bE_i[\abs{\delta \Delta N_\ip}^2] 
					\le (1+c \, h) \bE_i[\abs{\delta Y_\ip}^2] + C \cI_i
			\end{align*}
		where 
			\begin{align*}
				\cI_i &= \bE_i\Big[ \Rmon(\ti,Y_\tip,Y_\ip,\widehat{Z}_i) \Big] h 																																	\\
								& \qquad + \bE_i\Big[ \RregY(\ti,Y_\tip,Y_\ip,\widehat{Z}_i)^2 \Big] h^2 + \bE_i\Big[ \RregY(\ti,Y_\tip,Y_\ip,0)^2 \Big] h^2 						\\
					&= A_i h + B_i h^2 + B^0_i h^2 .
			\end{align*}		
	\end{lemma}

	\begin{proof}
		The proof presented here works with the extra assumption that $L^h_z>0$, while \htreg assumes $L^h_z\geq 0$. 
		The case of $L^h_z=0$ is easy to derive and does not require the constant $\alpha$ below. 
		Overall, the estimations here are very similar to those for the one-step size estimate (see Proposition \ref{proposition---one-step.size.estimate}). 
		Squaring and taking conditional expectation, we have
			\begin{align*}
				\abs{\delta Y_i}^2 + \bE_i\big[\abs{\delta \Delta M_\ip}^2\big] 
					&= \bE_i\Big[ \abs{\delta Y_\ip}^2 + 2 \scalar{ \delta Y_\ip } { \big\{ f^h(t_i,Y_\tip,\widehat{Z}_i) - f^h(t_i,Y_\ip,Z_i) \big\}h }											\\
					&\hspace{5cm}	+ \abs{ f^h(t_i,Y_\tip,\widehat{Z}_i) - f^h(t_i,Y_\ip,Z_i) }^2 h^2 \Big] . 
			\end{align*}
		As usual, we add-and-subtract $f^h(t_i,Y_\ip,\widehat{Z}_i)$ and then make use of \htmon and \htreg to estimate the 2nd term on the RHS and obtain 
			\begin{align*}
				\abs{\delta Y_i}^2 + \bE_i\big[\abs{\delta \Delta M_\ip}^2\big]
					&\le  \Big( 1 + 2(M^h_y + \alpha) h \Big) \bE_i[\abs{\delta Y_\ip}^2] + \Ai h 																											\\
					&\qquad	+ \frac{(L^h_z)^2}{2\alpha} \abs{\delta Z_i}^2 h + \bE_i\Big[\abs{ f^h(t_i,Y_\tip,\widehat{Z}_i) - f^h(t_i,Y_\ip,Z_i) }^2 \Big] h^2 .
			\end{align*}			
		Due to the orthogonality of $H_\ip h$ and $\delta \Delta N_\ip$ and using a Young inequality, we have
			\begin{align*}
				\bE_i\big[ \abs{\delta \Delta M_\ip}^2 \big] 
					&= \abs{\delta \zeta_i}^2 \Lambda^{-2} \bE_i[\abs{H_\ip h}^2] + \bE_i[\abs{\delta \Delta N_\ip}^2]																		\\
					&= \abs{\delta Z_i - \delta D_i}^2 \Lambda^{-1} d h + \bE_i[\abs{\delta \Delta N_\ip}^2]																						\\
					&\ge \half \abs{\delta Z_i}^2 \Lambda^{-1} d h - \abs{\delta D_i}^2 \Lambda^{-1} d h + \bE_i[\abs{\delta \Delta N_\ip}^2] .
			\end{align*}		
		So, the inequality for $\abs{\delta Y_i}^2 + \bE_i\big[\abs{\delta \Delta M_\ip}^2\big]$ becomes, since $1 \le \Lambda^{-1}$,
			\begin{align*} 
				\abs{\delta Y_i}^2 &+ \Big( \half - \frac{(L^h_z)^2}{2\alpha d} \Big) \abs{\delta Z_i}^2 d h + \bE_i[\abs{\delta \Delta N_\ip}^2]
				\\
					&
					\le \Big( 1 + 2(M^h_y + \alpha) h \Big) \bE_i[\abs{\delta Y_\ip}^2] + \Ai h 
					\\
					& \hspace{4cm}+ \abs{\delta D_i}^2 \Lambda^{-1} d h												
					+ \bE_i\Big[\abs{ f^h(t_i,Y_\tip,\widehat{Z}_i) - f^h(t_i,Y_\ip,Z_i) }^2 \Big] h^2 .
			\end{align*}
		
		We now focus on $\delta D_i$.
			\begin{align*}
				-\delta D_i &=\delta \zeta_i - \delta Z_i 																																										\\
					&= \bE_i\Big[ \Big( \{ f^h(t_i,Y_\tip,\widehat{Z}_i) - f^h(t_i,Y_\ip,Z_i) \} - (1-\theta')\{ f^h(t_i,Y_\tip,0) - f^h(t_i,Y_\ip,0) \} \Big) h H_\ip^* \Big] .
			\end{align*}
		Hence, using the Cauchy--Schwartz inequality, $\bE_i[ \abs{H_\ip}^2 ] = \Lambda d h^{-1}$ and $\Lambda \le 1$, we have
			\begin{align*}
			&
				\abs{\delta D_i}^2 \Lambda^{-1} d h
				\\
				&\quad 
					\le \Lambda^{-1} d h \bE_i\big[ \abs{H_\ip}^2 \big]
					\\
			&\qquad\quad \times \bE_i\Big[ \Abs{\big( f^h(t_i,Y_\tip,\widehat{Z}_i) - f^h(t_i,Y_\ip,Z_i) \big) - (1-\theta') \big( f^h(t_i,Y_\tip,0) - f^h(t_i,Y_\ip,0) \big) }^2 \Big] h^2
			\\
					&\quad \le 2d^2  \bE_i\Big[ \abs{ f^h(t_i,Y_\tip,\widehat{Z}_i) - f^h(t_i,Y_\ip,Z_i) }^2 \Big] h^2
					\\
					&\hspace{4.4cm} + 2d^2 (1-\theta')^2 \bE_i\Big[ \abs{  f^h(t_i,Y_\tip,0) - f^h(t_i,Y_\ip,0) }^2 \Big] h^2 .
			\end{align*}	
		Reinjecting this in the previous estimate leads to 
			\begin{align*}
				\abs{\delta Y_i}^2 
				&+ \Big( \half - \frac{(L^h_z)^2}{2\alpha d} \Big) \abs{\delta Z_i}^2 d h + \bE_i[\abs{\delta \Delta N_\ip}^2]
				\\
					&\le \Big( 1 + 2(M^h_y + \alpha) h \Big) \bE_i[\abs{\delta Y_\ip}^2] + \Ai h
					\\
							&\hspace{2.0cm} + (1+2d^2) \bE_i\Big[ \abs{ f^h(t_i,Y_\tip,\widehat{Z}_i) - f^h(t_i,Y_\ip,Z_i) }^2 \Big] h^2
							\\
							&\hspace{3.5cm} +  2 d^2 (1-\theta')^2 \bE_i\Big[ \abs{ f^h(t_i,Y_\tip,0) - f^h(t_i,Y_\ip,0) }^2 \Big] h^2 .
			\end{align*} 
		
		We now use \htregY and \htreg to estimate
			\begin{align*} 
				\abs{ f^h(t_i,Y_\tip,\widehat{Z}_i) &- f^h(t_i,Y_\ip,Z_i) }^2  																																													\\
					&\le 4 (L^h_y)^2 \abs{\delta Y_\ip}^2  + 4 \RregY(\ti,Y_\tip,Y_\ip,\widehat{Z}_i)^2 + 2 (L^h_z)^2 \abs{\delta Z_i}^2  																			\\
				\abs{f^h(t_i,Y_\tip,0) & - f^h(t_i,Y_\ip,0) }^2 																																																		\\
					&\le 2(L^h_y)^2 \abs{\delta Y_\ip}^2 + 2 \RregY(\ti,Y_\tip,Y_\ip,0)^2 .
			\end{align*}
		This leads to 
			\begin{align*}
				\abs{\delta Y_i}^2  &+ \Big( \half - \frac{(L^h_z)^2}{2\alpha d} \Big) \abs{\delta Z_i}^2 d h + \bE_i[\abs{\delta \Delta N_\ip}^2]																\\
					&\le \bigg( 1 + \Big[ 2(M^h_y + \alpha) + 4 (1+2d^2) (L^h_y)^2 h + 4 d^2 (1-\theta')^2 (L^h_y)^2 h \Big] h \bigg) \bE_i[\,\abs{\delta Y_\ip}^2] 								\\
						&\qquad	+ \Ai h  + 4(1+2d^2)\Bi h^2 + 4d^2(1-\theta')^2\Biz h^2 + 2 (1+2d^2) (L^h_z)^2 \bE_i\big[\, \abs{\delta Z_i}^2 h \big] h .
			\end{align*}		
		All that remains is to estimate the term with $\abs{\delta Z_i}^2 h$ on the RHS. We have, since $\Lambda \le 1$,
			\begin{align*}
				\abs{\delta Z_i}^2 h 
					&= \bE_i\Big[ \Big( \delta Y_\ip + (1-\theta')\big\{ f^h(t_i,Y_\tip,0) - f^h(t_i,Y_\ip,0) \big\} h \Big) 	 H_\ip^* \Big]^2 h																				\\	
					&\le h^{-1} \bE_i\Big[ \abs{H_\ip h}^2 \Big] \bE_i\Big[ \Abs{ \delta Y_\ip + (1-\theta')\big\{ f^h(t_i,Y_\tip,0) - f^h(t_i,Y_\ip,0) \big\} h }^2 \Big]											\\
					&\le \Lambda d \times 2\ \bE_i\Big[ \abs{\delta Y_\ip}^2 + (1-\theta')^2 \abs{ f^h(t_i,Y_\tip,0) - f^h(t_i,Y_\ip,0) }^2 h^2 \Big]																	\\
					&\le  2 d \Big( 1 +  2 (1-\theta')^2 (L^h_y)^2 h^2 \Big) \bE_i[\delta Y_\ip^2] + 4 d (1-\theta')^2  \Biz h^{2} .
			\end{align*}
		Define the quantities
			\begin{align*}
			\nonumber 
				c^h 
				&
				:= 2(M^h_y + \alpha)  + 4 (1+2d^2) (L^h_y)^2 h + 4 d^2 (1-\theta')^2 (L^h_y)^2 h
				\\
				\nonumber
				& \hspace{4cm} 
				    + 4 (1+2d^2) (L^h_z)^2  d \Big( 1 +  2(1-\theta')^2 (L^h_y)^2 h^2 \Big)
			 \\
		   %\label{eq:Chconstant-theImportantOne}
				\text{and}\quad  
				C^h & := 1 + 4(1+2d^2) + 4d(1-\theta')^2 \big[ d+ 2 (1+2d^2) (L^h_z)^2 h\big] .
			\end{align*}
		Then, by plugging in the estimates for $\abs{\delta Z_i}^2 h$ and re-organizing the terms we have 
			\begin{align*}
				\abs{\delta Y_i}^2 + \Big( \half - \frac{(L^h_z)^2}{2\alpha d} \Big) \abs{\delta Z_i}^2 d h &+ \bE_i[\abs{\delta \Delta N_\ip}^2]																\\
					&\leq (1+c^h \, h) \bE_i[\abs{\delta Y_\ip}^2] + C^h \Big( \Ai h + \Bi h^2 + \Biz h^2\Big) .
			\end{align*}		
		To complete the proof, first, we choose $\alpha = \frac{2 (L^h_z)^2}{d}$.
		Second, we note that with this choice of $\alpha$, and the assumptions on the constants made in \htregY, \htreg and \htmon, there exist $c, C \ge 0$ such that $c^h \le c$ and $C^h \le C$
		for all $h>0$ (provided $h \le T$).
	\end{proof}

Remark \ref{rem:ontheconstChch} applies to the constants $c^h$ and $C^h$.  
We can now prove that the scheme is almost-stable. We recall that $\mu$ is defined in subsection \ref{sec:PC}.

	\begin{proposition}		\label{proposition---stability.of.the.tamed.scheme.NEWVERSION.FINAL.DEFINITIVE.FOR.REAL}
		Under \htmon, \htreg and \htregY, there exist $c, C \ge 0$ such that for all $N \ge 1$ we have
			\begin{align*}
				\bE\big[ \abs{\delta Y_i}^2 \big] + \frac{1}{4} \bE\big[ \abs{\delta Z_i}^2  d h \big]
						\le (1+c \, h) \bE\big[\,\abs{\delta Y_\ip}^2\big] + C h^{\mu+1} .
			\end{align*}
	\end{proposition}

	\begin{proof}
		The first and easy step is to take expectation in the estimate from Lemma \ref{proposition---stability.of.the.tamed.scheme.NEWVERSION}. 
		One must then estimate the imperfection terms in the $\bE[\cI_i]$.
		We treat each case separately and we recall that $C$ is a constant whose value can change from line to line.

		\paragraph*{}		
		\textit{Case 1.}
			We assume for convenience that $\RregY$ and $\Rmon$ satisfy \htfcvg.1 with the same constants.
			In practice, for the multiplicative tamings (see Appendix \ref{appendix:VerifyingExamples}), we find $q=2m$ and $p=2$ for $\RregY$, and $q=4m$ and $p=2$ for $\Rmon$.
			But $q$ does not affect the rates. Alternatively we could argue that $\abs{y}^q \le 1 + \abs{y}^{2q}$ so both cases fit with $q=4m$.
			Let us start with $\RregY$.
				\begin{align*}
					B_i h^2 
					&
					= 
					\bE\Big[ \abs{ \RregY  (\ti,Y_\tip,Y_\ip,\widehat{Z}_i)  }^2 \Big]  h^2
					\\
						&
						\le 
						C \bE \Big[ \big( 1 + \abs{Y_\tip}^q + \abs{Y_\ip}^q + \abs{\widehat{Z}_i}^p \big)^2 \Big] h^{2\alpha} \ h^2
						\\
						&
						\le 
						C \Big( 1 + \bE[\abs{Y_\tip}^{2q}] + \bE[\abs{Y_\ip}^{2q}]  + \bE[\abs{\widehat{Z}_i}^{2p}] \Big) h^{2\alpha+2}
						%\\
						%&
						\le C h^{2\alpha+2} . 
				\end{align*}
			Here, we have used moment bounds from Theorem \ref{theorem---path.regularity} and Propositions \ref{proposition---moment.estimate} and \ref{lemm:integrabilityWidehatZi}.
			The case of $B^0_i h^2$ is similar.
			Define consequently $\mu_1 = -1+2\alpha+2 = 2\alpha+1$.
			
			Let us now handle $\Rmon$.
				\begin{align*}
					A_i h = \bE\Big[ { \Rmon (\ti,Y_\tip,Y_\ip,\widehat{Z}_i) } \Big] h		 																																							
						&\le C \bE \Big[ 1 + \abs{Y_\tip}^q + \abs{Y_\ip}^q + \abs{\widehat{Z}_i}^p \Big] h^{\alpha} \ h																										\\
						& =  C \Big( 1 + \bE[\abs{Y_\tip}^q] + \bE[\abs{Y_\ip}^q]  + \bE[\abs{\widehat{Z}_i}^{p}] \Big) h^{\alpha+1}
						\\
						&\le C h^{\alpha+1}.
				\end{align*}
			Define consequently $\mu_2 = -1+\alpha+1=\alpha$. The proof finishes by taking  $\mu = \min( \mu_1, \mu_2 ) = \alpha$.

		\paragraph*{}
		\textit{Case 2.}
			We assume again that $\RregY$ and $\Rmon$ satisfy \htfcvg.2 with the same constants.
			In practice, for the outer taming by projection, we find $q=2m$ and $p=0$ for $\RregY$, and $q=2m$ and $p=2$ for $\Rmon$. 
			Let us start with $\RregY$. 
			Using $\1_{A \cup B} \le \1_A + \1_B$ and the Cauchy-Schwartz inequality,
				\begin{align*}
					\bE\Big[ & \abs{ \RregY  (\ti,Y_\tip,Y_\ip,\widehat{Z}_i) }^2 \Big]  h^2
					\\
						&\le 
						C \bE \Big[ \big( 1 + \abs{Y_\tip}^{q} + \abs{Y_\ip}^{q} + \abs{\widehat{Z}_i}^{p} \big)^2 
										\big( \1_{\{ \abs{f(\ti,Y_\tip,\widehat{Z}_i)} > r(h) \}} + \1_{\{ \abs{f(\ti,Y_i,\widehat{Z}_i)} > r(h) \}} \big)^2 \Big] \ h^2
										\\
						&
						\le 
						C \bE \Big[ \big( 1 + \abs{Y_\tip}^{2q} + \abs{Y_\ip}^{2q} + \abs{\widehat{Z}_i}^{2p} \big) 
										\big( \1_{\{ \abs{f(\ti,Y_\tip,\widehat{Z}_i)} > r(h) \}} + \1_{\{ \abs{f(\ti,Y_\ip,\widehat{Z}_i)} > r(h) \}} \big) \Big] \ h^2
										\\
						&
						\le 
						C \bE \Big[ 1 + \abs{Y_\tip}^{4q} + \abs{Y_\ip}^{4q} + \abs{\widehat{Z}_i}^{4p} \Big]^\half	
						\Big( \bE\big[ \1_{\{ \abs{f(\ti,Y_\tip,\widehat{Z}_i)} > r(h) \}} \big] 
										+ \bE\big[ \1_{\{ \abs{f(\ti,Y_\ip,\widehat{Z}_i)} > r(h) \}} \big]\Big)^\half \ h^2	.	
				\end{align*}
			Now, we systematically use 
			the moment bounds from Theorem \ref{theorem---path.regularity} and Propositions \ref{proposition---moment.estimate} and \ref{lemm:integrabilityWidehatZi}, 
			as well as the Markov inequality with a power $l \ge 1$ yet to be determined, 
			to have
				\begin{align*}
					\bE\Big[ \abs{ \RregY & (\ti,Y_\tip,Y_\ip,\widehat{Z}_i) }^2 \Big]  h^2																																										\\
						&\le C \Big( \bE\big[ \abs{f(\ti,Y_\tip,\widehat{Z}_i)}^l \big] + \bE\big[ \abs{f(\ti,Y_\ip,\widehat{Z}_i)}^l \big]\Big)^\half r(h)^{-\frac{l}{2}} \ h^2														\\
						&\le C \Big( 1 + \bE[\abs{Y_\tip}^{lm}] + \bE[\abs{Y_\ip}^{lm}] + \bE[\abs{\widehat{Z}_i}^{l}] \big]\Big)^\half h^{\frac{\beta l}{2}} \ h^2
						%\\
						%&
						\le C h^{\frac{\beta l}{2} +2 } 	.
				\end{align*}
			Define $\mu_1 = -1+\frac{\beta l}{2} + 2 = \frac{\beta l}{2} + 1$.
			
			Let us now handle $\Rmon$.
			We use the Cauchy-Schwartz inequality, the inequality $(\sum_{i=1}^n a_i)^k \le n^{k-1} \sum_{i=1}^n a_i^k$, 
			the Markov inequality with a power $l \ge 1$ yet to be determined, 
			and the moment bounds from Theorem \ref{theorem---path.regularity} and Propositions \ref{proposition---moment.estimate} and \ref{lemm:integrabilityWidehatZi}.
				\begin{align*}
					\bE\Big[ \abs{ \Rmon & (\ti,Y_\tip,Y_\ip,\widehat{Z}_i) } \Big] h		 																																										\\
					&\le C \bE \Big[ \big( 1 + \abs{Y_\tip}^q + \abs{Y_\ip}^q + \abs{\widehat{Z}_i}^p \big) 
												\big( 1_{\{ \abs{f(t_i,Y_\tip,\widehat{Z}_i)} > r(h) \}} + 1_{\{ \abs{f(t_i,Y_i,\widehat{Z}_i)} > r(h) \}} \big) \Big] \ h																			\\
					&\le C \bE \Big[ 1 + \abs{Y_\tip}^{2q} + \abs{Y_\ip}^{2q} + \abs{\widehat{Z}_i}^{2p} \Big]^\half	
												\bE\Big[ 1_{\{ \abs{f(t_i,Y_\tip,\widehat{Z}_i)} > r(h) \}} + 1_{\{ \abs{f(t_i,Y_\ip,\widehat{Z}_i)} > r(h) \}} \Big]^\half \ h																	\\
					&\le C \Big( \bE\big[ \abs{f(t_i,Y_\tip,\widehat{Z}_i)}^l \big] + \bE\big[ \abs{f(t_i,Y_\ip,\widehat{Z}_i)}^l \big]\Big)^\half r(h)^{-\frac{l}{2}} \ h .	
				\end{align*}
			Using \hfgrowth we therefore have
				\begin{align*}
					\bE\Big[ \abs{ \Rmon (\ti,Y_\tip,Y_\ip,\widehat{Z}_i) } \Big] h
						&\le C \Big( 1 + \bE[\abs{Y_\tip}^{lm} ] + \bE[\abs{Y_\ip}^{lm} ] + \bE[\abs{\widehat{Z}_i}^{l}] \Big)^\half h^{\frac{\beta l}{2} +1} 																	\\
						&\le C h^{\frac{\beta l}{2} +1}	.
				\end{align*}
			Define $\mu_2 = -1+\frac{\beta l}{2} + 1=\frac{\beta l}{2}$.

			We have the desired result with $\mu = \min( \mu_1, \mu_2) = \frac{\beta l}{2}$.
			Note that since $\beta > 0$, by taking $l$ big enough one can make the exponent of $h$ be as big as wanted.
			Naturally, the constant $C$ depends on the powers $q$, $p$ and $l$ eventually chosen.

		\paragraph*{}			
		\textit{Case 3.}
			We assume again that $\RregY$ and $\Rmon$ satisfy \htfcvg.3 with the same constants.
			In practice, for the inner taming by projection, we find $\RregY = 0$, and $q=m$ and $p=0$ for $\Rmon$. 
			Let us start with $\RregY$. 
			Using $1_{A \cup B} \le 1_A + 1_B$ and the Cauchy-Schwartz inequality,
				\begin{align*}
					\bE\Big[ \abs{ \RregY & (\ti,Y_\tip,Y_\ip,\widehat{Z}_i) }^2 \Big]  h^2																																											\\
					&\le C \bE \Big[ \big( 1 + \abs{Y_\tip}^{q} + \abs{Y_\ip}^{q} +  \abs{\widehat{Z}_i}^{p} \big)^2 \big( 1_{\{ \abs{Y_\tip} > r(h) \}} + 1_{\{ \abs{Y_i} > r(h) \}} \big)^2 \Big] \ h^2		\\
					&\le C \bE \Big[ \big( 1 + \abs{Y_\tip}^{2q} + \abs{Y_\ip}^{2q}  + \abs{\widehat{Z}_i}^{2p} \big) \big( 1_{\{ \abs{Y_\tip} > r(h) \}} + 1_{\{ \abs{Y_i} > r(h) \}} \big) \Big] \ h^2			\\
					&\le C \bE \Big[ 1 + \abs{Y_\tip}^{4q} + \abs{Y_\ip}^{4q} + \abs{\widehat{Z}_i}^{4p} \Big]^\half	
												\Big( \bE\big[ 1_{\{ \abs{Y_\tip} > r(h) \}} \big] + \bE\big[ 1_{\{ \abs{Y_\ip} > r(h) \}} \big]\Big)^\half \ h^2		.	
				\end{align*}
			Now, we systematically use 
			the moment bounds from Theorem \ref{theorem---path.regularity} and Propositions \ref{proposition---moment.estimate} and \ref{lemm:integrabilityWidehatZi}, 
			as well as the Markov inequality with a power $l \ge 1$ yet to be determined, 
			to have
				\begin{align*}
					\bE\Big[ \abs{ \RregY (\ti,Y_\tip,Y_\ip,\widehat{Z}_i) }^2 \Big]  h^2
						&\le C \Big( \bE\big[ \abs{Y_\tip}^l \big] + \bE\big[ \abs{Y_\ip}^l \big]\Big)^\half r(h)^{-\frac{l}{2}} \ h^2
						%\\
						%&
						\le C  h^{\frac{\gamma l}{2}} \ h^2		.
				\end{align*}
			Define $\mu_1 = -1+\frac{\gamma l}{2} + 2 = \frac{\gamma l}{2} + 1$.
			
			Let us now handle $\Rmon$.
			We use the Cauchy-Schwartz inequality, the inequality $(\sum_{i=1}^n a_i)^k \le n^{k-1} \sum_{i=1}^n a_i^k$, 
			the Markov inequality with a power $l \ge 1$ yet to be determined, 
			and the moment bounds from Theorem \ref{theorem---path.regularity} and Propositions \ref{proposition---moment.estimate} and \ref{lemm:integrabilityWidehatZi}.
				\begin{align*}
					\bE\Big[ \Rmon & (\ti,Y_\tip,Y_\ip,\widehat{Z}_i)  \Big] h
					\\
						&
						\le C \bE \Big[ \big( 1 + \abs{Y_\tip}^q + \abs{Y_\ip}^q + \abs{\widehat{Z}_i}^p \big) 
						\big( \1_{\{ \abs{Y_\tip} > r(h) \}} 
						    + \1_{\{ \abs{Y_\ip} > r(h) \}} \big) \Big] \ h				
						\\
						&
						\le 
						C \bE \Big[ 1 + \abs{Y_\tip}^{2q} + \abs{Y_\ip}^{2q} + \abs{\widehat{Z}_i}^{2p} \Big]^\half	
						\bE\Big[ \1_{\{ \abs{Y_\tip} > r(h) \}} + \1_{\{ \abs{Y_\ip} > r(h) \}} \Big]^\half \ h
												\\
						&\le C \Big( \bE\big[ \abs{Y_\tip}^l \big] + \bE\big[ \abs{Y_\ip}^l \big]\Big)^\half r(h)^{-\frac{l}{2}} \ h
						\\
						&\le C h^{\frac{\gamma l}{2} +1} 		.
				\end{align*}
			Define $\mu_2 = -1+\frac{\gamma l}{2} + 1=\frac{\gamma l}{2}$.

			We have the desired result with $\mu = \min( \mu_1, \mu_2 ) = \frac{\gamma l}{2}$.
			Note that since $\gamma > 0$, by taking $l$ big enough one can make the exponent of $h$ be as big as wanted.
			Naturally, the constant $C$ depends on the power $l$ eventually chosen.
	\end{proof}

%////////////////////////////////////////////////////////////////////////////////////////////////////////////////////////////////////////////////////////////////////////////////////////
%////////////////////////////////////////////////////////////////////////////////////////////////////////////////////////////////////////////////////////////////////////////////////////
\subsection{Time-discretization errors} 

We now turn to the estimation of the local errors, that is to say the error between the BSDE dynamics and the time-discretization scheme \eqref{equation---reference--definition.of.the.scheme} introduced over one time-step, and in particular their total sum.

	\begin{proposition}		
	\label{proposition---estimation.of.the.discretization.error}	
		Assume \htreg, \htfcvg and \hH. There exists a constant $C \ge 0$ such that, for all $N \ge 1$,
			\begin{align*}
					\sum_{i=0}^{N-1} \bE\big[\, \abs{\overline{Z}_\ti - \widehat{Z}_i}^2 \big] h \le C \ h
				\qquad\text{and}\qquad 
					\sum_{i=0}^{N-1} \bE\big[\, \abs{Y_\ti - \widehat{Y}_i}^2 \big] \le C \ h^2 \ .
			\end{align*}
	\end{proposition}

The proof of this estimate is split in two parts. The first is the estimations for the $Z$-component, while the second those for the $Y$-component. 

	\begin{proof}[Proof of the estimate for the $Z$-component in Proposition \ref{proposition---estimation.of.the.discretization.error}] 
		First recall that from the martingale increment property of $H_\ip$ we have
			\begin{align*}
				\bE_i\left[ \int_\ti^\tip Z_u \ud W_u \ H_\ip^* \right] 
					= \bE_i\left[ \Big( Y_\tip + \int_\ti^\tip f(u,Y_u,Z_u) \udu \Big) \ H_\ip^* \right] \ .
			\end{align*}
		We write 
			\begin{align*}
				\overline{Z}_\ti - \widehat{Z}_i 
					&= \bE_i\Bigg[ \int_\ti^\tip Z_u \ud W_u \ \frac{\Delta W_\tip^*}{h} \Bigg] - \bE_i\Bigg[ \Big(Y_\tip + (1-\theta') f^h(t_i,Y_\tip,0)h \Big) H_\ip^* \Bigg]							\\
					&= \bE_i\Bigg[ \int_\ti^\tip Z_u \ud W_u \ \frac{\Delta W_\tip^*}{h} \Bigg] - \bE_i\Bigg[ \int_\ti^\tip Z_u \ud W_u \ H_\ip^* \Bigg]													\\
							& \qquad + \bE_i\Bigg[ \Big( Y_\tip + \int_\ti^\tip f(u,Y_u,Z_u) \udu \Big) \ H_\ip^* \Bigg] 																											\\
							& \qquad \qquad - \bE_i\Bigg[ \Big(Y_\tip + (1-\theta') f^h(Y_\tip,0)h \Big) H_\ip^* \Bigg] .
			\end{align*}		
		Regrouping the terms yields
			\begin{align*}
				\overline{Z}_\ti - \widehat{Z}_i
					&= \bE_i\Bigg[ \int_\ti^\tip Z_u dW_u \bigg(\frac{\Delta W_\tip}{h} - H_\ip \bigg)^* \Bigg]
							   + \bE_i\Bigg[ \theta' \int_\ti^\tip f(u,Y_u,Z_u) \udu \ H_\ip^* \Bigg]
								\\
							&\hspace{3cm}	+ \bE_i\Bigg[ (1-\theta')\int_\ti^\tip f(u,Y_u,Z_u) - f^h(t_i,Y_\ip,0) \udu \ H_\ip^* \Bigg] .
			\end{align*}
		Further decomposing the last term leads finally to 
			\begin{align*}
				\overline{Z}_\ti - \widehat{Z}_i 			
					&= \bE_i\Bigg[ \int_\ti^\tip Z_u dW_u \bigg(\frac{\Delta W_\tip}{h} - H_\ip \bigg)^* \Bigg]
									+ \bE_i\Bigg[ \theta' \int_\ti^\tip f(u,Y_u,Z_u) \udu \ H_\ip^* \Bigg] 																																		\\
							&\qquad	+ (1-\theta') \bE_i\Bigg[ \int_\ti^\tip f(u,Y_u,Z_u) - f(u,Y_\tip,Z_u) \udu \ H_\ip^* \Bigg] 																									\\
							&\qquad	+ (1-\theta') \bE_i\Bigg[ \int_\ti^\tip f(u,Y_\tip,Z_u) - f(u,Y_\tip,0) \udu \ H_\ip^* \Bigg] 																									\\
							&\qquad	+ (1-\theta') \bE_i\Bigg[ \int_\ti^\tip f(u,Y_\tip,0) - f(t_i,Y_\tip,0) \udu \ H_\ip^* \Bigg] 																									\\
							&\qquad	+ (1-\theta') \bE_i\Bigg[ \int_\ti^\tip f(t_i,Y_\tip,0) - f^h(t_i,Y_\ip,0) \udu \ H_\ip^* \Bigg]																									\\
					&= \cE_H + \cE_{\theta'} + \cE_{\mathrm{PR:Y}} + \cE_{Z} 
					+ \cE_{t} +\cE_{\mathrm{tamed.f}}.
			\end{align*}
		We now want to estimate (the expected square of) each of these terms $\cE_\cdot$.
			
		\paragraph*{}\textit{Estimation of $\cE_H$.}
		Using the Cauchy--Schwartz inequality and the It\^o isometry we have
			\begin{align*}
				\abs{\cE_H}^2 
					&\le d\, \bE_i\Bigg[ \int_\ti^\tip \abs{Z_u}^2 \udu \Bigg] \bE_i\Bigg[ \Abs{\frac{\Delta W_\tip}{h} - H_\ip }^2 \Bigg] 																			\\
					&=   d\, \bE_i\Bigg[ \int_\ti^\tip \abs{Z_u}^2 \udu \Bigg] \bE\Bigg[ \Abs{\frac{\Delta W_\tip}{h} - H_\ip }^2 \Bigg] ,
			\end{align*}
		since $\frac{\Delta W_\tip}{h} - H_\ip$ is independent from $\F_i$.
		Hence, taking expecations,
			\begin{align*}
				\bE\Big[\abs{\cE_H}^2\Big]
					&\le d\, \bE\Bigg[ \int_\ti^\tip \abs{Z_u}^2 \udu \Bigg] \bE\Bigg[ \Abs{\frac{\Delta W_\tip}{h} - H_\ip }^2 \Bigg] .
			\end{align*}
			
		\paragraph*{}\textit{Estimation of $\cE_{\theta'}$.}
		Using the Cauchy--Schwartz inequality and \hH we have
			\begin{align*}
				\abs{\cE_{\theta'}}^2 
					&\le {\theta'}^2 \bE_i\Bigg[ \Abs{\int_\ti^\tip f(u,Y_u,Z_u) \udu }^2 \Bigg] \bE_i\Big[ \Abs{H_\ip}^2 \Big]
					\\
					&\le {\theta'}^2 \bE_i\bigg[ h \int_\ti^\tip \Abs{f(u,Y_u,Z_u)}^2 \udu \bigg] \frac{\Lambda d}{h} .
			\end{align*}		
		Hence, taking expectations, since $\Lambda \le 1$, $\bE\Big[\abs{\cE_{\theta'}}^2\Big] \le {\theta'}^2 d \ \bE\big[ \int_\ti^\tip \Abs{f(u,Y_u,Z_u)}^2 \udu \big]$ .
	
		\paragraph*{}\textit{Estimation of $\cE_{\mathrm{PR:Y}}$.}
		Using the Cauchy--Schwartz inequality and \hH as above, and then the $Y$-regularity \hfregY,
			\begin{align*}
				\abs{\cE_{\mathrm{PR:Y}}}^2 
					&\le (1-\theta')^2 \bE_i\Bigg[ h \int_\ti^\tip \Abs{f(u,Y_u,Z_u) - f(u,Y_\tip,Z_u)}^2 \udu \Bigg] \frac{\Lambda d}{h}																								\\
					&\le (1-\theta')^2 \Lambda d \, L_y^2\  \bE_i\Bigg[ \int_\ti^\tip (1+\abs{Y_u}^{m-1}+\abs{Y_\tip}^{m-1})^2 \Abs{Y_u - Y_\tip}^2 \udu \Bigg] .
			\end{align*}		
		When taking the expectation we obtain, using the Cauchy--Schwartz inequality, $\Lambda \le 1$, and that $Y\in\cS^p$ for any $p\geq 2$
			\begin{align}
			\nonumber
				%&
				\bE\Big[\abs{\cE_{\mathrm{PR:Y}}}^2\Big]
				%\\
				%&\le (1-\theta')^2 d \, L_y^2 \int_\ti^\tip \bE \Big[ (1+\abs{Y_u}^{m-1}+\abs{Y_\tip}^{m-1})^2 \Abs{Y_u - Y_\tip}^2 \Big]\udu
				%\\
				%&\le (1-\theta')^2 d \, L_y^2 \int_\ti^\tip \bE \Big[ (1+\abs{Y_u}^{m-1}+\abs{Y_\tip}^{m-1})^4\Big]^\half \bE\Big[ \Abs{Y_u - Y_\tip}^4 \Big]^\half du 										\\
				%&\le 3^{3/2} (1-\theta')^2 d \, L_y^2 
									%\int_\ti^\tip \bE \Big[ 1+\abs{Y_u}^{4(m-1)}+\abs{Y_\tip}^{4(m-1)}\Big]^\half \bE\Big[ \Abs{Y_u - Y_\tip}^4 \Big]^\half \udu
				&\le 3^{3/2} (1-\theta')^2 d \, L_y^2
				\big(1 + 2 \|Y\|_{\cS^{4(m-1)}}^{4(m-1)}\big)^{\frac12}
									\int_\ti^\tip \bE\Big[ \Abs{Y_u - Y_\tip}^4 \Big]^\half \udu
				\\
				\label{errorPR:Y}
				&
			  \le 3^{3/2} (1-\theta')^2 d \, L_y^2 
				%\big(1 + 2 \|Y\|_{\cS^{4(m-1)}}^{4(m-1)}\big)^{\frac12}
				C_Y  \ h \  \big(\mathrm{REG}_{Y,4}(h)\big)^\half
				%\\
				%&
			  %\le 3^{3/2} (1-\theta')^2 d \, L_y^2 
				%\big(1 + 2 \|Y\|_{\cS^{4(m-1)}}^{4(m-1)}\big)^{\frac12}
				%\times h \times \big( \mathrm{REG}_{Y,4}(h)\big)^\half
			\end{align}
	With the term $\mathrm{REG}_{Y,4}(h)$ following from the path-regularity Theorem \ref{theorem---path.regularity} and from Theorem \ref{theo:BasicExistandUniqTheo} it holds that $1 + 2 \|Y\|_{\cS^{4(m-1)}}^{4(m-1)}\leq C_Y$ for some constant $C_Y>0$.
	
		\paragraph*{}\textit{Estimation of $\cE_{Z}$.}
		Using the Cauchy--Schwartz inequality and \hH as above, and then the $Z$-regularity \hfreg, and $\Lambda \le 1$,
			\begin{align*}
				\abs{\cE_{Z}}^2 
					&\le (1-\theta')^2 \bE_i\Bigg[ h \int_\ti^\tip \Abs{f(u,Y_\tip,Z_u) - f(u,Y_\tip,0)}^2 \udu \Bigg] \frac{\Lambda d}{h}																								\\
					&\le (1-\theta')^2 d\, \Lambda \, L_z^2\  \bE_i\Bigg[ \int_\ti^\tip \Abs{Z_u}^2 \udu \Bigg] 
					\\
					& \qquad \qquad \Rightarrow \bE\big[\abs{\cE_{Z}}^2\big] \le (1-\theta')^2 d \, L_z^2\  \bE\bigg[ \int_\ti^\tip \Abs{Z_u}^2 \udu \bigg].
			\end{align*}		
		%When taking the expectation we obtain 
		%$\bE\big[\abs{\cE_{Z}}^2\big] \le (1-\theta')^2 d \, L_z^2\  \bE\big[ \int_\ti^\tip \Abs{Z_u}^2 \udu \big]$.
		
		\paragraph*{}\textit{Estimation of $\cE_{t}$.}
		Using the Cauchy--Schwartz inequality and \hH as above, and then the $t$-regularity \hfreg, and $\Lambda \le 1$,
			\begin{align*}
				\abs{\cE_{t}^2 }
					&\le (1-\theta')^2 \bE_i\Bigg[ h  \int_\ti^\tip \Abs{f(u,Y_\tip,0) - f(t_i,Y_\ip,0)}^2 \udu \Bigg] \frac{\Lambda d}{h}														\\
					&\le (1-\theta')^2 d L_t \frac{h^{2}}{2}  \le (1-\theta')^2 d L_t \, h^{2} .
			\end{align*}		
	
		\paragraph*{}\textit{Estimation of $\cE_{\mathrm{tamed.f}}$.}
		Using the Cauchy--Schwartz inequality and \hH as above,
			\begin{align*}
				\abs{\cE_{\mathrm{tamed.f}}}^2 
					&\le (1-\theta')^2 \bE_i\Bigg[ h \int_\ti^\tip \Abs{f^h(t_i,Y_\tip,0) - f^h(t_i,Y_\ip,0)}^2 \udu \Bigg] \frac{\Lambda d}{h}																								\\
					&\le (1-\theta')^2 \Lambda d h \ \bE_i\Big[ \Abs{(f-f^h)(Y_\tip,0)}^2 \Big] .
			\end{align*}		
		Hence, in expectation, 
		$\bE\big[\abs{\cE_{\mathrm{tamed.f}}}^2\big] \le (1-\theta')^2 d h \ \bE\big[ \Abs{(f-f^h)(t_i,Y_\tip,0)}^2 \big]$.
		
		\paragraph*{}\textit{Gathering the estimates.}	
		We finally obtain 
			\begin{align*}
				&
				\sum_{i=0}^{N-1} \bE\Big[ \abs{\overline{Z}_\ti - \widehat{Z}_i}^2 \Big] h
				\\
				&
				\le 
				6 d h\, \bE\Bigg[ \int_0^T \abs{Z_u}^2 \udu \Bigg] \max_{i=0 \ldots N-1} \bE\Bigg[ \Abs{\frac{\Delta W_\tip}{h} - H_\ip }^2 \Bigg]	
				+ 6 {\theta'}^2 d h \ \bE\bigg[ \int_0^T \Abs{f(u,Y_u,Z_u)}^2 \udu \bigg]
				\\
				& 
				\qquad + 6\cdot 3^{3/2} (1-\theta')^2 d h \, L_y^2  C_Y  \ h \  \big(\mathrm{REG}_{Y,4}(h)\big)^\half
				+ 6 (1-\theta')^2 d h \, L_z^2\  \bE\Bigg[ \int_0^T \Abs{Z_u}^2 du \Bigg]
				%
				 %\sum_{i=0}^{N-1} \int_\ti^\tip \bE \Big[ 1+\abs{Y_u}^{4(m-1)}+\abs{Y_\tip}^{4(m-1)}\Big]^\half \udu \ \mathrm{REG}_{Y,4}(h)^\half
				\\
				& 
				\qquad 	 
				+ 6 (1-\theta')^2 d L_t \, h^{2} 
				+ 6 (1-\theta')^2 d h^2 \sum_{i=0}^{N-1} \bE\Big[ \Abs{(f-f^h)(\ti,Y_\tip,0)}^2 \Big].
			\end{align*}
		Here, $\mathrm{REG}_{Y,4}(h) = \sup_{\abs{s-t}\le h} \bE\big[ \Abs{Y_s - Y_t}^4 \big]$.
		From the path-regularity Theorem \ref{theorem---path.regularity}, there exists $C_{\mathrm{PR}}$ such that $\mathrm{REG}_{Y,4}(h)^\half \le C_{\mathrm{PR}} \, h$.
		Consequently, there exists a constant $C$ (independent of $N$) such that
			\begin{align*}
				\sum_{i=0}^{N-1} \bE\Big[ \abs{\overline{Z}_\ti - \widehat{Z}_i}^2 \Big] h
					&\le 			C h \sum_{i=0}^{N-1} \bE\Bigg[ \Abs{\frac{\Delta W_\tip}{h} - H_\ip }^2 \Bigg]
									+ C {\theta'}^2 \ h
									+ C L_y^2 \, h^2																																																								\\
					&\quad 	+ C (1-\theta')^2 L_z^2 \ h
									+ C (1-\theta')^2 L_t \ h^2 
									+ C h^2 \sum_{i=0}^{N-1} \bE\Big[ \Abs{(f-f^h)(\ti,Y_\tip,0)}^2 \Big] .
			\end{align*}
			The result then follows from the \hH.3 and Lemma \ref{lemma---vanishing.effect.of.taming.the.driver} below.		
	\end{proof}
	
	\begin{lemma}		
	\label{lemma---vanishing.effect.of.taming.the.driver}
		Under \htfcvg there exists a constant $C \ge 0$ such that, for all $N$,
			\begin{align*}
				\sum_{i=0}^{N-1} \bE\Big[ \abs{(f-f^h)(\ti,Y_\tip,0)}^2 \Big] \le C \quad\text{and}\quad \sum_{i=0}^{N-1} \bE\Big[ \abs{(f-f^h)(\ti,Y_\tip,\overline{Z}_\ti)}^2 \Big] \le C .
			\end{align*}		
	\end{lemma}
	
We postpone the proof of the lemma above to Appendix \ref{sec--proof.of.lemma---vanishing.effect.of.taming.the.driver} and proceed with the second part of the proof of Proposition \ref{proposition---estimation.of.the.discretization.error}.
		
	\begin{proof}[Proof of the estimate for the $Y$-component of Proposition \ref{proposition---estimation.of.the.discretization.error}] 
		We first decompose
			\begin{align*}
				& Y_\ti - \widehat{Y}_i 
					= \bE_i\bigg[ \int_\ti^\tip f(u,Y_u,Z_u) \udu - f^h(\ti,Y_\tip,\widehat{Z}_i) h \bigg]					 		\\
					&= \bE_i\bigg[ \int_\ti^\tip f(u,Y_u,Z_u) - f(u,Y_\tip,Z_u) \udu \bigg]																																									
							+ \bE_i\bigg[ \int_\ti^\tip f(u,Y_\tip,Z_u) - f(u,Y_\tip,\overline{Z}_\ti) \udu \bigg]			\\
					& \quad + \bE_i\bigg[ \int_\ti^\tip f(u,Y_\tip,\overline{Z}_\ti) - f(\ti,Y_\tip,\overline{Z}_\ti) \udu \bigg]			
								+ \bE_i\bigg[ f(\ti,Y_\tip,\overline{Z}_\ti) - f^h(\ti,Y_\tip,\overline{Z}_\ti) \bigg] h		\\				
					& \quad 	+ \bE_i\bigg[ f^h(\ti,Y_\tip,\overline{Z}_\ti) - f^h(\ti,Y_\tip,\widehat{Z}_i) \bigg] h				\\
					&= \cE_{\mathrm{PR:Y}} + \cE_{\mathrm{PR:Z}} + \cE_{t} + \cE_{\text{tamed.f}} + \cE_{\tau_i(Z)}
			\end{align*}	
		We now estimate (the expected square of) each of these terms.
		
		\paragraph*{}\textit{Estimation of $\cE_\mathrm{PR:Y}$.}
		Using the Cauchy--Schwartz inequality and \hfregY we have
			\begin{align*}
				\abs{\cE_\mathrm{PR:Y}}^2 
					&\le \bE_i\bigg[ h \int_\ti^\tip \abs{f(u,,Y_u,Z_u) - f(u,Y_\tip,Z_u)}^2 \udu \bigg]																	\\
					&\le h L_y^2 \ \bE_i\Bigg[ \int_\ti^\tip (1+\abs{Y_u}^{m-1}+\abs{Y_\tip}^{m-1})^2 \Abs{Y_u - Y_\tip}^2 \udu \Bigg] .
			\end{align*}		
		Hence, taking expectations and arguing as in \eqref{errorPR:Y} we have 
		%using the Cauchy--Schwartz inequality, we obtain
			\begin{align*}
				\bE\Big[\abs{\cE_{\mathrm{PR:Y}}}^2\Big]
					&
					\le 
					h L_y^2\ \times
					C_Y  \ h \  \big(\mathrm{REG}_{Y,4}(h)\big)^\half
					%\int_\ti^\tip \bE \Big[ (1+\abs{Y_u}^{m-1}+\abs{Y_\tip}^{m-1})^2 \Abs{Y_u - Y_\tip}^2 \Big] \udu
					%\\
					%&
					%\le 
					%3^{3/2} h L_y^2 \int_\ti^\tip \bE \Big[ 1+\abs{Y_u}^{4(m-1)}+\abs{Y_\tip}^{4(m-1)}\Big]^\half 
									%\bE\Big[ \Abs{Y_u - Y_\tip}^4 \Big]^\half \udu 
									.
			\end{align*}

		\paragraph*{}\textit{Estimation of $\cE_\mathrm{PR:Z}$.}
		Using Cauchy--Schwartz's inequality and \hfreg we have
			\begin{align*}
				\abs{\cE_\mathrm{PR:Z}}^2 
					\le \bE_i\bigg[ h \int_\ti^\tip \abs{f(u,Y_\tip,Z_u) - f(u,Y_\tip,\overline{Z}_\ti)}^2 \udu \bigg]
					\le h L_z^2 \ \bE_i\bigg[ \int_\ti^\tip \abs{Z_u - \overline{Z}_\ti}^2 \udu \bigg] .
			\end{align*}		
		Hence, taking expectations, we obtain 
		$\bE\big[\abs{\cE_{\mathrm{PR:Z}}}^2\big] \le h L_z^2 \ \bE\big[ \int_\ti^\tip \abs{Z_u - \overline{Z}_\ti}^2 \udu \big]$.

		\paragraph*{}\textit{Estimation of $\cE_{t}$.}
		Using Cauchy--Schwartz's inequality and \hfreg we have
			\begin{align*}
				\abs{\cE_{t}}^2 
					\le \bE_i\bigg[ h \int_\ti^\tip \abs{f(u,Y_\tip,\overline{Z}_\ti) - f(t_i,Y_\tip,\overline{Z}_\ti)}^2 \udu \bigg]
					\le L_t^2 \frac{h^3}{2}		
				\end{align*}		

		\paragraph*{}\textit{Estimation of $\cE_{\mathrm{tamed.f}}$.}
		Taking the square and using the Cauchy--Schwartz inequality,
			\begin{align*}
				\abs{\cE_{\mathrm{tamed.f}}}^2 
					&\le h^2 \ \bE_i\big[ \abs{f(\ti,Y_\tip,\overline{Z}_\ti) - f^h(\ti,Y_\tip,\overline{Z}_\ti)}^2 \big] .
			\end{align*}		
		Hence, taking expectations, we obtain $	\bE\big[\abs{\cE_{\mathrm{tamed.f}}}^2\big] \le h^2 \ \bE\big[ \abs{(f-f^h)(\ti,Y_\tip,\overline{Z}_\ti)}^2 \big]$.

		\paragraph*{}\textit{Estimation of $\cE_{\tau_i(Z)}$.}
		Taking the square, using the Cauchy--Schwartz inequality and using the $Z$-regularity \htreg we have
			\begin{align*}
				\abs{\cE_{\tau_i(Z)}}^2 
					\le h^2 \ \bE_i\Big[ \abs{f^h(\ti,Y_\tip,\overline{Z}_\ti) - f^h(\ti,Y_\tip,\widehat{Z}_i)}^2 \Big]
					\le h^2 (L^h_z)^2 \ \bE_i\Big[ \abs{\overline{Z}_\ti - \widehat{Z}_i }^2 \Big] .
			\end{align*}		
		Hence, taking expectations, we obtain $\bE\big[\abs{\cE_{\tau_i(Z)}}^2\big] \le h^2 (L^h_z)^2 \ \bE\big[ \abs{\overline{Z}_\ti - \widehat{Z}_i }^2 \big]$.

		\paragraph*{}\textit{Gathering the estimates.}
		We finally obtain
			\begin{align*}
				\sum_{i=0}^{N-1}  \bE\Big[ \abs{Y_\ti - \widehat{Y}_i}^2 \Big]
				%\\
					&
					\le 
					4 . 3^{3/2} h^2 L_y^2 \, 
					C_Y \,  \big(\mathrm{REG}_{Y,4}(h)\big)^\half \times N
					%
					%\sum_{i=0}^{N-1} \int_\ti^\tip \bE \Big[ 1+\abs{Y_u}^{4(m-1)}+\abs{Y_\tip}^{4(m-1)}\Big]^\half \udu \times \mathrm{REG}_{Y,4}(h)^\half				
					%\\
						%& \qquad				
						+ 4 h L_z^2 \times \mathrm{REG}_{Z,2}(h) 
						+ 5 h^3 L_t^2 \times N
						\\
						&
						+ 4 h^2 \ \sum_{i=0}^{N-1} \bE\bigg[ \abs{(f-f^h)(\ti,Y_\tip,\overline{Z}_\ti)}^2 \bigg]
						%\\
						%& \qquad \qquad 	
						+ 4 h (L^h_z)^2 \ \sum_{i=0}^{N-1} \bE_i\bigg[ \abs{\overline{Z}_\ti - \widehat{Z}_i }^2 \bigg] h .
			\end{align*}
		Here we used again the notation set in the path-regularity Theorem \ref{theorem---path.regularity} for $\mathrm{REG}_{Y,4}(h)$ and $\mathrm{REG}_{Z,2}(h)$. 
		From the said result we have $\mathrm{REG}_{Y,4}(h)^\half \le C h$ and $\mathrm{REG}_{Z,2}(h) \le C h$
		and hence, using the estimates on the size of the solution (see Theorem \ref{theo:BasicExistandUniqTheo})
		the first two terms on the RHS of the above estimate are bounded above by some $C h^2$.  
		The same goes for the third term by Lemma \ref{lemma---vanishing.effect.of.taming.the.driver} and the first part of the proof guarantees that the same goes for the last term. 
		This completes the proof of Proposition \ref{proposition---estimation.of.the.discretization.error}.
	\end{proof}

%////////////////////////////////////////////////////////////////////////////////////////////////////////////////////////////////////////////////////////////////////////////////////////
%////////////////////////////////////////////////////////////////////////////////////////////////////////////////////////////////////////////////////////////////////////////////////////
\subsection{Proof of convergence (Theorem \ref{theorem---main--convergence.of.tamed.scheme})} 
%\label{sec:proofofmaintheo}

We have established with Proposition \ref{proposition---stability.of.the.tamed.scheme.NEWVERSION.FINAL.DEFINITIVE.FOR.REAL} that the scheme is almost-stable, with the rate $\mu$ introduced in subsection \ref{sec:PC}. 
By the fundamental lemma \ref{lem:fundamental}, we therefore have the global error estimate
	\begin{align*}
		\big(\mathrm{ERR}_N\big)^2 \le C \bE[\,\abs{\xi - \xi^N}^2] + C \Bigg( \sum_{i=0}^{N-1} \frac{\tau_i(Y)}{h}   + \tau_i(Z)  \Bigg)  + C h^\mu .
	\end{align*}	
By assumption \hxiN, the first term is bounded above by $C h$ for some $C$. 
Proposition \ref{proposition---estimation.of.the.discretization.error} guarantees that the second term $\sum_{i=0}^{N-1} \frac{\tau_i(Y)}{h}   + \tau_i(Z)$ is also bounded above by $C h$.
Therefore, we have proven
$\big(\mathrm{ERR}_N\big)^2 \le C \, h + C h^\mu$, 
as claimed in Theorem \ref{theorem---main--convergence.of.tamed.scheme}.

%%%%%%%%%%%%%%%%%%%%%%%%%%%%%%%%%%%%%%%%%%%%%%%%%%%%%%%%%%%%%%%%%%%%%%%%%%%%%%%%%%%%%%%%%%%%%%%%%%%%%%%%%%%%%%%%%%%%%%%%%%%%%%%%%%%%%%%%%%%%%%%%%
%%%%%%%%%%%%%%%%%%%%%%%%%%%%%%%%%%%%%%%%%%%%%%%%%%%%%%%%%%%%%%%%%%%%%%%%%%%%%%%%%%%%%%%%%%%%%%%%%%%%%%%%%%%%%%%%%%%%%%%%%%%%%%%%%%%%%%%%%%%%%%%%%
%%%%%%%%%%%%%%%%%%%%%%%%%%%%%%%%%%%%%%%%%%%%%%%%%%%%%%%%%%%%%%%%%%%%%%%%%%%%%%%%%%%%%%%%%%%%%%%%%%%%%%%%%%%%%%%%%%%%%%%%%%%%%%%%%%%%%%%%%%%%%%%%%
%%%%%%%%%%%%%%%%%%%%%%%%%%%%%%%%%%%%%%%%%%%%%%%%%%%%%%%%%%%%%%%%%%%%%%%%%%%%%%%%%%%%%%%%%%%%%%%%%%%%%%%%%%%%%%%%%%%%%%%%%%%%%%%%%%%%%%%%%%%%%%%%%
%%%%%%%%%%%%%%%%%%%%%%%%%%%%%%%%%%%%%%%%%%%%%%%%%%%%%%%%%%%%%%%%%%%%%%%%%%%%%%%%%%%%%%%%%%%%%%%%%%%%%%%%%%%%%%%%%%%%%%%%%%%%%%%%%%%%%%%%%%%%%%%%% 
\section{Qualitative properties: discrete comparison and preservation of positivity}
\label{section---qualitative.properties}

In this section we discuss, in the $1$-dimensional case, the preservation of order and (consequently) of positivity by the discretization scheme.  As was pointed out in \cite{LionnetReisSzpruch2015}, the standard explicit scheme may fail to preserve positivity.
Here we are interested in determining conditions under which one can guarantee that the $Y_i$ approximation remains positive when continous-time solutions is positive. This problem is of significant importance, qualitatively and for numerical stability. On the one hand, prices should be non-negative as well as sizes of populations and chemical quantities. On the other hand, the driver may be monotone only on $D=\R_+$ and should one input $Y_\ip$ not be a.s. positive, the scheme may explode.
We are aware of only two other works stating similar comparison results, see \cite{CheriditoStadje2013} and \cite{ChassagneuxRichou2016} but, even in the Lipschitz setting, they do not deal with explicit schemes. The discussion on the explicit scheme is, to the best of our knowledge, new.

The analysis below is based on a linearization technique, as in \cite{ChassagneuxRichou2016}. 
We first show a generic discrete comparison result and then positivity follows as a corollary. 
Essentially due to technicalities arising by the explicit component in the scheme, the comparison result is not as general as one might expect from the implicit scheme. 
The several result below require at least the following assumption.

\begin{assumption}
\label{assump:forpositivity}
Assume $n=1$, that \htregY holds with $\RregY=0$ and
\begin{align*} 
\text{for }\theta'\in[0,1]\qquad 
\sup_{i=0,\cdots,N-1}	
   h \Big( L_y^h +  L_z^h |H_\ip| + (1-\theta') h L_z^h |H_\ip| L_y^h
	 \Big)<1.
\end{align*}
\end{assumption}

Since we work in the one-dimensional setting, the above assumption is not a drawback. 

\begin{proposition}[Discrete comparison for the explicit scheme] 
\label{prop:comparisonforexplicitscheme} 
Let Assumption \ref{assump:forpositivity} hold. For $j\in\{1,2\}$ take the modified drivers $f^{h,j}$ and numerical terminal conditions $\xi^{N,j}$, as well as the outputs $(Y_i^j,Z_i^j)_{i = 0 \ldots N}$ obtained through scheme \eqref{equation---reference--definition.of.the.scheme}. % is well defined ($L^2$-integrable). 
Define, for $0 \leq i \leq N-1$, 
\begin{align}
\label{betacoeff22}
\beta_\ip
& :=\frac{f^{h,1}(\ti,Y_{\ip}^1, Z_i^1)-f^{h,1}(\ti,Y_{\ip}^2, Z_i^1)}{ Y_{\ip}^1-Y_{\ip}^2} \1_{\{Y_{\ip}^1- Y_{\ip}^2 \neq 0\}},
\\
\label{betacoeff22hat}
\widehat \beta_\ip
& :=\frac{f^{h,1}(\ti,Y_{\ip}^1,0)-f^{h,1}(\ti,Y_{\ip}^2, 0)}{ Y_{\ip}^1-Y_{\ip}^2} \1_{\{Y_{\ip}^1- Y_{\ip}^2 \neq 0\}}
\\
\label{gammacoeff22}
\gamma_\ip
& :=\frac{f^{h,1}(\ti,Y_{i+1}^2, Z_i^1)-f^{h,1}(\ti,Y_{i+1}^2, Z_i^2 )}{\abs{Z_i^1-Z_i^2}^2} \left(Z_i^1-Z_i^2 \right)^* \1_{\{Z_i^1-Z_i^2 \neq 0\}},
\quad \text{and} 
\\
B_\ip &:= 1+h\beta_\ip 
			 + h \gamma_\ip \big( 1+(1-\theta')h\widehat\beta_\ip \big) H_\ip^*.
\end{align} 
Assume further that $\gamma_\ip$ is $\cF_i$-measurable for all $i$; and that either $(f^{h,1}-f^{h,2})(\ti,Y^2_\ip,0)$ is $\cF_i$-measurable $\forall i$ or $\theta'=1$ or $\forall i$ $\gamma_\ip=0$. Then, with the convention $\prod_{j=k}^{l} \cdot = 1$ for $l<k$, 
\begin{align}
\label{eq:schemelinearizationcomparison} 
	Y_{i}^1- Y_{i}^2
	 = \bE_i\Bigg[ 
	 		(\xi^{N,1}- \xi^{N,2}) \prod_{j=i}^{N-1} B_\jp 
			+ h \sum_{j=i}^{N-1}  \big(f^{h,1}-f^{h,2}\big)(\tj,Y_{j+1}^2,Z_j^2) \prod_{k=i}^{j-1} B_{k+1}
	      \Bigg].
\end{align}
If $\xi^{N,1}- \xi^{N,2}\geq 0$ and $\big(f^{h,1}-f^{h,2}\big)(\ti,Y_{i+1}^2,Z_i^2)\geq 0$ for all $0 \leq i \leq N-1$ then $Y_{i}^1\geq Y_{i}^2$ for all $0 \leq i \leq N$.
\end{proposition}

\begin{remark}[On the assumptions of the comparison theorem]
Two of the assumptions stand as non-trivial, and perhaps slightly opaque, namely that $\gamma_\ip$ is $\cF_i$-measurable 
and the $\cF_i$-measurability of $\big(f^{h,1}-f^{h,2}\big)(\ti,Y^2_\ip,0)=0$. 
The reason for both is of technical nature and due to the presence of the $Y_\ip$-term in the scheme.

Concerning the first, it happens for instance when one is able to write $f(t,y,z)=\widetilde f(t,y)+\widehat f(t,z)$ $\forall t,y,z$ or when the $Y^2_\ip$'s are deterministic. 
When $f$ does not depend on $z$ then $\gamma_\ip=0$, this is a case of interest for reaction-diffusion equations. 
When $\theta'\neq 1$, the restriction $\big(f^{h,1}-f^{h,2}\big)(Y^2_\ip,0)$ is $\cF_i$-measurable is a real limitation to the comparison result, as in general $Y^2_\ip$ is not $\cF_i$-measurable. 
%In general, one can still use the comparison result but only in situations where the drivers satisfy $f^{h,1}_\cdot(\cdot,0)=f^{h,2}_\cdot(\cdot,0)$. 
However, one is often interested in comparing the scheme against a constant: this will be case case when we prove a corollary on preservation of positivity.
\end{remark}

\begin{proof}
Let $i\in\{0,\cdots,N-1\}$. Under assumptions \htreg and \htregY (with $\RregY=0$) the random variables $\gamma$, $\beta$ and $\widehat \beta$ are well defined and satisfy:
$|\beta_\ip| \leq L_y^h$, $|\widehat \beta_\ip|\leq L_y^h$ and $
|\gamma_\ip|\leq L_z$.
%, these follows from the obvious implication of the Lipschitz condition that $f^{1,h}$ satisfies. 
Moreover, $\beta_\ip,\widehat \beta_\ip$ and $\gamma_\ip$ are $\cF_\ip$-adapted. 
%These estimates are easy to show as they follow from the obvious implication of the Lipschitz condition they satisfy (for fixed $h>0$), e.g. from \htreg follows that $z\mapsto f^h(\cdot,\cdot,z)$ is Lipschitz and hence $|\gamma_\cdot|\leq L_z^h$. The estimations on $\beta_\cdot$ and $\widehat \beta_\cdot$ are similar and follow from Assumption \htregY as in the same way as for $\gamma$. 

Define $\delta Y_i:= Y_i^1- Y_i^2$, $\delta Z_i:=Z_i^1- Z_i^2$ and 
$\widehat{\delta f}_\ip := \big(f^{h,1}-f^{h,2}\big)(\ti,Y_{\ip}^2,0)$. 
Recalling \eqref{equation---reference--definition.of.the.scheme}, and the notation \eqref{betacoeff22}, \eqref{betacoeff22hat} and \eqref{gammacoeff22}, we can write 
\begin{align*}
\delta Z_i 
& 
=
\bE_i\Big[ 	\Big( \delta Y_\ip
            			+(1-\theta') \big[ f^{h,1}(\ti,Y^1_\ip,0)-f^{h,2}(\ti,Y^2_\ip,0)\big] h \Big)  H_\ip ^* \Big]
\\
&
= \bE_i\left[ \Big(\delta Y_{i+1} 
                      + (1-\theta') h \widehat \beta_\ip \delta Y_\ip 
					+ (1-\theta') h \widehat{\delta f}_\ip \Big) H^*_\ip \right]
\\
&
= \bE_i\left[\delta Y_{i+1} \Big(1 + (1-\theta') h \widehat \beta_\ip \Big) H^*_\ip 
					+ (1-\theta') h \widehat{\delta f}_\ip H^*_\ip \right].
\end{align*}
For the $\delta Y$ component, we first linearize the driver terms then inject the above expression for $\delta Z$. 
Namely, defining $\delta f_\ip:= \big(f^{h,1}-f^{h,2}\big)(\ti,Y_{\ip}^2,Z_i^2)$ we have 
\begin{align*} 
\delta Y_i
& 
=\bE_{i}\Big[ 
      \delta Y_{i+1}\big( 1	+h \beta_\ip \big)
			+h \gamma_\ip \delta Z_i +h\delta f_\ip\Big]
\\
&
=\bE_{i}\Big[ 
      \delta Y_{i+1}\Big(
			1+h\beta_\ip 
			 + h \gamma_\ip \big( 1+(1-\theta')h\widehat\beta_\ip \big) H_\ip^*
			\Big)	
\\
&
\hspace{5cm}	
	+h\delta f_\ip
	 +h \gamma_\ip (1-\theta') h \widehat{\delta f}_\ip H_\ip^*	 \Big]
\\
&
=\bE_{i}\Big[ 
      \delta Y_{i+1}B_\ip +h\delta f_\ip \Big],
\end{align*}
where we used the assumption that $\gamma_\ip$ is $\cF_i$-measurable and then that either $\theta'=1$ or $\gamma_\ip=0$ or $\widehat{\delta f}\ip$ is $\cF_i$-measurable, so the last term vanishes due to the fact that $\bE_i[H_\ip]=0$. Furthermore, it is clear that without this second assumption, it not possible to have a comparison result as it is not possible to control the sign of $H_\ip$. Iterating the last inequality from $i$ to $N$ yields \eqref{eq:schemelinearizationcomparison}.

\emph{The comparison statement:} from Assumption \ref{assump:forpositivity} it follows that all $B_i$ terms are positive and hence the comparison statement follows provided $\delta f_\ip\geq 0$ and $\delta Y_N\geq 0$. 
\end{proof}

%\begin{remark}
%One may also want to consider the scheme given by $\hat{Y}_N = \xi^N$ ($\hat{Z}_N=0$) and 
%	\begin{align} 
%	\label{equation---reference--definition.of.the.scheme2}
%		\left\{\begin{aligned}
%			\hat{Y}_i &= \bE_i\Big[ \hat{Y}_\ip + f^h(\hat{Y}_\ip,\hat{Z}_\ip)h \Big], 
%			\\
%			\hat{Z}_i &=  \bE_i\Big[ \big( \hat{Y}_\ip +(1-\theta') f^h(\hat{Y}_\ip,\hat{Z}_\ip)h \big) H_\ip^* \Big],
%		\end{aligned}\right.
%	\end{align}
%The calculations carried out in Section \ref{section---qualitative.properties}, due to its explicitness-in-$Z$ (it requires $\hat{Z}_\ip$ to be part of the input, in addition to $\hat{Y}_\ip$), this scheme does not allow to recover a comparison result and consequently may not preserve positivity. Also, the approximation of $\hat{Z}_N$ needs also to be addressed. 
%\end{remark}

\begin{remark}
One may also want to consider instead of \eqref{equation---reference--definition.of.the.scheme}, a scheme using as input $(Y_\ip,Z_\ip)$, for instance setting 
$Y_i = \bE_i[ Y_\ip + f^h_i(Y_\ip,Z_\ip)h ]$.
However, due to the presence of $Z_\ip$ in the scheme, %ruins the computations above and 
one cannot guarantee a comparison result for this scheme.
\end{remark}

As a corollary of the previous result we have a preservation of positivity result.

\begin{corollary}[Preservation of positivity]
Let Assumption \ref{assump:forpositivity} hold.  If $\xi^N\geq 0$ and $f^h(\ti,0,0)\geq 0$ for all $0\leq i\leq N-1$ then $Y_i\geq 0$ for any $1\leq i\leq N$, in other words the tamed explicit scheme \eqref{equation---reference--definition.of.the.scheme} is positivity preserving. 

Moreover, if $\xi^N>0$ and $f^h(\ti,0,0)\geq 0$ for any $i\in\{0,\cdots, N-1\}$ then $Y_i>0$ $\forall i$.
\end{corollary}

\begin{proof} 
In Proposition \ref{prop:comparisonforexplicitscheme} take $(Y^2_i,Z^2_i)_{i=0 \ldots N}=0$ and $f^{h,2}=0$. Under this setting the random variables $\gamma_\ip$ defined in \eqref{gammacoeff22} are $\cF_i$-measurable, as is $\widehat{\delta f}_\ip= f^{h,1}(\ti,0,0)-0$. The expression \eqref{eq:schemelinearizationcomparison} then simplifies to 
\begin{align*}
	Y^1_i
	&
	=
		\bE_i\Big[ Y^1_N \prod_{j=i}^{N-1}B_\jp 
		             + h \sum_{j=i}^{N-1} f^{h,1}(\ti,0,0) \prod_{k=i}^{j-1} B_{k+1} 
	             \Big].
\end{align*}
Under Assumption \eqref{assump:forpositivity} the random variables $B_\ip$ are all positive and since $\xi^N\geq 0$ and $f^h(\ti,0,0)\geq 0$ the statement follows. The particular case of the strict inequalities follow trivially.
\end{proof}

%%%%%%%%%%%%%%%%%%%%%%%%%%%%%%%%%%%%%%%%%%%%%%%%%%%%%%%%%%%%%%%%%%%%%%%%%%%%%%%%%%%%%%%%%%%%%%%%%%%%%%%%%%%%%%%%%%%%%%%%%%%%%%%%%%%%%%%%%%%%%%%%%
%%%%%%%%%%%%%%%%%%%%%%%%%%%%%%%%%%%%%%%%%%%%%%%%%%%%%%%%%%%%%%%%%%%%%%%%%%%%%%%%%%%%%%%%%%%%%%%%%%%%%%%%%%%%%%%%%%%%%%%%%%%%%%%%%%%%%%%%%%%%%%%%%
%%%%%%%%%%%%%%%%%%%%%%%%%%%%%%%%%%%%%%%%%%%%%%%%%%%%%%%%%%%%%%%%%%%%%%%%%%%%%%%%%%%%%%%%%%%%%%%%%%%%%%%%%%%%%%%%%%%%%%%%%%%%%%%%%%%%%%%%%%%%%%%%%
%%%%%%%%%%%%%%%%%%%%%%%%%%%%%%%%%%%%%%%%%%%%%%%%%%%%%%%%%%%%%%%%%%%%%%%%%%%%%%%%%%%%%%%%%%%%%%%%%%%%%%%%%%%%%%%%%%%%%%%%%%%%%%%%%%%%%%%%%%%%%%%%%
%%%%%%%%%%%%%%%%%%%%%%%%%%%%%%%%%%%%%%%%%%%%%%%%%%%%%%%%%%%%%%%%%%%%%%%%%%%%%%%%%%%%%%%%%%%%%%%%%%%%%%%%%%%%%%%%%%%%%%%%%%%%%%%%%%%%%%%%%%%%%%%%%
%\newpage 
\section{Numerical simulations}
\label{section---numerical.simulations}

In this section, we illustrate with numerical simulations the study performed above regarding the convergence and qualitative properties of explicit schemes with tamed drivers. 

%Introducting M and K, and the rough idea for how to approximate the condional expectations.
Our analysis was concerned with the time-discretization. In practice, we also need to approximate the conditional expectations in the scheme \eqref{equation---reference--definition.of.the.scheme}. We do so using the method of regression on a family of function, see for instance \cite{GobetTurkedjiev2016}. 
Given a uniform time-discretization grid with $N$ time-steps, $\pi^N = (t_i)_{i=0 \ldots N}$ with $t_i=i\, h$ and $h=T/N$, we therefore simulate a sample of $M$ paths (which are an approximation) of the forward process $X$ : $(X^{N}_{m,i})$, for $m=1 \ldots M$ and $i=0 \ldots N$. In our examples $X$ will be an arithmetic Brownian motion and simulated exactly (not using a numerical scheme for SDEs). We then use this sample for the regression, at each time $t_i$, on a family of $K$ functions, which we take to be the first $K$ Hermite polynomials.

%//////////////////////////////////////////////////////////////////////////////////////////////////////////////////////////////////////////////////////////////////////////////////////
%//////////////////////////////////////////////////////////////////////////////////////////////////////////////////////////////////////////////////////////////////////////////////////
\subsection{Observing the convergence of the schemes}

We start by observing the convergence of the schemes. 
%
%The example PDE/FBSDE that we use
We consider here the following FBSDE, with time-horizon $T=1$. The forward process $X$ in \eqref{canonicSDE} is an Brownian motion started at $x=0$, with drift $b=0$ and diffusion coefficient $\sigma=1$. The BSDE \eqref{canonicBSDE} has driver $f(t,y,z)=f(y) = -y^3$, which is monotone (decreasing) on $\R$, and terminal condition $\xi=g(X_T)$ where $g(x)=\textrm{id}(x) = x$ is unbounded. 

\paragraph*{} %Schemes on the bench
As explained in the introduction, we can consider many generic ways of obtaining a modified driver $f^h$ from $f$. Here we consider an example of inner-taming, $f^h(y) = f\big( T^h(y) \big)$, an example of outer-taming $f^h(y) = T^h\big( f(y) \big)$, and an example of multiplicative taming $f^h(y)= \chi^h(y) f(y)$.
We take the functions $T^h$ to be projections on a ball centered in $0$ with growing radius, and we take the taming factor $\chi^h(y)$ of the form $\frac{1}{1+F(y) \tilde{R}(h)^{-1}}$.
An analysis of these examples and how they fit in our framework is provided in Appendix \ref{appendix:VerifyingExamples}.
Specifically, we compute approximations of the solutions using the following schemes. 
	\begin{enumerate}
		\item (black) The implicit scheme.
		\item (blue) A modified explicit scheme with driver $f^h$ tamed from inside by a projection, $f^h(y) = f\big( T^h(y) \big)$,
				where $T^h$ is the projection of the centered ball of radius $r^h = r_0 h^{-\gamma}$, with $r_0=1$ and $\gamma = \frac{1}{2(m-1)}$, 
				$m=3$ being the degree of the polynomial $f$ (see Appendix \ref{appendix:VerifyingExamples} for the choice of $\gamma$).
		\item (green) A modified explicit scheme with driver $f^h$ tamed from outside by a projection, $f^h(y) = T^h\big( f(y) \big)$,
				where $T^h$ is the projection of the centered ball of radius $R^h = R_0 h^{-\beta}$, with $R_0=1.5$ and $\beta = \frac{1}{2}$.
		\item (cyan) A modified explicit scheme with driver $f^h$ given by $f^h(y)= \frac{f(y)}{1+\abs{y}^{m-1} \tilde{R}_0^{-1} h^{\alpha}}$,
				where $\tilde{R}_0 = 1$ and $\alpha = \half$.
	\end{enumerate}

\paragraph*{} %Measuring the error : what to measure ?
We generate a sequence of uniform partitions $(\pi^{N})_{N}$ of $[0,T]$ with mesh $h=T/N$ for $N \in \{ 2, 4, 8, 16, \ldots , 2048 \}$ (we simulate first the Brownian paths on the finest partition, and then use these to compute the forward paths $X^N$ on all partition). 
Since we do not know the exact solution to the FBSDE, we use as a proxy the average $Y^{\textrm{proxy}}$ of the results returned by the schemes 1 and 2 for the finest time-grid. 
We measure as error the distance $\textrm{dist}(Y,Y') = \max_{i=0 \dots N} \bE\big[ \abs{Y_i - Y'_i}^2 \big]^{\half}$ between, on the one hand, the output $Y^N=(Y^N_i)_{i=0 \ldots N}$ of one of the schemes (1 to 4) and, on the other hand, the proxy $Y^{\textrm{proxy}}$ for the solution.

\paragraph*{} %Picture
On Figure \ref{Figure---convergence-error.vs.N}, we plot first the error versus the number of time-steps (left picture) and then the computation time versus the error (right picture), both in log-log scales. 

\begin{figure}[htb]
	\label{Figure---convergence-error.vs.N}
	\centering
	\subfigure
	{	\includegraphics[scale=0.35]{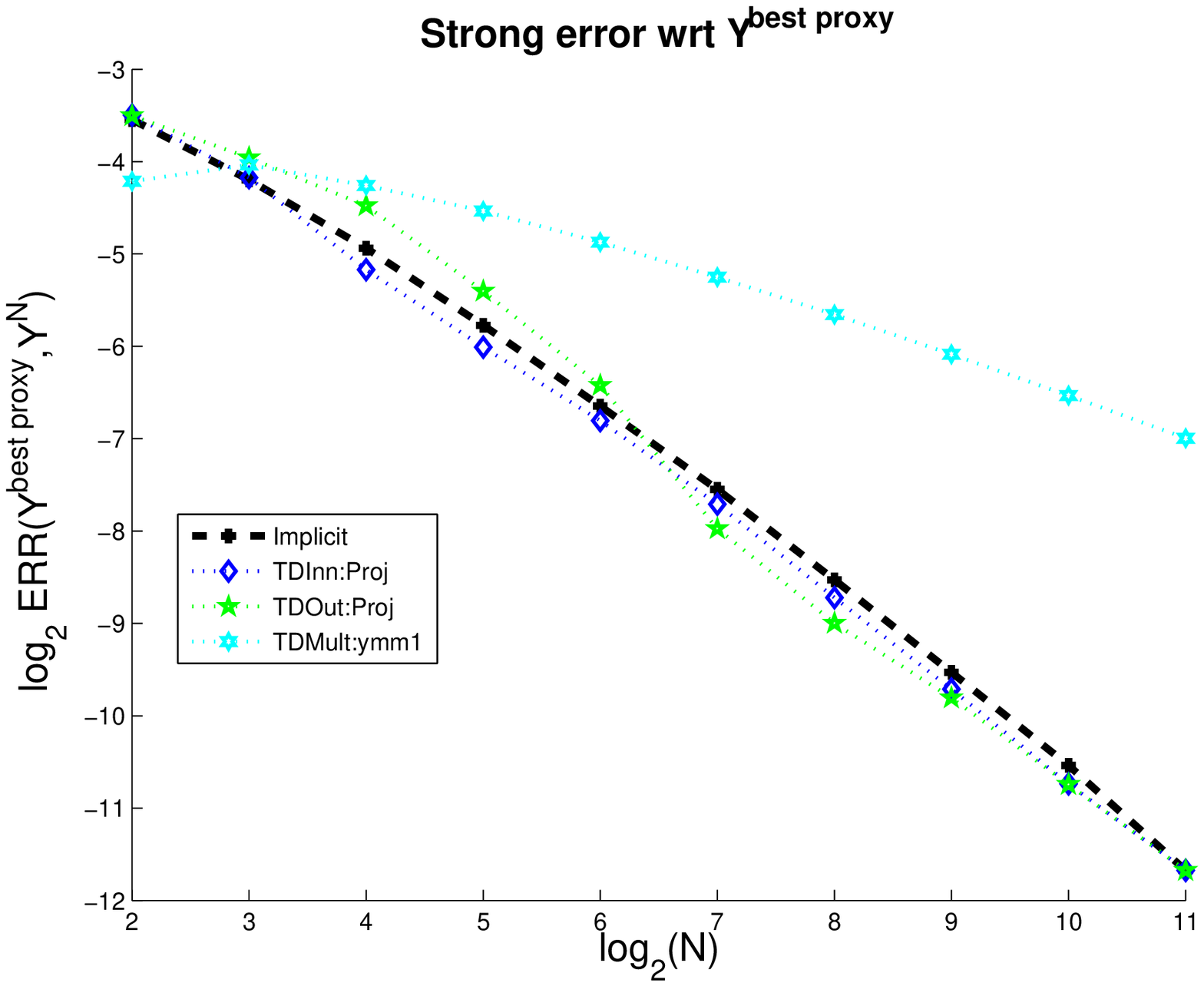}	}
	\subfigure
	{	\includegraphics[scale=0.35]{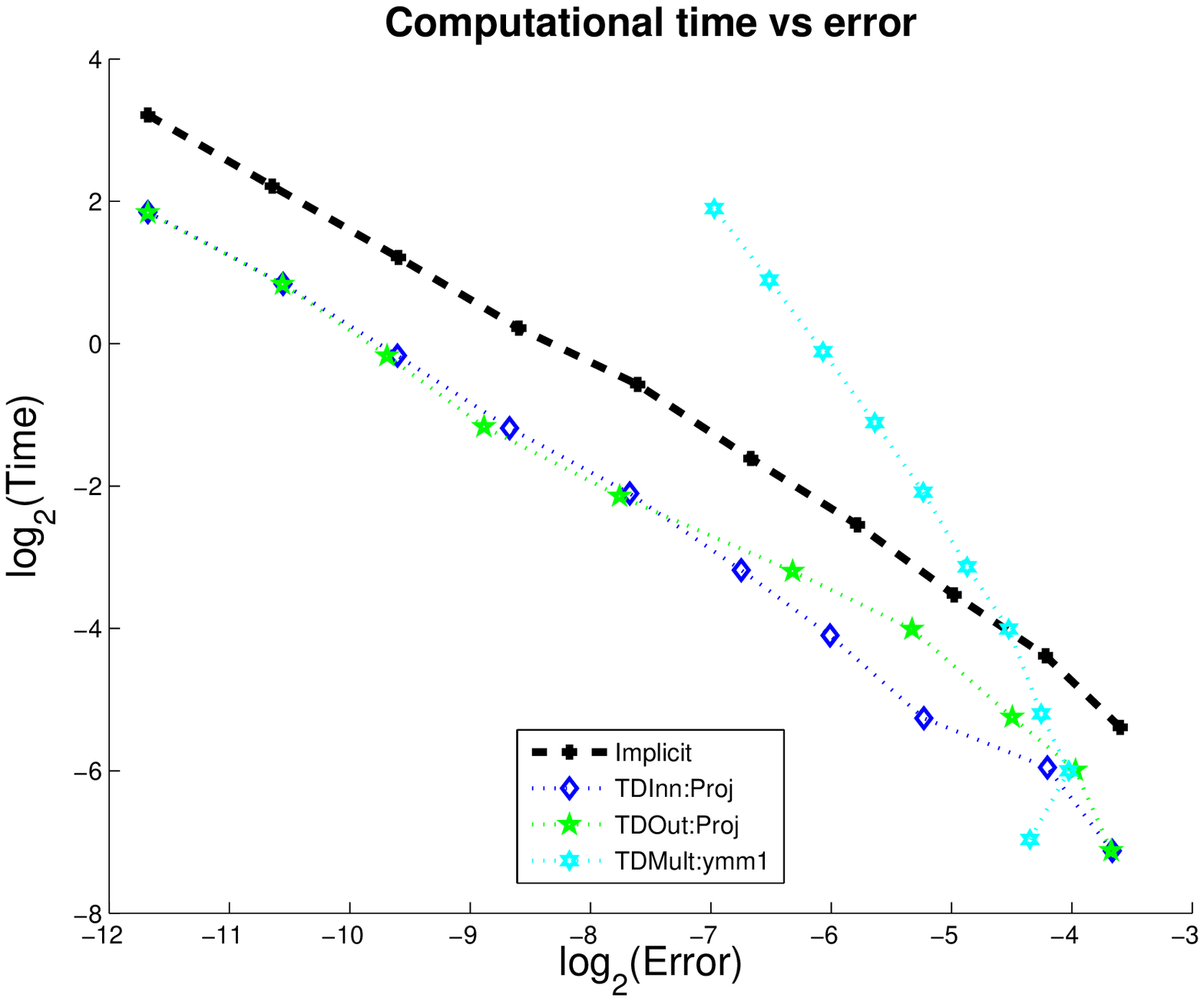}	}
	\caption{ Errors computed for $N \in \{4,8,16, \ldots, 2048 \}$, with {$M=100$k} and $K=10$. }
\end{figure}

We observe that the modified explicit schemes 2 and 3 provide errors comparable to the implicit scheme. However, as they are of explicit type, they benefit from a lower computation time. Scheme 4 however does not perform as well. 
We took $\alpha = \half$, as suggested by Appendix \ref{appendix:VerifyingExamples} so that the scheme fits in our framework. Then, Theorem \ref{theorem---main--convergence.of.tamed.scheme} guarantees convergence but since $\mu = \alpha < 1$ the modified scheme 4 has a convergence rate lower than the usual rate.
Essentially, what slows down the multiplicative schemes with $\chi^h(y)$ of the form $1/(1+F(y) \tilde{R}(h)^{-1})$ is that even when $F(y)$ (which can be $\abs{f(y)}$, $\abs{y}^m$, $\abs{y}^{m-1}$, ...) is small, i.e. $F(y) \le \tilde{R}(h)$, one does not have $f^h(y) = f(y)$. This creates an error compared to the true dynamics which is not necessary (since $f(y)$ is not big) and that error vanishes too slowly.

%//////////////////////////////////////////////////////////////////////////////////////////////////////////////////////////////////////////////////////////////////////////////////////
%//////////////////////////////////////////////////////////////////////////////////////////////////////////////////////////////////////////////////////////////////////////////////////
%\newpage 
\subsection{Numerical stability : preservation of positivity}

We now look at the qualitative behavior of the modified explicit schemes used above. 

\paragraph*{} %The example PDE/FBSDE that we use
We consider the following FBSDE, with $T=1$. The forward component $X$ in \eqref{canonicSDE} is a Brownian motion started at $x=0$, with drift $b=0$ and diffusion coefficient $\sigma=1.25$. The BSDE \eqref{canonicBSDE} has driver $f(t,y,z)=f(y) = -y^2$, which is monotone decreasing on the domain $D=[0,+\infty[$, and terminal condition $\xi=g(X_T)$ where $g(x)=x^2$ is positive. 

The solution of the continuous-time BSDE remains positive (i.e. in the domain $D$) and we have proven in Section \ref{section---qualitative.properties} that the modified explicit schemes should reproduce this property, at least under certain sufficient conditions ---Assumption \ref{assump:forpositivity} reduces in our case to $h \, L^h_y < 1$. 
Figure \ref{Figure---preservation.of.bounds} shows the empirical maximum and minimum of $Y^N_i$, as $\ti$ goes from $T$ to $0$. 

\begin{figure}[htb]
	\label{Figure---preservation.of.bounds}
	\centering
	{	\includegraphics[scale=0.75]{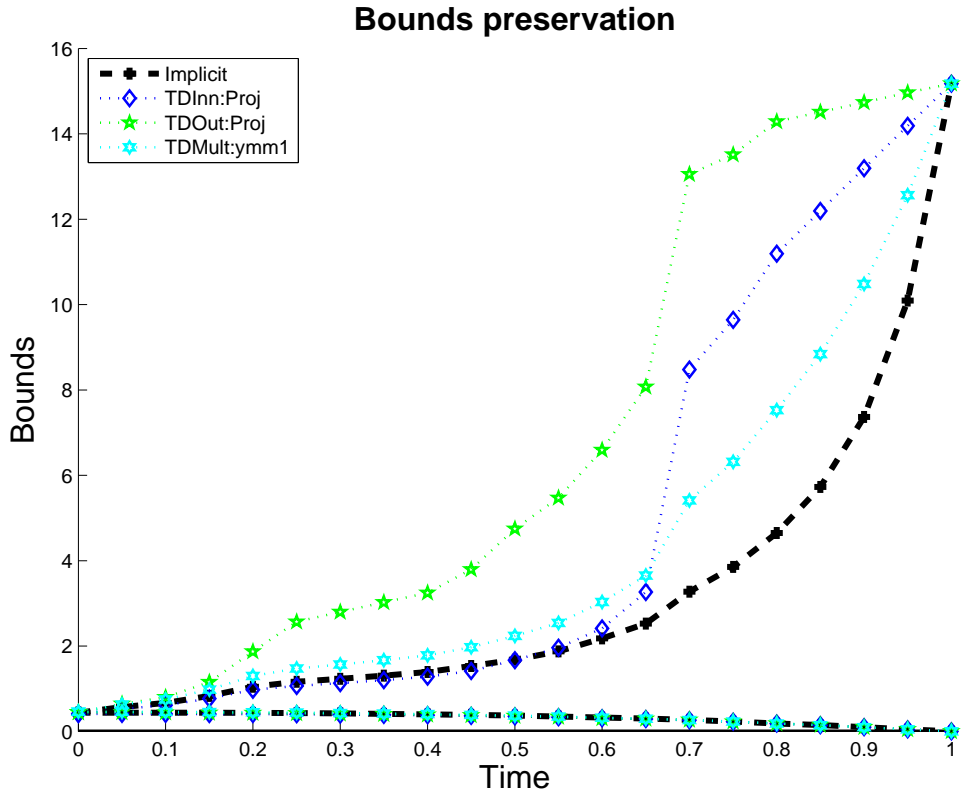}	}
	\caption{ Computed with $N=10$, {$M=50$k} and $K=30$.}
\end{figure}

We observe the desired preservation of positivity, and note that $N$ is only equal to $10$. 
We also observe, when looking at the upper bounds, a regular decay on the implicit scheme. This is due to the strict monotonicity of the driver $f$. The modified explicit schemes appear to preserve this qualitative behavior as well, though in a limited way. 

It must be stressed that such a numerical stability properties, the preservation of positivity and monotonicity, can fail due to an imperfect approximation of the conditional expectations. Our results in Section \ref{section---qualitative.properties} hold only for the time-discretization scheme and rely on the order-preserving property of the conditional expectation operators. In practice, when we tried to approximate the conditional expectations with $K=4$ functions, we found that frequently (the numerical method we used is a Monte-Carlo one) positivity was violated at some of the grid-points. Due to the taming of the driver, explosion was prevented, but the qualitative behavior is not satisfactory if the conditional expectations are poorly approximated.

%########################################################################################################################################
%####################### \BEGIN APPENDIX ##################################################################################################
%########################################################################################################################################
\appendix 

%%%%%%%%%%%%%%%%%%%%%%%%%%%%%%%%%%%%%%%%%%%%%%%%%%%%%%%%%%%%%%%%%%%%%%%%%%%%%%%%%%%%%%%%%%%%%%%%%%%%%%%%%%%%%%%%%%%%%%%%%%%%%%%%%%%%%%%%%%%%%%%%%
%%%%%%%%%%%%%%%%%%%%%%%%%%%%%%%%%%%%%%%%%%%%%%%%%%%%%%%%%%%%%%%%%%%%%%%%%%%%%%%%%%%%%%%%%%%%%%%%%%%%%%%%%%%%%%%%%%%%%%%%%%%%%%%%%%%%%%%%%%%%%%%%%
%%%%%%%%%%%%%%%%%%%%%%%%%%%%%%%%%%%%%%%%%%%%%%%%%%%%%%%%%%%%%%%%%%%%%%%%%%%%%%%%%%%%%%%%%%%%%%%%%%%%%%%%%%%%%%%%%%%%%%%%%%%%%%%%%%%%%%%%%%%%%%%%%
%%%%%%%%%%%%%%%%%%%%%%%%%%%%%%%%%%%%%%%%%%%%%%%%%%%%%%%%%%%%%%%%%%%%%%%%%%%%%%%%%%%%%%%%%%%%%%%%%%%%%%%%%%%%%%%%%%%%%%%%%%%%%%%%%%%%%%%%%%%%%%%%%
%%%%%%%%%%%%%%%%%%%%%%%%%%%%%%%%%%%%%%%%%%%%%%%%%%%%%%%%%%%%%%%%%%%%%%%%%%%%%%%%%%%%%%%%%%%%%%%%%%%%%%%%%%%%%%%%%%%%%%%%%%%%%%%%%%%%%%%%%%%%%%%%%
%\newpage
\section{Auxiliary results}
\label{appendix:SideResults}

%////////////////////////////////////////////////////////////////////////////////////////////////////////////////////////////////////////////////////////////////////////////////////////
%////////////////////////////////////////////////////////////////////////////////////////////////////////////////////////////////////////////////////////////////////////////////////////
\subsection{Background results on monotone BSDE with polynomial growth}
\label{appendix-BackgroundResults}

The results stated in this section hold under the assumption listed in Section \ref{subsection-AnalyticAssumption}. They can be found in \cite{LionnetReisSzpruch2015}*{Sections 2 and 3} and are slightly adapted to suit the framework in the main body of the present work.

\begin{theorem}[Existence and uniqueness]
\label{theo:BasicExistandUniqTheo}
	The FBSDE \eqref{canonicSDE}-\eqref{canonicBSDE} has a unique solution $(X,Y,Z)\in \cS^p\times \cS^p \times \cH^p$ for any $p\geq 2$. Moreover, it holds that 
		\begin{align*}
			\|Y\|_{\cS^{p}}^{p}
			+\| Z \|_{\cH^{p}}^{p} 
			& \leq C_p \big\{ 
			\|g(X_T)\|_{L^{p}}^{p}
			+ \| f(\cdot,X_\cdot,0,0) \|_{\cH^{p}}^{p} \big\}
		 \end{align*}
\end{theorem}

Take the uniform partition $\pi = (\ti)_{i=0, \ldots, N}$ of $[0,T]$ as defined in Section \ref{sec--AssumptionSection} and with mesh size $|\pi|=h=T/N$. 
Define the random variables $(\overline{Z}_\ti)_{i=0,\cdots,N-1}$ by
			\begin{align*}
				\overline{Z}_\ti := \frac1h \bE_i\Big[  \int_\ti^\tip Z_u \udu \Big].
			\end{align*}
We have then the following result (\cite{LionnetReisSzpruch2015}*{Theorem 3.5 and Corollary 3.6})

	\begin{theorem}[Integrability and Path regularity]		
	\label{theorem---path.regularity}
		For any $p\geq 2$ there exists a positive constant $C$ independent of $h$ such that: 
			\begin{align*} 
				\sup_{\ti\in\pi} \bE\big[ \abs{Y_\ti}^{p} \big] \le C
				\quad \text{and}\quad
				\mathrm{REG}_{Y,p}(h) := \sup_{\abs{s-t}\le h;\ t,s\in[0,T]} \bE\big[ \Abs{Y_s - Y_t}^p \big]\leq C h^{\frac p2}	,
			\end{align*}
		 moreover, for the control component $Z$ we have 
			\begin{align*}
				\sum_{i=0}^{N-1} \bE\Big[ \big( \abs{\overline{Z}_\ti}^2 h\big)^{\frac p2} \Big] \le C
				\quad \text{and}\quad
				\sup_{\ti\in\pi\cap[0,T)} \bE\big[\, \abs{ \overline Z_\ti}^{p} \big] \le C
			\end{align*}
		and the respective path regularity result
			\begin{align*}
				\mathrm{REG}_{Z,2}(h) :=\bE\Big[\, \sum_{i=0}^{N-1} \int_\ti^\tip \abs{Z_t - \overline{Z}_\ti}^2 \udt\, \Big] \le C h.
			\end{align*}
	\end{theorem}

%////////////////////////////////////////////////////////////////////////////////////////////////////////////////////////////////////////////////////////////////////////////////////////
%////////////////////////////////////////////////////////////////////////////////////////////////////////////////////////////////////////////////////////////////////////////////////////
\subsection{Proof of Lemma \ref{lem:mrt} (discrete-time martingale representation)}
\label{appendix-proof-discrete-MRT}

\begin{proof}[Proof of Lemma \ref{lem:mrt}]
Equation \eqref{equation---reference--definition.of.Yi.with.martingale} clearly defines $\Delta M_ip$ as being $\cY_\ip + f^h(\cY_\ip,\cZ_i)h - \bE_i[\cY_\ip + f^h(\cY_\ip,\cZ_i)h]$. 
Conversely, defining $\Delta M_\ip$ clearly yields a martingale increment: $\bE_i[\Delta M_\ip] = 0$.
So we now want to prove the existence and uniqueness of the decomposition $(\zeta_i,\Delta N_\ip)$ of the martingale increment $\Delta M_\ip$.

\textit{Uniqueness.} 
Let $(\zeta_i,\Delta N_\ip)$ be such a decomposition. That is, we have $\zeta_i$ a $\F_i$-measurable r.v. and $\Delta N_\ip$ a martingale increment orthogonal to $H_\ip h$ (in other words $\bE_i[\Delta N_\ip]= 0$ and $\bE_i[\Delta N_\ip (H_\ip h)^*]=0$), satisfying the decomposition \eqref{equation---reference--decomposition.of.martingale.increment}.
Multiplying the equation \eqref{equation---reference--decomposition.of.martingale.increment} by $H_\ip^*$ on the right, taking conditional expectation, using the orthogonality between $H_\ip^*$ and $\Delta N_\ip$, and recalling \hH.2 implies
	\begin{align*}
		\bE_i\Big[ \Delta M_\ip H_\ip^* \Big] 
			&= \bE_i\Big[ \zeta_i \Lambda^{-1}  (H_\ip h) (H_\ip^*) \Big] + \bE_i\Big[ \Delta N_\ip H_\ip^* \Big]						\\
		 	&= \zeta_i \Lambda^{-1} \ h^{-1} \bE_i\Big[  (H_\ip h) (H_\ip h)^* \Big]	+ 0
			= \zeta_i ,
	\end{align*}
which yields the uniqueness of $\zeta_i$.
On the other hand \eqref{equation---reference--decomposition.of.martingale.increment} directly implies 
	\begin{align*}
		\Delta N_\ip = \Delta M_\ip - \zeta_i \Lambda^{-1} H_\ip h \ ,
	\end{align*} 
which yields the uniqueness of $\Delta N_\ip$.

\textit{Existence.} 
Define
	\begin{align*}
		\zeta_i = \bE_i\Big[ \Delta M_\ip H_\ip^* \Big] 
		\quad \text{and} \quad 
		\Delta N_\ip = \Delta M_\ip - \zeta_i \Lambda^{-1} H_\ip h \ .
	\end{align*}
It is then obvious that we have $\Delta M_\ip = \zeta_i \Lambda^{-1} H_\ip h + \Delta N_\ip$, and that $\zeta_i$ is $\F_i$-measurable. 
It remains to check that $\Delta N_\ip$ is a martingale increment and that it is orthogonal to $H_\ip h$.  
The first point follows easily from the fact that $\Delta M_\ip$ and $H_\ip$ are martingale increments and that $\zeta_i$ is $\F_i$-measurable:
	\begin{align*}
		\bE_i\big[ \Delta N_\ip \big] 
			= \bE_i\big[ \Delta M_\ip - \zeta_i \Lambda^{-1} H_\ip h \big] 
			= \bE_i\big[ \Delta M_\ip \big] - \zeta_i \Lambda^{-1} \bE_i\big[ H_\ip \big] h 
			= 0 \ .
	\end{align*}
The second point follows by computing, using the definition of $\zeta_i$,
	\begin{align*}
		\bE_i\Big[ \Delta N_\ip (H_\ip h)^* \Big]
			&= \bE_i\bigg[ \Big( \Delta M_\ip - \zeta_i \Lambda^{-1} H_\ip h \Big) (H_\ip h)^* \bigg]																						\\
			&= \bE_i\Big[ \Delta M_\ip H_\ip^* \Big] h - \zeta_i \Lambda^{-1} \ \Lambda h I_d
			%\\
			%&
			= \bE_i\Big[ \Delta M_\ip H_\ip^* \Big] h - \zeta_i h 	= 0 .
	\end{align*}
Therefore, the pair $(\zeta_i,\Delta N_\ip)$ so-defined is a solution, which proves existence.

To conclude, we just rewrite 
	\begin{align*}
		\zeta_i 
			&= \bE_i\Big[ \Delta M_\ip H_\ip^* \Big]
					= \bE_i\Big[ \big( \cY_\ip + f^h(\cY_\ip,\cZ_i)h - \bE_i[\cY_\ip + f^h(\cY_\ip,\cZ_i)h] \big) H_\ip^* \Big]														\\
			&= \bE_i\Big[ \big( \cY_\ip + f^h(\cY_\ip,\cZ_i)h \big) H_\ip^* \Big] - \bE_i[\cY_\ip + f^h(\cY_\ip,\cZ_i)h] \ \bE_i\Big[ H_\ip^* \Big]						\\
			&= \bE_i\Big[ \big( \cY_\ip + f^h(\cY_\ip,\cZ_i)h \big) H_\ip^* \Big] .
	\end{align*}
\end{proof}

%/////////////////////////////////////////////////////////////////////////////////////////////////////////////////////////////////////////////////////////////////////////////////////////
%/////////////////////////////////////////////////////////////////////////////////////////////////////////////////////////////////////////////////////////////////////////////////////////
\subsection{Iteration and Fundamental Lemma}		\label{sec--proof.of.lemmas---iteration.and.fundamental}

We first state a particularly useful ``iteration lemma''.

	\begin{lemma}		\label{lemma---iterating.a.one.step.estimate}
		 Let $(a_i)$, $(b_i)$, $(c_i)$, $i \in \{0, \ldots, N\}$, be sequences of positive numbers. %such that for all $i$, $a_i,b_i \ge 0$ and  $c_i\in\bR$.
		Assume that, there exist constants $c \ge 0$ and $h>0$, such that, for all $i \in \{0, \ldots, N-1\}$,
			\begin{align}		\label{eq:rec}
				 a_i + b_i \le (1+ch)a_{i+1} + c_i .
			\end{align}
		Then, for all $i$, 
		\begin{align*}
			a_i  + \sum_{j=i}^{N-1} b_j  \le  e^{c(N-i)h} a_N + \sum_{j=i}^{N-1} e^{c(j-i)h} c_j \le e^{c(N-i)h} a_N + e^{c(N-1-i)h} \sum_{j=i}^{N-1} c_j.
		\end{align*}
	\end{lemma}

	\begin{proof}
		The first estimate is clearly true for $i=N-1$ (even for $i=N$ in fact), since $1+ch \le e^{ch}$. 
		Then, for any given $i \le N-2$, if it is true for $i+1$, by multiplying both sides by $e^{ch}$ we find that
		\begin{align*}
			e^{ch} a_\ip + e^{ch} \sum_{j=i+1}^{N-1} b_j \le e^{c(N-i)h} a_N +
		\sum_{j=i+1}^{N-1} e^{c(j-i)h} c_j
		\end{align*}
		Summing this inequality with \eqref{eq:rec} and noting that 
		$\sum_{j=i+1}^{N-1} b_j \le e^{ch} \sum_{j=i+1}^{N-1} b_j$ due to the positivity of the
		$b_j$ terms gives the first estimate for the given $i$.
		The second follows from the fact that
		$\sum_{j=i}^{N-1} e^{c(j-i)h} c_j \le e^{c(N-1-i)h} \sum_{j=i}^{N-1} c_j$ since the $c_i$'s are positive.
	\end{proof}

We now prove the Fundamental Lemma. 

	\begin{proof}[Proof of Lemma \ref{lem:fundamental}]
		We write
			\begin{align*}
				\bE\big[ \abs{Y_\ti - Y_i}^2 \big]
						& =  \bE\big[ \abs{Y_\ti -\widehat{Y}_i + \widehat{Y}_i - Y_i }^2 \big]
						 \le \Big(1+\frac{1}{h}\Big) \bE\big[ \abs{Y_\ti -\widehat{Y}_i}^2 \big] + (1+h) \bE\big[ \abs{\widehat{Y}_i - Y_i }^2 \big]							\\
				\bE\big[ \abs{\overline{Z}_\ti - Z_i}^2 \big]
						& =  \bE\big[ \abs{\overline{Z}_\ti -\widehat{Z}_i + \widehat{Z}_i - Z_i }^2 \big]
						 \le 2 \bE\big[ \abs{\overline{Z}_\ti -\widehat{Z}_i}^2 \big] + 2 \bE\big[ \abs{\widehat{Z}_i - Z_i }^2 \big].
			\end{align*}
		Since the scheme is almost-stable, we have
			\begin{align*}
				\bE\big[ &\abs{Y_\ti - Y_i}^2 \big] + \frac{1}{8} \bE\big[ \abs{\overline{Z}_\ti - Z_i}^2 \big]	 h																				\\
					&\le \Big(1+\frac{1}{h}\Big) \bE\big[ \abs{Y_\ti -\widehat{Y}_i}^2 \big] + (1+h) \bE\big[ \abs{\widehat{Y}_i - Y_i }^2 \big]
									+ \frac{1}{4} \bE\big[ \abs{\overline{Z}_\ti -\widehat{Z}_i}^2 \big] h + \frac{1}{4} \bE\big[ \abs{\widehat{Z}_i - Z_i }^2 \big] h			\\
					&\le (1+h) \Big( \bE\big[ \abs{\widehat{Y}_i - Y_i}^2 \big] + \frac{1}{4} \bE\big[ \abs{\widehat{Z}_i - Z_i}^2 \big] h \Big)
									+ \Big(1+\frac{1}{h}\Big) \tau_i(Y) + \frac{1}{4} \tau_i(Z)																													\\
					&\le (1+h) (1+ch) \bE\big[ \abs{Y_\tip - Y_\ip}^2 \big] + (1+h)C h^{\mu+1} 
									+ \Big(1+\frac{1}{h}\Big) \tau_i(Y) + \frac{1}{4} \tau_i(Z)																													\\
					&\le (1+ch) \bE\big[ \abs{Y_\tip - Y_\ip}^2 \big]
									+ C \Big( \frac{\tau_i(Y)}{h} + \tau_i(Z) \Big) + C h^{\mu+1} .
			\end{align*}		
		where the constants $c$ and $C$ changed on the last line and we used $h \le T$.
		Applying Lemma \ref{lemma---iterating.a.one.step.estimate} then gives, since $\ti = ih$ and $h=\frac{T}{N}$,
			\begin{align*}
				\bE\big[ &\abs{Y_\ti - Y_i}^2 \big] + \frac{1}{8} \sum_{j=i}^{N-1} \bE\big[ \abs{\overline{Z}_\tj - Z_j}^2 \big] h													\\
					&\le e^{c(T-\ti)} \Bigg( \bE\big[ \abs{\xi - \xi^N}^2 \big] 
								+ \sum_{j=i}^{N-1} C \Big( \frac{\tau_i(Y)}{h} + \tau_i(Z) \Big) + \sum_{j=i}^{N-1} C h^{\mu+1}	\Bigg)												\\
					&\le C \bE\big[ \abs{\xi - \xi^N}^2 \big] + C \bigg( \sum_{j=i}^{N-1} \frac{\tau_i(Y)}{h} + \tau_i(Z) \bigg) + C h^{\mu} .
			\end{align*}
	\end{proof}

%/////////////////////////////////////////////////////////////////////////////////////////////////////////////////////////////////////////////////////////////////////////////////////////
%/////////////////////////////////////////////////////////////////////////////////////////////////////////////////////////////////////////////////////////////////////////////////////////
\subsection{Proof of Proposition \ref{proposition---moment.estimate} (moment bounds for the scheme)}
\label{sec--proof.of.proposition---Prop:Integrability}

	\begin{proof}[Proof of Proposition \ref{proposition---moment.estimate}]
		Taking the power $p \ge 1$ in the estimate of Proposition \ref{prop:aspathestimate}, using $(a+b)^p \le 2^{p-1}(a^p + b^p)$,
		we have
			\begin{align*}
				\abs{Y_i}^{2p} 
					&\le \Big( e^{cT} \bE_i[\abs{\xi^N}^2] + e^{cT} C T \Big)^p									
					%\\
					%&
					\le 2^{p-1} e^{pcT} \bE_i[\abs{\xi^N}^2]^p + 2^{p-1} e^{pcT} (CT)^p .
			\end{align*}		
		Using the Jensen inequality and taking the expectation we therefore have
			\begin{align*}
				\bE[\abs{Y_i}^{2p}] \le 2^{p-1} e^{pcT} \bE[\abs{\xi^N}^{2p}] + 2^{p-1} e^{pcT} (CT)^p .
			\end{align*}		
		Given the moment assumption in \hxiN, this proves the first estimate.
		For the second, we come back to the one-step estimate of Proposition \ref{proposition---one-step.size.estimate}, take the power $p \ge 1$, 
		use $a^p + b^p \le (a+b)^p$ and use the Jensen inequality to write
			\begin{align*}
				\abs{Y_i}^{2p} + \Big(\frac{d}{4}\Big)^p \big( \abs{Z_i}^2 h \big)^p 
					&\le \left( e^{ch} \bE_i[\abs{Y_\ip}^2] + Ch \right)^p																														\\
					& =  e^{pch}\bE_i[\abs{Y_\ip}^2]^p + \sum_{k=1}^p \matrix{p\\k} \Big( e^{ch}\bE_i[\abs{Y_\ip}^2] \Big)^{p-k} (Ch)^k						\\
					&\le e^{pch}\bE_i[\abs{Y_\ip}^{2p}] + \sum_{k=1}^p \matrix{p\\k} \Big( e^{ch}\bE_i[\abs{Y_\ip}^2] \Big)^{p-k} (Ch)^k .
			\end{align*}
		Iterating this estimate (see Lemma \ref{lemma---iterating.a.one.step.estimate}) we obtain
			\begin{align*}
				\Big(\frac{d}{4}\Big)^p \bE_i & \bigg[ \sum_{j=i}^{N-1} \big( \abs{Z_i}^2 h \big)^p \bigg]																												\\
					&\le e^{pc(N-i)h}\bE_i[\abs{Y_N}^{2p}] 																																			
								+ \sum_{j=i}^{N-1} e^{pc(j-i)h} \bE_i\bigg[ \sum_{k=1}^p \matrix{p\\k} \Big( e^{ch}\bE_j[\abs{Y_\jp}^2] \Big)^{p-k} (Ch)^k \bigg] 									\\
					&\le e^{pc(T-\ti)}\bE_i[\abs{\xi^N}^{2p}] 																																			
								+ e^{pc(N-1-i)h}  \sum_{j=i}^{N-1} \sum_{k=1}^p \matrix{p\\k} e^{c(p-k)h} \bE_i\Big[ \bE_j[\abs{Y_\jp}^2]^{p-k} \Big] (Ch)^k  .
			\end{align*}
		One can then use the H?Â¶lder inequality and the Jensen inequality to further obtain
			\begin{align*}
				\Big(\frac{d}{4}\Big)^p \bE_i & \bigg[ \sum_{j=i}^{N-1} \big( \abs{Z_i}^2 h \big)^p \bigg]																												\\
					&\le e^{pc(T-\ti)}\bE_i[\abs{\xi^N}^{2p}] 																																			
								+ e^{pc(N-1-i)h} e^{pch} \sum_{j=i}^{N-1} \sum_{k=1}^p \matrix{p\\k} \bE_i\Big[ \bE_j[\abs{Y_\jp}^2]^{p} \Big]^{\frac{p-k}{p}}  (Ch)^k  							\\
					&\le e^{pc(T-\ti)}\bE_i[\abs{\xi^N}^{2p}] 																																			
								+ e^{pc(T-\ti)} \sum_{j=i}^{N-1} \sum_{k=1}^p \matrix{p\\k} \bE_i\Big[ \bE_j[\abs{Y_\jp}^{2p}] \Big]^{\frac{p-k}{p}}  (Ch)^k  											\\
					& =  e^{pc(T-\ti)}\bE_i[\abs{\xi^N}^{2p}] 																																			
								+ e^{pc(T-\ti)} \sum_{k=1}^p \matrix{p\\k} \sum_{j=i}^{N-1} \bE_i\big[ \abs{Y_\jp}^{2p} \big]^{\frac{p-k}{p}}  (Ch)^k .
			\end{align*}
		In particular, for $i=0$, we obtain in the end
			\begin{align*}
				&
				\Big(\frac{d}{4}\Big)^p \bE  \bigg[ \sum_{j=0}^{N-1} \big( \abs{Z_i}^2 h \big)^p \bigg]
				\\
				& \qquad 
				\le e^{pcT}\bE[\abs{\xi^N}^{2p}]
					+ e^{pcT} \sum_{k=1}^p \matrix{p\\k} \Big(\sup_{0 \le j \le N-1} \bE\big[ \abs{Y_\jp}^{2p} \big] \Big)^{\frac{p-k}{p}}  C^k h^{k-1} T ,
			\end{align*}
		which, in view of the moment estimate just proved for $(Y_i)$, yields the desired result.
	\end{proof}

%/////////////////////////////////////////////////////////////////////////////////////////////////////////////////////////////////////////////////////////////////////////////////////////
%/////////////////////////////////////////////////////////////////////////////////////////////////////////////////////////////////////////////////////////////////////////////////////////
\subsection{Proof of Lemma \ref{lemma---vanishing.effect.of.taming.the.driver} (vanishing effect of the tamed drivers)} 
\label{sec--proof.of.lemma---vanishing.effect.of.taming.the.driver} 

	\begin{proof}[Proof of Lemma \ref{lemma---vanishing.effect.of.taming.the.driver}]
		We verify the claim of the lemma in each of the three cases covered by our assumption \htfcvg. 
		We only prove the second estimate, as it is clear from the proof how the first follows. 
		Recall that $R^h=f-f^h$, and that $C$ denotes a constant whose value may change from line to line.
		
		\textit{Case 1.} 
		We use the inequality $(\sum_{i=1}^n a_i)^2 \le n \sum_{i=1}^n a_i^2$ to write
			\begin{align*}
				\bE\Big[ \abs{R^h(Y_\tip,\overline{Z}_\ti)}^2 \Big]
					&\le C \bE\Big[ \big( 1 + \abs{Y_\tip}^{q} + \abs{\overline{Z}_\ti}^{p} \big)^2 \Big] h^{2\alpha}
					%\\
					%&
					\le 
					C \bE\Big[ 1 + \abs{Y_\tip}^{2q} + \abs{\overline{Z}_\ti}^{2p} \Big] h^{2\alpha} .
			\end{align*}
		Hence, using the moment bounds from Theorem \ref{theorem---path.regularity}, 
			\begin{align*}
				\sum_{i=0}^{N-1} \bE\big[\, \abs{R^h(Y_\tip,\overline{Z}_\ti)}^2 \big]																																					
				& \le C h^{2\alpha-1} + C \sup_{i} \bE\big[\, \abs{Y_\tip}^{2q} \big] h^{2\alpha-1} + C \sup_{i} \bE\big[ \abs{\overline Z_\ti}^{2p} \big] h^{2\alpha-1}							\\
				&\le C h^{2\alpha-1} \le C ,
			\end{align*}
			since $\alpha \ge \half$ by assumption. 
			This proves the desired result for Case 1.

		\textit{Case 2.} 
		Using $(\sum_{i=1}^n a_i)^q \le C \sum_{i=1}^n a_i^q$, the Cauchy-Schwartz inequality, the Markov inequality with a power $l \ge 1$ yet to be determined, and \hfgrowth, we have
			\begin{align*}
        			\bE\Big[ \abs{R^h(Y_\tip,\overline{Z}_\ti)}^2 \Big]
        				&\le C \bE\Big[ \big( 1+ \abs{Y_\tip}^{2q} + \abs{\overline{Z}_\ti}^{2q} \big) \1_{\{\abs{f(Y_\tip,\overline{Z}_\ti) > r(h)}\}} \Big]																				\\ 
					&\le C \bE\Big[ 1 + \abs{Y_\tip}^{4q} + \abs{\overline{Z}_\ti}^{4p} \Big]^\half \Big( \bE\big[ \abs{f(Y_\tip,\overline{Z}_\ti)}^l \big] r(h)^{-l} \Big)^\half											\\
					&\le C \bE\Big[ 1 + \abs{Y_\tip}^{4q} + \abs{\overline{Z}_\ti}^{4q} \Big]^\half \bE\Big[ 1 + \abs{Y_\tip}^{lm} + \abs{\overline{Z}_\ti}^{l}  \Big]^\half h^{\frac{\beta l}{2}} .
			\end{align*}
		Using systematically the moment bounds from Theorem \ref{theorem---path.regularity}, we then have
			\begin{align*}
				\bE\Big[ \abs{R^h(Y_\tip,\overline{Z}_\ti)}^2 \Big]
					\le C h^{\frac{\beta l}{2} }.
			\end{align*}
		The desired result then follows since, given that $\beta > 0$, we can take $l = \frac{2}{\beta}$, so that $\frac{\beta l}{2} \ge 1$.
					
		\textit{Case 3.} 
		Using $(\sum_{i=1}^n a_i)^q \le C \sum_{i=1}^n a_i^q$, the Cauchy-Schwartz inequality, and the Markov inequality with a power $l \ge 1$ yet to be determined, we have
			\begin{align*}
				\bE\Big[ \abs{R^h(Y_\tip,\overline{Z}_\ti)}^2 \Big]
					&
					\le 
					C \bE\Big[ \big( 1 + \abs{Y_\tip}^{2q} + \abs{\overline{Z}_\ti}^{2p} \big) \1_{\{\abs{Y_\tip}>r(h)\}} \Big]
					\\
					&
					\le 
					C \bE\Big[ 1 + \abs{Y_\tip}^{4q} + \abs{\overline{Z}_\ti}^{4p} \Big]^\half \bE\Big[ \1_{\{\abs{Y_\tip}>r(h)\}} \Big]^\half
					\\
					%&
					%\le 
					%C \bE\Big[ 1 + \abs{Y_\tip}^{4q} + \abs{\overline{Z}_\ti}^{4p} \Big]^\half \Big( \bE\big[ \abs{Y_\tip}^l \big] r(h)^{-l} \Big)^\half
					%\\
					& 
					\leq  
					C \bE\Big[ 1 + \abs{Y_\tip}^{4q} + \abs{\overline{Z}_\ti}^{4p} \Big]^\half \bE\big[ \abs{Y_\tip}^l \big]^\half h^{\frac{\gamma l}{2}}					.
			\end{align*}
		Using systematically the moment bounds from Theorem \ref{theorem---path.regularity}, we then have
			\begin{align*}
				\bE\Big[ \abs{R^h(Y_\tip,\overline{Z}_\ti)}^2 \Big]	
					\le C h^{\frac{\gamma l}{2} }.
			\end{align*}
		The desired result then follows since, given that $\gamma > 0$, we can take $l$, so that $\frac{\gamma l}{2} \ge 1$.
	\end{proof}

%%%%%%%%%%%%%%%%%%%%%%%%%%%%%%%%%%%%%%%%%%%%%%%%%%%%%%%%%%%%%%%%%%%%%%%%%%%%%%%%%%%%%%%%%%%%%%%%%%%%%%%%%%%%%%%%%%%%%%%%%%%%%%%%%%%%%%%%%%%%%%%%%
%%%%%%%%%%%%%%%%%%%%%%%%%%%%%%%%%%%%%%%%%%%%%%%%%%%%%%%%%%%%%%%%%%%%%%%%%%%%%%%%%%%%%%%%%%%%%%%%%%%%%%%%%%%%%%%%%%%%%%%%%%%%%%%%%%%%%%%%%%%%%%%%%
%%%%%%%%%%%%%%%%%%%%%%%%%%%%%%%%%%%%%%%%%%%%%%%%%%%%%%%%%%%%%%%%%%%%%%%%%%%%%%%%%%%%%%%%%%%%%%%%%%%%%%%%%%%%%%%%%%%%%%%%%%%%%%%%%%%%%%%%%%%%%%%%%
%%%%%%%%%%%%%%%%%%%%%%%%%%%%%%%%%%%%%%%%%%%%%%%%%%%%%%%%%%%%%%%%%%%%%%%%%%%%%%%%%%%%%%%%%%%%%%%%%%%%%%%%%%%%%%%%%%%%%%%%%%%%%%%%%%%%%%%%%%%%%%%%%
%%%%%%%%%%%%%%%%%%%%%%%%%%%%%%%%%%%%%%%%%%%%%%%%%%%%%%%%%%%%%%%%%%%%%%%%%%%%%%%%%%%%%%%%%%%%%%%%%%%%%%%%%%%%%%%%%%%%%%%%%%%%%%%%%%%%%%%%%%%%%%%%%
\section{Verifications for some standard ways to tame the driver}
\label{appendix:VerifyingExamples}

%/////////////////////////////////////////////////////////////////////////////////////////////////////////////////////////////////////////////////////////////////////////////////////////
%/////////////////////////////////////////////////////////////////////////////////////////////////////////////////////////////////////////////////////////////////////////////////////////
\subsection{Examples of modified drivers}

In the main body of this work, we have shown how the set of general assumptions of subsubsection \ref{sec:fh} about the modified drivers $f^h$ ensure the convergence of the scheme.
As we stated in the introduction, our aim was to isolate a set of properties that would guarantee the convergence of the modified explicit scheme \eqref{equation---reference--definition.of.the.scheme} for a large class of modified drivers $f^h$, rather than just treating one particular case.
To use this result, we then have to show that various modifications of $f$ result in drivers $f^h$ which fit in our framework.

We present here several natural ways to modify the driver $f$ to avoid explosion of the scheme and obtain its convergence. We organize them in three categories, based on how the driver is modified.

%**********************************************************************************************************************************************************************************************************
\subsubsection*{Multiplicative taming}

Consider a radius $r(h) = r_0 h^{-\alpha}$, with $r_0 > 0$ and $\alpha > 0$. So $r(h) \goesto +\infty$ as $h \goesto 0$. 
We can tame the high values of $f$ by multiplicating it by a damping factor,
	\begin{align*}
		f^h(t,y,z) = \chi^h(y) f(t,y,z) , \quad \text{where} \quad \chi^h(y) = \frac{1}{1+F(y) r(h)^{-1}}.
	\end{align*}
Several choices are possible for the function $F$. We only consider the four following ones.
	\begin{description}
		\item[(a)] $F(y) = \abs{f(0,y,0)}$.
		\item[(b)] $F(y) = \frac{ \abs{f(0,y,0) - f(0,0,0)} } {\abs{y}} \ 1_{\{y \ne 0\}}$.
		\item[(c)] $F(y) = \abs{y}^m$.
		\item[(d)] $F(y) = \abs{y}^{m-1}$.
	\end{description}
The choices (a) and (b) use only the outputs of $f$ and require no detailed knowledge of $f$ (black-box taming), while the choices (c) and (d) use the input $y$ and require knowing the degree $m$ of the polynomial growth. 
Also, as can be seen already, the choices (a) and (c) result in a bounded driver (in the variable $y$, for fixed $h$) while (b) and (d) result in a driver with linear growth in $y$.

%**********************************************************************************************************************************************************************************************************
\subsubsection*{Outer taming}

Consider a radius $r(h) = r_0 h^{-\beta}$, with $r_0 > 0$ and $\beta > 0$.
The outer taming is given by 
	\begin{align*}
		f^h(t,y,z) = T^h \big( f(t,y,z) \big) , 
	\end{align*}
where $T^h$ is essentially a projection on the ball of $\R^n$ of center $0$ and radius $r(h)$. Specifically, we can consider at least the following two choices.
	\begin{itemize}
		\item[(A)] The projection (or truncation) : $T^h(f) = \frac{f}{ \max( 1,\abs{f}r(h)^{-1} ) } = \frac{r(h) f}{ \max( r(h),\abs{f} ) }$.
		\item[(B)] A smoothed projection : $T^h(f) = \frac{f}{ 1 + \abs{f}r(h)^{-1}  } = \frac{r(h) f}{ r(h) + \abs{f} }$.
	\end{itemize}
One could also consider a projection on the ball of radius $r(h)+1$ that is smoothed only at the transition region from identity to constant, so that it would remain the identity on the ball of radius $r(h)$ and be constant in any $y$-direction beyond $r(h)+1$.

Notice some general properties of $T^h$ : $\abs{T^h(f)} \le \abs{f}$ and $\abs{T^h(f)} \le r(h)$, for all $f \in \R^n$. 
The projection also satisfies $T^h(f) = f$ when $\abs{f} \le r(h)$.

\paragraph*{}
Notice that for both the case of the standard projection (A) and the case of the particular smoothed projection (B), the taming can be written multiplicatively, $f^h(t,y,z) = \chi^h(t,y,z) f(t,y,z)$. 
Indeed, we have 
	\begin{align*}
		T^h(f) = \frac{1}{\max(1, \abs{f}r(h)^{-1})}  \ f
			\qquad \text{and} \qquad 
		T^h(f) = \frac{1}{1 + \abs{f}r(h)^{-1} }  \ f
	\end{align*}
in cases (A) and (B), respectively.

Case (A), the standard projection, can therefore be viewed as a generalization-variation of the multiplicatively tamed driver, case (a), the generalization being that we consider a damping factor $\chi^h(t,y,z)$ instead of just $\chi^h(y)$, the variation being that we have to deal with $\max(1,x)$ instead of $1+x$.
For case (B), it is only a generalization of (a). 
If we consider a driver depending only on $y$, as we do in all our examples, we see that the outer taming (B) was already treated as the multiplicative taming (a).  So let us ignore this case and focus only on the standard projection (A), in this section.
From now on, for the outer taming, $T^h$ is the standard projection on the ball of radius $r=r(h)$.

%**********************************************************************************************************************************************************************************************************
\subsubsection*{Inner taming}

Another way to avoid values of $f$ that are too high is to limit the size of the inputs entered in $f$.
Consider a radius $r(h) = r_0 h^{-\gamma}$, with $r_0 > 0$ and $\gamma > 0$.
The inner taming is given by 
	\begin{align*}
		f^h(t,y,z) = f \big(t,T^h(y),z \big) , 
	\end{align*}
where $T^h$ is the projection on the ball of $\R^n$ of center $0$ and radius $r(h)$. Here again, as for the outer taming, one could consider a number of variations for the ``projection'' $T^h$. We will only study the example of the standard inner projection.
Recall the basic properties of $T$ : $\abs{T(y)} \le r$ and $\abs{T(y)} \le \abs{y}$ for all $y$, and $T(y) = y$ if $\abs{y} \le r$.

%/////////////////////////////////////////////////////////////////////////////////////////////////////////////////////////////////////////////////////////////////////////////////////////
%/////////////////////////////////////////////////////////////////////////////////////////////////////////////////////////////////////////////////////////////////////////////////////////
\subsection{Verifications for the inner taming by projection}

We present these verifications in full for the inner taming. Notice that such an approach, taming the driver from inside, was used in \cite{ChassagneuxJacquierMihaylov2014} for scalar SDEs. It was verified there that the inner projection yields a suitable driver, in dimension 1. Our verifications here work in any dimension $n \in \N^*$.

%*********************************************************************************************************************************************************************************************************
\subsubsection{Verification of \htgrowth}

Using \hfgrowth,
	\begin{align*}
		\abs{f^h(t,y,z)} \le K_t + K_y \abs{T^h(y)}^m + K_z \abs{z} 
			\le K_t + K_y r(h)^{m-1} \abs{y} + K_z \abs{z}  .
	\end{align*}
So we set $K^h_t = K_t$, $K^h_y = K_y r(h)^{m-1}$ and $K^h_z = K_z$. We have $(K^h_y)^2h$ bounded iff $\beta \le \frac{1}{2(m-1)}$.

%*********************************************************************************************************************************************************************************************************
\subsubsection{Verification of \htmongrowth}

We write
	\begin{align*}
		\scalar{y}{f^h(t,y,z)} 
			&=   \scalar{y}{f(t,T^h(y),z)}																																											\\
			&=   \scalar{y}{f(t,T^h(y),z)-f(t,0,z)} + \scalar{y}{f(t,0,z)}																																	\\
			&\le \scalar{y-0}{f(t,T^h(y),z)-f(t,0,z)} + \alpha \abs{y}^2 + \frac{K_t^2}{2\alpha} + \frac{K_z^2}{2\alpha} \abs{z}^2		,								\\
	\end{align*}
by the standard manipulations using \hfgrowth. We now want handle the main term. 
If $\abs{y} \le r(h)$, then by \hfmon %$\scalar{y}{f(t,T^h(y),z)} \le \bar{M}_t + \bar{M}_y \abs{y}^2 + \abs{M}_z \abs{z}^2$.	
	\begin{align*}
		\scalar{y-0}{f(t,T^h(y),z)-f(t,0,z)} = \scalar{y-0}{f(t,y,z)-f(t,0,z)} \le M_y \abs{y}^2 \le \max(0,M_y) \abs{y}^2 .
	\end{align*}
If 	$\abs{y} > r(h)$, then we note that $T^h(y) = r(h) \frac{y}{\abs{y}}$, or equivalently $y =  \frac{\abs{y}}{r(h)} T^h(y)$. Hence, from \hfmon we have
	\begin{align*}
		\scalar{y-0}{f(t,T^h(y),z)-f(t,0,z)} 
			&= \frac{\abs{y}}{r(h)} \scalar{T^h(y)-0}{f(t,T^h(y),z)-f(t,0,z)} 									
			\\
			&\le \frac{\abs{y}}{r(h)} M_y \abs{T^h(y)}^2															
			%\\
			%&
			= \frac{\abs{y}}{r(h)} M_y r(h)^2 = M_y r(h) \abs{y}												
			\\
			&\le \max(0,M_y) \abs{y}^2 .
	\end{align*}
since $r < \abs{y}$.
So we have the same upper bound in both cases, and we can conclude that
	\begin{align*}
		\scalar{y}{f^h(t,y,z)} 
			&\le (\max(0,M_y)+\alpha) \abs{y}^2 + \frac{K_t^2}{2\alpha} + \frac{K_z^2}{2\alpha} \abs{z}^2 .
	\end{align*}
Therefore, taking $\bar{M}^h_t = \frac{K_t^2}{2\alpha}$, $\bar{M}^h_y = \max(0,M_y)+\alpha$ and $\bar{M}_z = \frac{K_z^2}{2\alpha}$ suits.

%*********************************************************************************************************************************************************************************************************
\subsubsection{Verification of \htreg}

We see immediately that, since $f^h$ only alters the argument $y$, \hfreg is unchanged :
	\begin{align*}
		\abs{ f^h(t',y,z') - f^h(t,y,z) } = \abs{ f(t',T^h(y),z') - f(t,T^h(y),z) } \le L_t \abs{t'-t}^\half + L_z \abs{z'-z}.
	\end{align*}

%*********************************************************************************************************************************************************************************************************
\subsubsection{Verification of \htregY}

Using \hfregY and the $2$-Lipschitzness of $T^h$, we have 
	\begin{align*}
		\abs{ f^h(t,y',z) - f^h(t,y,z) } 
			&=   \abs{ f(t,T^h(y'),z) - f(t,T^h(y),z) } 																					\\
			&\le L_y \big( 1 + \abs{T^h(y')}^{m-1} + \abs{T^h(y)}^{m-1} \big) \abs{T^h(y')-T^h(y)}			\\
			&\le 2 L_y \big( 1 + r(h)^{m-1} + r(h)^{m-1} \big) \abs{y'-y}													\\
			&= 2 L_y \big( 1 + 2 r(h)^{m-1} \big) \abs{y'-y} ,
	\end{align*}
so we set $L^h_y = 2 L_y \big( 1 + 2 r(h)^{m-1} \big)$. As usual, $(L^h_y)^2 h$ is bounded if $\beta \le \frac{1}{2(m-1)}$. We conclude that $y\mapsto f^h(\cdot,y,\cdot)$ is $L^h_y$-Lipschitz.

%*********************************************************************************************************************************************************************************************************
\subsubsection{Verification of \htmon}

Sadly, it does not seem to be true that $f^h$ satisfies \hfmon i.e. has $\Rmon = 0$.
We can however do the standard estimation via $R^h$, through   
	%\begin{align*}
		$f^h(t,y,z) = f(t,y,z) - R^h(t,y,z)$.
	%\end{align*}
Then, using \hfmon for $f$,
	\begin{align*}
		 \scalar{ y' - y }{ f^h(t,y',z) - f^h(t,y,z) } 	
			&=  \scalar{ y' - y } { f(t,y',z) - f(t,y,z) } 
			\\
			& \qquad  - \scalar{ y' - y } { R^h(t,y',z) -  R^h(t,y,z) }
			\\
			&\le M_y \Abs{y'-y}^2 + \Rmon(t,y',y,z),
	\end{align*}
where 
	%\begin{align*}
		$\Rmon(t,y',y,z)  = - \scalar{ y' - y } { R^h(t,y',z) -  R^h(t,y,z) } $.
	%\end{align*}
Using the estimate for $R^h$ found in the next subsubsection,
	\begin{align*}
		\abs{ \Rmon(t,y',y,z) }  
			&
			\le \abs{ y' - y } \abs{ R^h(t,y',z) -  R^h(t,y,z) } 																	
			\\
			&\le \big( \abs{y'} + \abs{y} \big) \big( \abs{R^h(t,y',z)} +  \abs{R^h(t,y,z)} \big) 					
			\\
			&\le C \big( \abs{y'} + \abs{y} \big) \Big( \big( 1 + \abs{y'}^m \big) \ \1_{\{\abs{y'}>r(h)\}} +  \big( 1 + \abs{y}^m \big) \ \1_{\{\abs{y}>r(h)\}} \Big) 					
			\\
			%&
			%\le C \big( \abs{y'} + \abs{y} \big) \Big( 1 + \abs{y'}^m + \abs{y}^m \Big) \ \1_{\{\abs{y'}>r(h) \text{ or } \abs{y}>r(h)\}}   											\\
			&
			\le C  \Big( 1 + \abs{y'}^{2m} + \abs{y}^{2m} \Big) \ \1_{\{\abs{y'}>r(h) \text{ or } \abs{y}>r(h)\}} .
	\end{align*}

%*********************************************************************************************************************************************************************************************************
\subsubsection{Verification of \htfcvg}

Using \hfregY and the fact that $\abs{y-T^h(y)}\le\abs{y}$, we have
	\begin{align*}
		\abs{R^h(t,y,z)} 
			&=   \abs{ f(t,y,z) - f^h(t,y,z) }																															\\
			&=   \abs{ f(t,y,z) - f(t,T^h(y),z) } \ 1_{\{\abs{y}>r(h)\}}																						\\
			&\le L_y \big( 1 + \abs{y}^{m-1} + \abs{T^h(y)}^{m-1} \big) \abs{y-T^h(y)}	\ 1_{\{\abs{y}>r(h)\}}					\\
			&\le L_y \big( 1 + \abs{y}^{m-1} + \abs{y}^{m-1} \big) \abs{y}	 \ 1_{\{\abs{y}>r(h)\}}										\\
			&\le C \big( 1 + \abs{y}^m \big) \ 1_{\{\abs{y}>r(h)\}} .
	\end{align*}

%/////////////////////////////////////////////////////////////////////////////////////////////////////////////////////////////////////////////////////////////////////////////////////////
%/////////////////////////////////////////////////////////////////////////////////////////////////////////////////////////////////////////////////////////////////////////////////////////
\subsection{Some verifications for the multiplicative tamings}

We provide here some of the verifications that the multiplicative tamings (a) to (d) proposed above satisfy the assumptions of Section \ref{sec:fh}. 
Full verifications can be found in side notes posted on the webpage of the first-listed author.

%*********************************************************************************************************************************************************************************************************
\subsubsection{Verification of \htmongrowth}

We use the fact that $f$ satisfies \hfmongrowth, as well as $\chi^h(y) \in [0,1]$, to write
	\begin{align*}
		\scalar{y}{f^h(t,y,z)} 
			&=  \chi^h(y) \scalar{y}{f(t,y,z)}
			\\
			&\le \chi^h(y) \Big( \bar{M}_t + \bar{M}_y \abs{y}^2 + \bar{M}_z \abs{z}^2 \Big)
			%\\
			%&\le \bar{M}_t + \chi^h(y) \bar{M}_y \abs{y}^2 + \bar{M}_z \abs{z}^2
			%\\
			%&
			\le \bar{M}^h_t + \bar{M}^h_y \abs{y}^2 + \bar{M}^h_z \abs{z}^2
	\end{align*}
where $\bar{M}^h_t = \bar{M}_t$, $\bar{M}^h_z = \bar{M}_z$ and $\bar{M}^h_y = \max(0,\bar{M}_y)$.
These constants do not depend on $h$, so they work for \htmongrowth (they are bounded as $h \goesto 0$).

%*********************************************************************************************************************************************************************************************************
\subsubsection{Verification of \htregY}

We write
	\begin{align*}
		f^{h}(t,y',z) - f^{h}(t,y,z)  
			&= \chi^h(y') f(t,y',z) - \chi^h(y)  f(t,y,z) 													\\
			&= \chi^h(y') \chi^h(y) \big(f(t,y',z) - f(t,y,z)\big) +  \RregY(t,y,y',z),
	\end{align*}
where 
	\begin{align*}
		\RregY(t,y,y',z) := \chi^h(y') \big(1-\chi^h(y)\big) f(t,y',z)  - \big(1-\chi^h(y')\big) \chi^h(y) f(t,y,z). 
	\end{align*}  

We first estimate the ``good term'', that will give the Lipschitz-in-$Y$ regularity for fixed $h$, and then will estimate the remainder term.
Using \htregY , we have
	\begin{align*}
		\abs{ \chi^h(y') \chi^h(y) \big(f(t,y',z) - f(t,y,z)\big) }
			&\le \chi^h(y') \chi^h(y) L_y \big( 1 + \abs{y'}^{m-1} + \abs{y}^{m-1} \big) \abs{y'-y}			\\
			&\le  L_y \big( 1 + \chi^h(y') \abs{y'}^{m-1} + \chi^h(y) \abs{y}^{m-1} \big) \abs{y'-y}	.
	\end{align*}

To estimate the terms $\chi^h(y) \abs{y}^{m-1}$ we need to distinguish cases. 
We do it only for cases (c) and (d).

%__________________________________________________________________________________________________________________________________________________________________________________________________
\paragraph{Case (c).} 

	\begin{align*}
		\chi^h(y) \abs{y}^{m-1} = \frac{\abs{y}^{m-1}}{ 1 + \abs{y}^{m-1} r(h)^{-1} } \le r(h) .
	\end{align*}
So we take $L^h_y := L_y (1+2r(h))$. And $(L^h_y)^2 h$ is bounded iff $\alpha \le \half$.

%__________________________________________________________________________________________________________________________________________________________________________________________________
\paragraph{Case (d).} 

%Before estimating, note that we always consider $h \le T$ ($N \ge 1$) and therefore $r(h) \ge r_0 T^{-\alpha}$, which we assume to be $\ge 1$ for convenience.
%Now,
	\begin{align*}
		\chi^h(y) \abs{y}^{m-1} = \frac{\abs{y}^{m-1}}{ 1 + \abs{y}^{m} r(h)^{-1} } 
			\le 	\left\{\begin{aligned}
					 	r(h) \quad &\text{if } \abs{y} \ge 1 \\
					 	1 \quad &\text{if } \abs{y} \le 1
					\end{aligned}\right.
			\le 1 + r(h) .
	\end{align*}
So we take $L^h_y := L_y (3+2r(h))$. And $(L^h_y)^2 h$ is bounded iff $\alpha \le \half$.

\paragraph*{}
In conclusion, the ``good term'' is indeed bounded from above by $L^h_y \abs{y'-y}$ and $(L^h_y)^2 h$ remain bounded as $h \goesto 0$.
It remains to estimate the remainder term.

\paragraph*{}
We recall before starting that $\chi^h(y) \in [0,1]$ and that
	\begin{align*}
		1-\chi^h(y) = \frac{ F(y) r(h)^{-1} } { 1 + F(y) r(h)^{-1} } = \chi^h(y) F(y) r(h)^{-1} \le F(y) r(h)^{-1}.
	\end{align*}
Then, we estimate
	\begin{align*}
		\abs{\RregY(t,y,y',z)} 
			&= \Abs{ \chi^h(y') \big(1-\chi^h(y)\big) f(t,y',z)  - \big(1-\chi^h(y')\big) \chi^h(y) f(t,y,z) }											\\
			&\le \chi^h(y') \big(1-\chi^h(y)\big) \Abs{f(t,y',z)} + \big(1-\chi^h(y')\big) \chi^h(y) \Abs{f(t,y,z)}									\\
			&\le  \Big\{ 1 \times F(y) \times \Abs{f(t,y',z)} + F(y') \times 1 \times \Abs{f(t,y,z)}	\Big\} r(h)^{-1} .
	\end{align*}
Now, we use \hfgrowth and \hfregY to claim that $F(y) \le C (1+\abs{y}^m)$. For this we separate again the cases.
	\begin{description}
		\item[Case (c)] $F(y) = \abs{y}^m \le C (1+\abs{y}^m)$.
		\item[Case (d)] $F(y) = \abs{y}^{m-1} \le C (1+\abs{y}^m)$.
	\end{description}
In the above, we have used $\abs{y}^p \le 1+ \abs{y}^q$ for $q \ge p$.
Using again \hfgrowth we have
	\begin{align*}
		\abs{\RregY(t,y,y',z)} 
			&\le C \Big\{ \big( 1+\abs{y}^m \big)  \big( 1+\abs{y'}^m + \abs{z} \big) + \big( 1+\abs{y'}^m \big) \big( 1+\abs{y}^m + \abs{z} \big)	\Big\} r(h)^{-1} 			\\
			&\le C \Big\{ 1 + \abs{y}^{2m} + \abs{y'}^{2m} + \abs{z}^2 \Big\} h^\alpha .
	\end{align*}
To obtain the last estimate we have used several times the inequality $ab \le a^2 + b^2$.
Therefore, \htregY is verified.

%/////////////////////////////////////////////////////////////////////////////////////////////////////////////////////////////////////////////////////////////////////////////////////////
%/////////////////////////////////////////////////////////////////////////////////////////////////////////////////////////////////////////////////////////////////////////////////////////
\subsection{Some verifications for the outer taming by projection}

We provide here some of the verifications that the outer taming (A) proposed above satisfy the assumptions of subsubsection \ref{sec:fh}. The aim is mainly to illustrate a different case in assumption \htfcvg.
Full verifications can be found in side notes posted on the webpage of the first-listed author.

%*********************************************************************************************************************************************************************************************************
\subsubsection{Verification of \htfcvg}

Using the fact that $T^h(f) = f$ for $\abs{f} \le r(h)$, we see that
	\begin{align*}
		R^h(t,y,z) = \big[ f(t,y,z) - T^h\big( f(t,y,z) \big) \big] 1_{\{\abs{f(t,y,z)}>r(h)\}} .
	\end{align*}
Here we can save a factor 2 in the constants by using the fact $\abs{f-T^h(f)} \le \abs{f}$ and so with \hfgrowth we have
	\begin{align*}
		\Abs{R^h(t,y,z)}
			&\le \Abs{f(t,y,z) - T^h\big( f(t,y,z) \big) } \ 1_{\{\abs{f(t,y,z)}>r(h)\}} 				\\
			&\le \abs{f(t,y,z)} \ 1_{\{\abs{f(t,y,z)}>r(h)\}} 																			\\
			&\le C \big( 1 + \abs{y}^m + \abs{z} \big) \ 1_{\{\abs{f(t,y,z)}>r(h)\}} .
	\end{align*}

Thus $f^h$ satisfies \htfcvg with the criterion 2.

%/////////////////////////////////////////////////////////////////////////////////////////////////////////////////////////////////////////////////////////////////////////////////////////
%/////////////////////////////////////////////////////////////////////////////////////////////////////////////////////////////////////////////////////////////////////////////////////////
\subsection{Verifications for the standard truncation of the Brownian increments}
%\label{sec:VerifyBrownianIncrements}

Here we prove, for a particular choice of increments $H_\cdot$ that the following inequality, (\textbf{HH}).3, is satisfied : there exists $C > 0$ such that, for all $h>0$,
	\begin{align*}
		\max_{i=0, \ldots, N-1} \bE\Bigg[ \Abs{\frac{\Delta W_\tip}{h} - H_\ip }^2 \Bigg]	
		\le C.
	\end{align*}

We now show that this is satisfied when $H_\ip$ is $\frac{T_{R(h)}(\Delta W_\tip)}{h}$. Here $T_R$ is the projection (for the distance induced by the infinity-norm $\abs{\cdot}_\infty$ on $\R^d$) on the centered ball of radius $R$ (for the norm $\abs{\cdot}_\infty$). Otherwise said, each coordinate of $\Delta W_\tip$ is capped at $R$. Note already that the norm $\abs{\,\cdot\,}$ appearing in the expectation (and indeed in all the computations that were done so far) is the Euclidian norm $\abs{\,\cdot\,}_2$.

For convenience in the following estimations, we rewrite the threshold $R(h)$ as $\sqrt{h} \, r(h)$. Intuitively, we want $R(h) \to +\infty$ as $h \to 0$ (so $r(h) \goesto +\infty$ faster than $1/\sqrt{h}$ in that case).

For $i \in \{0, \ldots, N-1 \}$, we write $\Delta W_\tip = \sqrt{h} \, G$ with $G \sim \mathcal{N}(0,1)$. We then have (with equalities in the sense of laws/distributions)
	\begin{align*}
		\delta :=& \frac{\Delta W_\tip}{h} - H_\ip =  \frac{\Delta W_\tip}{h} - \frac{T_{R(h)}\big( \Delta W_\tip \big)}{h}																\\
			=& \frac{1}{h} \left( \sqrt{h} G - T_{R(h)}\big( \sqrt{h} G \big) \right) = \frac{1}{h} \left( \sqrt{h} G - \sqrt{h} T_{R(h)/\sqrt{h}}\big(  G \big) \right)			\\
			=& \frac{1}{\sqrt{h}} \Big( G - T_{r(h)}\big(G\big) \Big) = \frac{1}{\sqrt{h}} \Big( G - T_{r(h)}\big(G\big) \Big) 1_{\{ \abs{G}_\infty > r(h) \}} \ .
	\end{align*}
Hence
	\begin{align*}
		\abs{\delta} 
		&
		\le \frac{1}{\sqrt{h}} \Abs{ G - T_{r(h)}\big(G\big) } \1_{\{ \abs{G}_\infty > r(h) \}}
		\le \frac{1}{\sqrt{h}} \Abs{G} \1_{\{ \abs{G}_\infty > r(h) \}}	
		\le \frac{1}{\sqrt{h}} \Abs{G} \1_{\{ \abs{G} > r(h) \}}.
	\end{align*}
Here we have used the fact that $\Abs{ G - T_{r(h)}\big(G\big) } \le \Abs{G}$ (one could have used a domination by $2\Abs{G}$ instead of this one). And also the fact that $\abs{G} \ge \abs{G}_\infty$.

Consequently, for $p \ge 1$, and $r(h) \ge 1$,
	\begin{align*}
		\bE\bigg[ \Abs{\frac{\Delta W_\tip}{h} - H_\ip}^p \bigg] 
		&
		\le \frac{1}{h^{p/2}} \bE\bigg[ \Abs{G}^p 1_{\{\abs{G}>r(h)\}} \bigg] 
		= \frac{1}{h^{p/2}} \int_{\R^d} \abs{x}^p 1_{\{\abs{x}>r(h)\}} 
		                                \frac{e^{-\frac{\abs{x}^2}{2}}}{(2\pi)^{d/2}} dx
		\\
		&
		= \frac{1}{h^{p/2}} \int_{\abs{x}>r(h)} \abs{x}^p \frac{e^{-\frac{\abs{x}^2}{2}}}{(2\pi)^{d/2}} dx	
		\\
		&= \frac{1}{h^{p/2}} \int_{r(h)}^{+\infty} \rho^p \frac{e^{-\frac{\rho^2}{2}}}{(2\pi)^{d/2}} \text{Surf}\Big(S^{d-1}(0,\rho)\Big) d\rho
		\\	
		&= \frac{c_d}{h^{p/2}} \int_{r(h)}^{+\infty} \rho^{p+d-1} e^{-\frac{\rho^2}{2}} d\rho
		\\	
		&\le  \frac{c_d \, C_{p+d-1}}{h^{p/2}} r(h)^{p+d-2} e^{-\frac{r(h)^2}{2}}\ ,
	\end{align*}
where we have used that the surface of the $d-1$ dimensional sphere of radius $\rho$ is $\tilde{c}_d \, \rho^{d-1}$, absorbed the $\frac{1}{(2\pi)^{d/2}}$ into $c_d$, and used that for $r \ge 1$, 
	\begin{align*}
		\int_r^{+\infty} x^q e^{-x^2/2} \ud x \le C_q \ r^{q-1} \ e^{-r^2/2} \ .
	\end{align*}

Now, in order to have \hH, it is sufficient to have (with the usual convention regarding constants)
	\begin{align*}
		\frac{1}{h^{p/2}} r(h)^{p+d-2} e^{-\frac{r(h)^2}{2}} \le C h^\alpha \ ,
	\end{align*}
for some $\alpha$ to be determined.
Taking the $\log$ and using the fact $\log r(h) \le \frac14 r(h)^2$ provided $r(h)$ is greater than some $r_0 \ge 1$ (which depends on $p$ and $d$), we see that
	\begin{align*}
		(p+d-2)\log\big(r(h)\big) - \frac{r(h)^2}{2} 
		&\le - \frac{r(h)^2}{4} 
		\le \log(C) + \big(\alpha+ \frac{p}{2}\big)\log(h)
	\end{align*}	
so it is sufficient to have
	\begin{align*}
		r(h)  \ge \sqrt{ 4 \log\Big(\frac{1}{C}\Big) + 4 \Big(\alpha+\frac{p}{2}\Big) \log\Big(\frac1h\Big) }  .
	\end{align*}
We can re-express this as a condition on $R(h)$:
	\begin{align*}
		R(h) = \sqrt{h}r(h) \ge \sqrt{ 4 \, h \, \log\Big(\frac{1}{C}\Big) + 4 \Big(\alpha+\frac{p}{2}\Big) \, h \log\Big(\frac1h\Big) } .
	\end{align*}
We need, in \hH, the case $p=2$, and it is satisfied if we take $\alpha = 0$.

%%%%%%%%%%%%%%%%%%%%%%%%%%%%%%%%%%%%%%%%%%%%%%%%%%%%%%%%%%%%%%%%%%%%%%%%%%
%%%% \BEGIN BIBLIOGRAPHY
%%%%%%%%%%%%%%%%%%%%%%%%%%%%%%%%%%%%%%%%%%%%%%%%%%%%%%%%%%%%%%%%%%%%%%%%%%
%\bibliographystyle{alpha}%wmaainf}%alpha}%plain}
%\bibliographystyle{abbrvnat} % Choose Phys. Rev. style for bibliography
%\bibliographystyle{abbrv}
\bibliography{BSDEpolyYtamed}  % ``name``.bib is the name of thedatabase
% \bib, bibdiv, biblist are defined by the amsrefs package.
%%%%%%%%%%%%%%%%%%%%%%%%%%%%%%%%%%%%%%%%%%%%%%%%%%%%%%%%%%%%%%%%%%%%%%%%%%
%%%% \END BIBLIOGRAPHY
%%%%%%%%%%%%%%%%%%%%%%%%%%%%%%%%%%%%%%%%%%%%%%%%%%%%%%%%%%%%%%%%%%%%%%%%%%
\end{document}